\documentclass[12pt]{amsart}

\usepackage[english]{babel}
\usepackage[alphabetic]{amsrefs}
\usepackage{graphicx}
\graphicspath{ {images/} }
\usepackage[T1]{fontenc}
\usepackage{txfonts}
\usepackage{times}
\usepackage{amssymb,verbatim,amscd,amsfonts,amsmath,amsthm,enumerate}
\usepackage[all]{xy}
\usepackage{subfig}
\usepackage{tikz}
\usetikzlibrary{shapes,arrows,shadows}
\usetikzlibrary{decorations.markings}
\usepackage{color}
\usepackage[normalem]{ulem}
\usepackage{hyperref}
\usepackage{wrapfig}
\usetikzlibrary{arrows}
\usepackage{pb-diagram}

\newcommand{\C}{\ensuremath{\mathbb{C}}}
\newcommand{\N}{\ensuremath{\mathbb{N}}}

\def\0{{\bf 0}}

\def\N{\mathbb{N}}

\def\C{\mathbb{C}}

\newtheorem{theorem}{Theorem}[section]
\newtheorem{proposition}[theorem]{Proposition}
\newtheorem{corollary}[theorem]{Corollary}
\newtheorem{lemma}[theorem]{Lemma}

\newtheorem{remark}[theorem]{Remark}

\newtheorem{definition}[theorem]{Definition}
\newtheorem{question}[theorem]{Question}

\begin{document}

\title{On Infinitely generated Fuchsian groups of some infinite genus surfaces}
\author{John A. Arredondo and Camilo Ram\'irez Maluendas}
\maketitle

\begin{abstract}
In this paper, for a non compact and orientable surface $S$ been either: the Infinite Loch Ness monster, the Cantor tree and the Blooming Cantor tree, we construct explicitly an infinitely generated Fuchsian group $\Gamma<PSL(2,\mathbb{R})$, such that the quotient  $\mathbb{H}/\Gamma$ is a hyperbolic surface homeomorphic to  $S$. 

\end{abstract}



\tableofcontents


\section{Introduction}



A \emph{classical} problem during the 19th century, in which several authors were  focused \emph{e.g.}, Felix Klein, Hermann Schwarz, between others, known as the \emph{uniformization problem, \cite{Abi}}, said that: being $S$ a Riemann surface, find all domains $\widetilde{S}\subset \hat{\mathbb{C}}$ and holomorphic  functions $t: \widetilde{S}\to S$ such that  at each point $p\in S$, $t$ is a local uniformizing variable at $p$. Equivalently, from view of the Covering Spaces theory, there is a topological disc $B\subset S$ with center $p$ such that the restriction of $t$  to each component of $t^{-1}(B)$ is a homeomorphism. Moreover, it is enforced the condition that the space$\widetilde{S}$ is a covering space of $S$ with holomorphic projection map $t: \widetilde{S}\to S$. However, the twenty-second problem of the Mathematical Problems published by David Hilbert \cite{Hilbert}, proposes a major challenge in the uniformization problem. It was to find one uniformization being $\widetilde{S}$ simply connected. The answer to this problem is known as the Uniformization Theorem, which says:

\begin{theorem}
\label{T:uni}
 \cite[p. 174]{Bear1} Let $S$ be a Riemann surface, let $\widetilde{S}$ be the universal covering surface of $S$ chosen from the surfaces $\hat{\mathbb{C}}$, $\mathbb{C}$, and $\Delta$. Let $\Gamma$ be the  cover group of $S$. Then
\begin{enumerate}
  \item $S$ is conformally equivalent to $\widetilde{S}/ \Gamma$;
  \item $\Gamma$  is a M\"{o}bius group which acts discontinuously on $\widetilde{S}$;
  \item Apart from the identity, the elements of $\Gamma$ have no fixed points in $\widetilde{S}$;
  \item The cover group $\Gamma$ is isomorphic to $\pi (S)$.
\end{enumerate}
\end{theorem}


Given that, the complex plane $\mathbb{C}$ uniformize itself, the cylinder and the torus \cite[p. 193]{Far}, it means, if the (holomorphic) universal covering surface of the Riemann surface $S$ is the complex plane $\mathbb{C}$, then $S$ is conformally equivalent to $\mathbb{C}$, $\mathbb{C}-\{0\}$ or a torus. When $\Gamma$ has the identity as unique element, then the Riemann surface $\mathbb{C}/\Gamma$ is $\mathbb{C}$. If is considered $\Gamma$ generated by the M\"{o}bius transformation $z\mapsto z+1$, then the quotient space $\mathbb{C}/\Gamma$ is conformally equivalent to $\mathbb{C}-\{0\}$  (the cylinder). Finally, considering the subgroup $\Gamma$ generated by the M\"{o}bius transformations $z\mapsto z+1$ and $z\mapsto z+\tau$, where $\tau \in \mathbb{C}$ and $Im (\tau)>0$, then the quotient space $\mathbb{C}/\Gamma$ is a torus. On the other hand, the only Riemann surface $S$ which has as universal covering the sphere, is the sphere itself \cite[IV. 6.3. Theorem]{Far}. Hence, it is natural to ask:

\begin{question}
Given any non-compact Riemann surface $S$. Which is the subgroup $\Gamma$ of the isometries of the hyperbolic plane $\mathbb{H}$, such that the quotient space $\mathbb{H}/ \Gamma$ is a Riemann surface homeomorphic to $S$?

\end{question}


The present work answers to this question in the case $S$ being either: the \emph{Infinite Loch Ness monster}, the \emph{Cantor tree} and the \emph{Blooming Cantor tree}. More precisely we prove that:

\begin{theorem}
\label{T:0.1}

Let  $\Gamma < PSL(2,\mathbb{R})$ be the subgroup generated by the set of M\"{o}bius transformations $\{f_n(z)$, $g_n(z)$, $f^{-1}_n(z)$, $g^{-1}_n(z): n\in \mathbb{Z}\}$, where
\begin{align*}
f_{n}(z)      & := \frac{(8n+4)z-(1+8n(8n+4))}{z-8n}, \\
g_{n}(z)      & := \frac{(8n+6)z +(-1-(8n+2)(8n+6))}{z-(8n+2)},\\
f_{n}^{-1}(z) & := \frac{-8nz+(1+8n(8n+4))}{-z+(8n+4)}, \\
g_{n}^{-1}(z) & := \frac{-(8n+2)z +(1+(8n+2)(8n+6))}{-z+(8n+6)}.
\end{align*}
Then $\Gamma$ is an infinitely generated Fuchsian group and the Riemann surface $\mathbb{H} / \Gamma$ is homeomorphic to the Infinite Loch Ness monster.
\end{theorem}



\noindent The surface with only one end and infinite genus is called the \emph{Infinite Loch Ness monster}. See Figure \ref{loch}.
\begin{figure}[h]
\begin{center}
\includegraphics[scale=0.8]{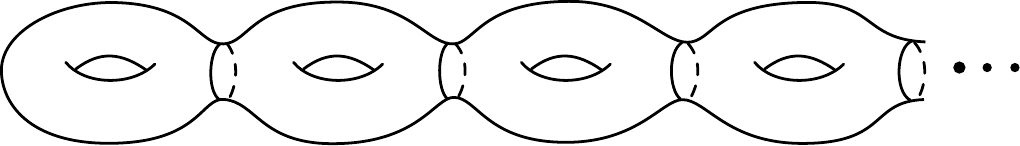}\\
  \caption{\emph{The Infinite Loch Ness monster.}}
   \label{loch}
\end{center}
\end{figure}

\begin{theorem}
\label{T:0.2}
Let $\Gamma <PSL(2,\mathbb{R})$ be the subgroup generated by the union $\cup_{n\in\mathbb{N}}J_n$, where the set
$$
J_n:= \{f_{n,k}(z), \, f^{-1}_{n,k}(z): k\in\{0,\ldots,2^{n-1}-1\}\}
$$
is formed by M\"{o}bius transformations, such as
\[
\begin{array}{rcl}
f_{n,k}(z) & = &\dfrac{-2\cdot(3^n\cdot 2 + 3 + 2s_{2k-1})z+\dfrac{2^2\cdot(3^n\cdot 2 + 3 + 2s_{2k-1})^2 -1 }{3^n \cdot 2^2}}{3^n\cdot 2^2z-2\cdot(3^n\cdot 2 + 3 + 2s_{2k-1})},\\
f^{-1}_{n,k}(z) & =&\dfrac{-2\cdot(3^n\cdot 2 + 3 + 2s_{2k-1})z-\dfrac{2^2\cdot (3^n\cdot 2 + 3 + 2s_{2k-1})^2 -1 }{3^n \cdot 2^2}}{-3^n\cdot 2^2 z-2\cdot(3^n\cdot 2 + 3 + 2s_{2k-1})}.\\
\end{array}
\]
Then $\Gamma$ is an infinitely generated Fuchsian group and the Riemann surface $\mathbb{H} / \Gamma$ is homeomorphic to the Cantor tree.
\end{theorem}

\noindent The surface with ends space the Cantor set and without genus is called the \emph{Cantor tree}. See Figure \ref{cantor}-a.
\begin{figure}[ht]
     \centering
     \begin{tabular}{ccc}
      \includegraphics[scale=0.3, angle=-90]{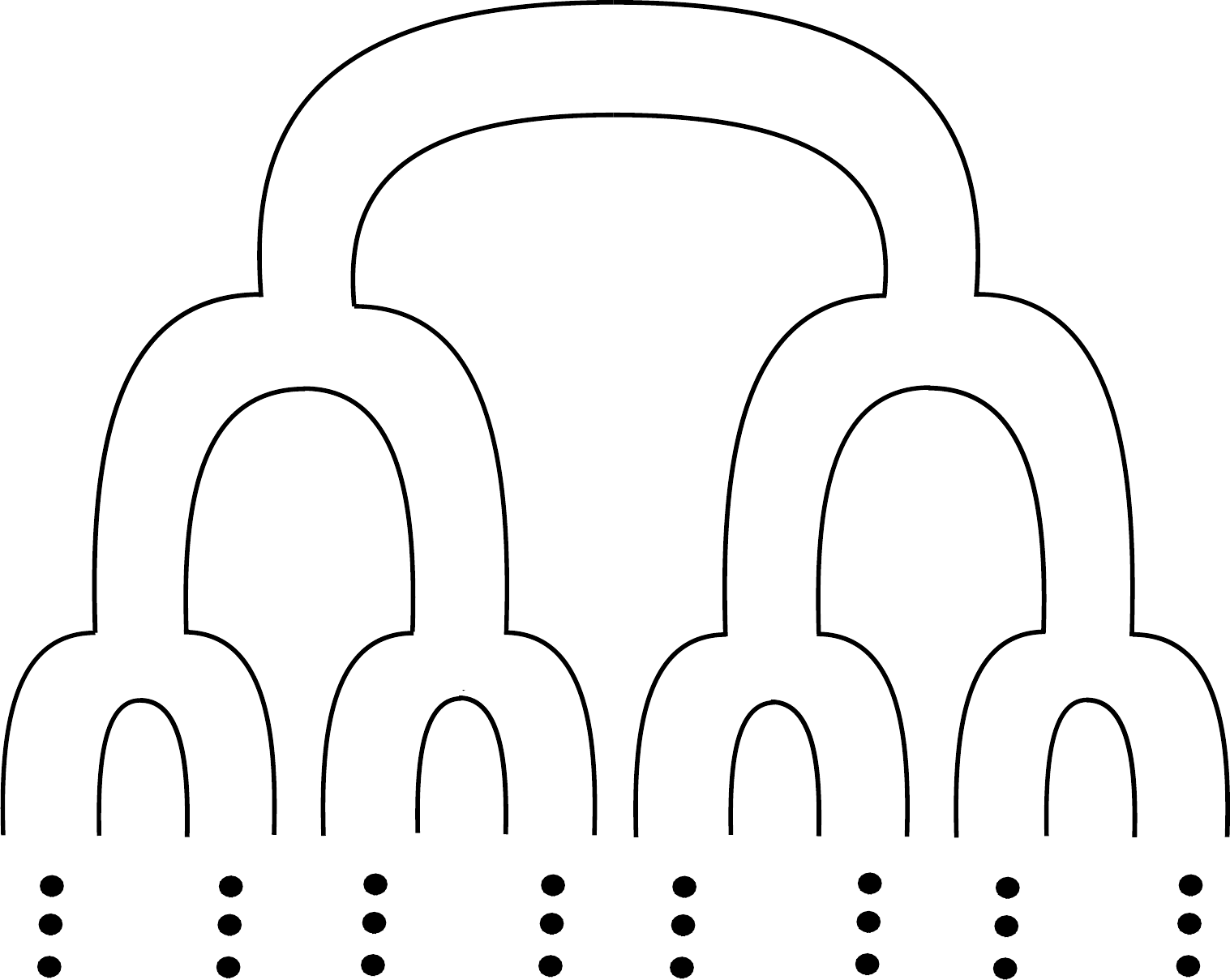} &    &  \includegraphics[scale=0.3, angle=-90]{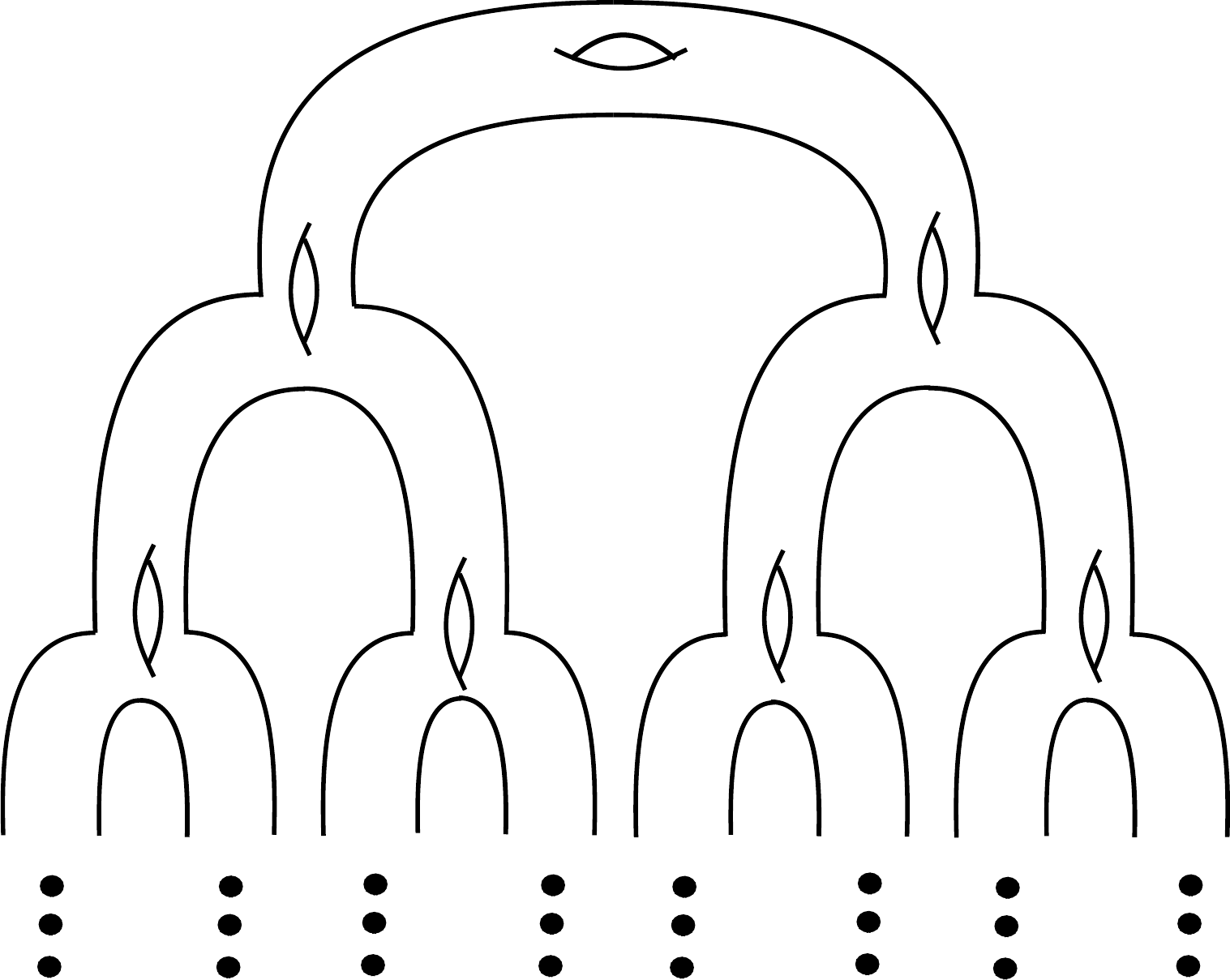} \\
     \text{a. \emph{The Cantor tree.}}                         &                 & \text{b. \emph{The Blooming Cantor tree.}} \\
     \end{tabular}
     \caption{\emph{Surfaces whose ends spaces are homeomorphic to the Cantor set.}}
     \label{cantor}
\end{figure}

\begin{theorem}
\label{T:0.3}

Let $\Gamma< PSL(2,\mathbb{R})$ be the group generated by  the union $\cup_{n\in \mathbb{N}}J_n$, where the set
\[
J_{n}:= \left\{f_{n,k}(z), \, f_{n,k}^{-1}(z), \, \Bigl( f_{n,k}\Bigr)_{s,m}(z), \, \Bigl( f_{n,k}\Bigr)_{s,m}^{-1}(z):\, k\in\{0,\ldots,2^{n-1}-1\},\, s\in \{1,\ldots, 4\}, \, m\in\mathbb{N}\right\},
\]
is composed by M\"{o}bius transformations\footnote{They will given explicitly in section 3.3}. Then $\Gamma$ is an infinitely generated Fuchsian group and the Riemann surface $\mathbb{H}/\Gamma$ is homeomorphic to the Blooming Cantor tree.
\end{theorem}

\noindent The surface with ends space the Cantor set and each ends has infinite genus is called the \emph{Blooming Cantor tree}. See Figure \ref{cantor}-b.

\begin{corollary}
The  infinitely Loch Ness Monster, the Cantor tree and the Blooming Cantor tree are geodesically complete.
\end{corollary}








\vspace{2mm}
The paper is organized as follows: In \textbf{section} \ref{section2} we collect the principal tools used through the paper and \textbf{section} \ref{section3} is dedicated to the proof of our main results, explicitly:

In \textbf{section} \ref{section3.1} we prove Theorem \ref{T:0.1}, defining a group $\Gamma$ from a suitable family of half circles $\mathcal{C}$. It means, the elements of $\mathcal{C}$ will be the half circle with center the integers numbers and radius one, then will define $J$, as a family of M\"{o}bius transformation having as isometric circles the elements of $\mathcal{C}$. Hence, the group $\Gamma$ will be the generated by $J$. Immediately, we shall show that $\Gamma$ is Fuchsian and prove that $\mathbb{H} / \Gamma$ is the desired non compact surface \emph{i.e.}, the Infinite Loche Ness monster.

In \textbf{section} \ref{section3.2} we prove Theorem \ref{T:0.2}, the idea to define the Fuchsian group $\Gamma$ is to make use of the geometrical construction of the Cantor set described in section \ref{section2.1}. By each step $n$ into this geometrical construction, we shall describe suitably in the hyperbolic plane two disjoint half circles $C(f_{n,k})$ and $C(f^{-1}_{n,k})$ by each closed subinterval of $I_n$. These ones will be symmetric with respect to the imaginary axis. Hence, we will define the set $\mathcal{C}_n=\{C(f_{n,k}), C(f^{-1}_{n,k}): k\in\{0,\ldots,2^{n-1}-1\}\}$ composed by half-circle mutually disjoint. Then we will calculate the explicit form of the M\"{o}bius transformation $f_{n,k}(z)$ and $f^{-1}_{n,k}(z)$, which have as isometric circles $C(f_{n,k})$ and $C(f^{-1}_{n,k})$, respectively. Hence, we will define the set $J_n=\{f_{n,k}(z), f^{-1}_{n,k}(z): k\in\{0,\ldots,2^{n-1}-1\}\}$. With this, the desired Fuschian group $\Gamma$ will be generated by $\bigcup\limits_{n\in\mathbb{N}} J_n$. Our choice of half circle drafts a simply connected region $F\subset \mathbb{H}$ having by boundary the set $\bigcup\limits_{n\in\mathbb{N}}C_n$. Then if we identify  this boundary appropriately, we hold the Cantor tree; it means the quotient set $\mathbb{H}/ \Gamma$ will be a hyperbolic surface homeomorphic to the Cantor tree.

Finally in \textbf{section} \ref{section3.3} we prove Theorem \ref{T:0.3}, first building recursively $J:=\{J_n\}_{n\in\mathbb{N}}$ a countable set of finite M\"obius transformation,  which help us in the geometrical construction of the Cantor set. Further, the elements of each family $J_n$ come their respective isometric circles. Once more, we will make use of the geometrical construction of the Cantor set described in section \ref{section2.1} to define a Fuchsian group $\Gamma$. In a similar way, as in the case of the Cantor tree and, by each step $n$ into the preceding geometrical construction we shall describe suitably in the hyperbolic plane the two disjoint half circles $C(f_{n,k})$ and $C(f^{-1}_{n,k})$ by each closed subinterval of $I_n$. These ones will be symmetric with respect to the imaginary axis.  Hence, we will define the set $\mathcal{C}_n=\{C(f_{n,k}), C(f^{-1}_{n,k}): k\in\{0,\ldots,2^{n-1}-1\}\}$ composed by half-circle mutually disjoint. These half-circle will be slightly different to the half-circles defined in the Cantor tree case. Then we will calculate the explicit form of the M\"{o}bius transformations $f_{n,k}(z)$ and $f^{-1}_{n,k}(z)$, which have as isometric circles $C(f_{n,k})$ and $C(f^{-1}_{n,k})$, respectively. Therefore, we will define $F_n=\{f_{n,k}(z),f^{-1}_{n,k}(z): k\in\{0,\ldots, 2^{n-1}-1\}\}$ the set composed by all M\"{o}bius transformations having as isometric circles the elements of $\mathcal{C}_n$. Additionally, for each $k\in \{0,\ldots , 2^{n-1}-1\}$ we will build eight appropriate sequences of half-circles closer to each of the end points of the half circles $C(f_{n,k})$ and $C(f_{n,k}^{-1})$. The radii of this half-circle converge to zero. Consequently, we will define a new countable set $F_{n,k}$ composed by M\"{o}bius transformations depending on the above sequence of half-circles. To introduce this kind of sequence we will induce infinite genus in each of the ends in our desired surface \emph{i.e.}, the blooming Cantor tree. Moreover, we will define $J_n$ as the union $F_n\cup \bigcup\limits_{k\in\{0,\ldots,2^{n-1}-1\}}F_{n,k}$ and the Fuschian group $\Gamma$ will be the group generated by the set $\bigcup\limits_{n\in\mathbb{N}}J_n$. By each step $n$ into the preceding geometrical construction we shall describe suitably in the hyperbolic plane two disjoint half circle $C(f_{n,k})$ and $C(f^{-1}_{n,k})$ by each closed subinterval of $I_n$. These ones will be symmetric with respect to the imaginary axis.  Hence, we will define the set $J_n=\{f_{n,k}(z), f^{-1}_{n,k}(z): k\in\{0,\ldots,2^{n-1}-1\}\}$. Then the desired Fuschian group $\Gamma$ will be the group generated by $\bigcup\limits_{n\in\mathbb{N}} J_n$. Our choice of half circle drafts a simply connected region $F\subset \mathbb{H}$ having by boundary the set $\bigcup\limits_{n\in\mathbb{N}}C_n$. Then if we identify  this boundary appropriately, we hold the Cantor tree; we mean the quotient set $\mathbb{H}/ \Gamma$ will be a hyperbolic surface homeomorphic to the Cantor tree.


\section{PRELIMINARIES}\label{section2}


\subsection{Geometrical construction of the Cantor set}\label{section2.1}


 We recall its geometrical construction by removing the middle third started with the closed interval $I_0:=[1,2]\subset \mathbb{R}$. We let $I_1$ be the closed subset of $I_0$ held from $I_0$ by removing its middle third $\left(1+\dfrac{1}{3},1+\dfrac{2}{3}\right)$ \emph{i.e.},
 \[
 I_{1}=\left[1, 1+\dfrac{1}{3}\right]\cup \left[1+\dfrac{2}{3},2\right].
  \]
The closed subset $I_1\subset I_0$ is the union  of two disjoint closed intervals having length $\dfrac{1}{3}$. We let $I_2$ be the closed subset of $I_0$ held from $I_1$ by removing its middle thirds $\left(1+\dfrac{1}{9},1+\dfrac{2}{9}\right)$ and $\left(1+\dfrac{7}{9},1+\dfrac{8}{9}\right)$ respectively \emph{i.e.},
\[
I_{2}=\left[1,1+\dfrac{1}{9}\right]\cup\left[1+\dfrac{2}{9},1+\dfrac{1}{3}\right]\cup\left[1+\dfrac{2}{3},1+\dfrac{7}{9}\right]\cup\left[1+\dfrac{8}{9},2\right].
\]
The closed subset $I_2\subset I_0$ is the union  of four disjoint closed intervals having length $\dfrac{1}{3^2}$. We now construct inductively  the closed subset $I_n\subset I_0$ from $I_{n-1}$ by removing its middle thirds. We note that each positive integer number $k$ can be written as the binary form
\[
k=t_0\cdot 2^0 + t_1 \cdot 2^1 +\ldots + t_i \cdot 2^i+ \ldots + t_m \cdot 2^m
\]
where $t_m=1$, $t_i\in \{0,1\}$ for all $i\in\{0,\ldots, m\}$ and any $m\in\mathbb{N}$. Thus, we define $s_j$ as following
\[
s_k:=2\cdot t_0 \cdot 3^0 +2\cdot t_1 \cdot 3^1+\ldots+ 2\cdot t_i \cdot 3^i+\ldots 2\cdot t_m\cdot 3^m.
\]
Contrary, if $k=0$ then $s_k:=0$.





\begin{theorem}\label{T:2.1}
\cite[Theorem 3.2.2]{Pra}.
For each $n\in\mathbb{N}$, we have
\[
I_n=\bigcup\limits_{k=0}^{2^n -1} \left[1+\frac{s_k}{3^n},1+ \frac{s_k +1}{3^n}\right]\subset I_0.
\]
\end{theorem}
\begin{remark}\label{r:2.2}
The closed subset $I_n\subset I_0$ is the union of $2^{n}$ disjoint closed subsets intervals having length $\dfrac{1}{3^n}$. Moreover, the middle thirds removed from $I_{n-1}$ \emph{i.e.}, $\left(1+\dfrac{s_{2k}+1}{3^n}, 1+\dfrac{s_{2k+1}}{3^n}\right)$ also have length $\frac{1}{3^n}$, for each $k\in \{0,\ldots, 2^{n-1}-1\}$.
\end{remark}
Therefore, the intersection of closed subset
\[
2^{\omega}:=\bigcap_{n\in\mathbb{N} } I_n
\]
is well-known as the Cantor set, which is the only totally disconnected, perfect compact metric space (up to homeomorphism), (see \cite[Corollary 30.4]{Will}).



\subsection{Ends spaces}\label{section2.2}

We star by introducing the ends space of a topological space $X$ in the most general context, we shall employ it to clear-cut topological spaces $X$: surfaces and the graph well-known as the \emph{Cantor binary tree}. Let $X$ be a locally compact, locally connected, connected Hausdorff space.


\begin{definition}\label{d:2.3}
\cite{Fre}. Let $U_1\supset U_2\supset \cdots$ be an infinite nested sequence of non-empty connected open subsets of $X$, so that the boundary of $U_n$ in $X$ is compact for every $n\in\mathbb{N}$, $\cap_{n\in\mathbb{N}}\overline{U}$, and for any compact subset $K$ of $X$ there is $l\in\mathbb{N}$ such that $U_{l}\cap K=\emptyset$. We shall denote the sequence $U_1\supset U_2\supset \cdots$ as $(U_n)_{n\in\mathbb{N}}$. Two sequences $(U_{n})_{n\in\mathbb{N}}$ and $(U_{n}^{'})_{n\in \mathbb{N}}$ are equivalent if for any $l \in \mathbb{N}$ it exists $k \in \mathbb{N}$ such that $U_{l}\supset U_k^{'}$ and $n \in \mathbb{N}$ it exists $m \in \mathbb{N}$ such that $U_{n}'\supset U_m$.
The corresponding equivalence classes are called the \textit{topological ends} of $X$. We will denote the space of ends  by $Ends(X)$ and each equivalence class $[U_{n}]_{n\in\mathbb{N}}\in Ends(X)$ is called an \textit{end} of $X$.
\end{definition}

For every non-empty open subset $U$ of $X$ in which its boundary $\partial U$ is compact, we define:
\begin{equation}\label{eq:1}
U^{*}:=\{[U_{n}]_{n\in\mathbb{N}}\in  Ends(X)\, : \, U_{j}\subset U \text{ for some }j\in\mathbb{N}\}.
\end{equation}
The collection formed by all sets of the form $U\cup U^{*}$, with $U$ open with compact boundary of $X$, forms a base for the topology of $X':=X\cup Ends(X)$.

\begin{theorem}\label{t:2.2}
\cite[Theorem 1.5]{Ray}.
Let $X':=X\cup Ends(X)$ be the topological space defined above. Then,
\begin{enumerate}
\item The space $X'$ is Hausdorff, connected and locally connected.
\item The space $Ends(X)$ is closed and has no interior points in $X'$.
\item The space $Ends(X)$ is totally disconnected in $X'$.
\item The space $X'$ is compact.
\item If $V$ is any open connected set in $X'$, then $V\setminus Ends(X)$ is connected.
\end{enumerate}
\end{theorem}

\textbf{Ends of surfaces.} When $X$ is a surface $S,$ the space $Ends(S)$ carries extra information, namely, those ends that carry \emph{infinite} genus. This data, together with the space of ends and the orientability class, determines the topology of $S$. The details of this fact are discussed in the following paragraphs. Given that, this article only deals with orientable surfaces; from now on, we dismiss the non-orientable case. 

A surface is said to be \textit{planar} if all of its compact subsurfaces are of genus zero. An end $[U_n]_{n\in\mathbb{N}}$ is called \textit{planar} if there is $l\in\mathbb{N}$ such that $U_l$ is planar. The \textit{genus} of a surface $S$ is the maximum of the genera of its compact subsurfaces. Remark that, if a surface $S$ has \textit{infinite genus}, there is no finite set $\rm \mathcal{C}$ of mutually non-intersecting simple closed curves with the property that $S\setminus \mathcal{C}$ is  \textit{connected and planar}. We define $Ends_{\infty}(S)\subset Ends(S)$ as the set of all ends of $S$ which are not planar. It comes from the definition that $Ends_{\infty}(S)$ forms a closed subspace of $Ends(S)$.

\begin{theorem}[Classification of non-compact and orientable surfaces, \cite{Ker}, \cite{Ian}]\label{t:2.5}
Two non-compact and orientable surfaces $S$ and $S'$  having the same genus are homeomorphic if and only if there is a homeomorphism $f: Ends(S)\to Ends(S^{'})$ such that $f( Ends_{\infty}(S))= Ends_{\infty}(S')$.
\end{theorem}

\begin{proposition}\label{p:2.6}
\cite[Proposition 3]{Ian}. The space of ends of a connected surface $S$ is totally disconnected, compact, and Hausdorff. In particular, $Ends(S)$ is homeomorphic to a closed subspace of the Cantor set.
\end{proposition}

Of huge zoo composed by all non-compact surfaces our interest points to three of them. The first one is that surface which has infinite genus and only one end. It is called \emph{the Loch Ness monster} (see Figure \ref{loch}). This nomenclature is due to \emph{Phillips, A.} and \emph{Sullivan, D.} \cite{PSul}.
Remark that a surface $S$ has only one end if and only if for all compact subset $K \subset S$ there is a compact $K^{'}\subset S$ such as $K\subset K^{'}$ and $S\setminus  K^{'}$ is connected (see \cite{SPE}).

The other two remaining surfaces $S$ and $S^{'}$ are those that have ends space homeomorphic to the Canto tree. Further, all ends of $S$ are planar, while the ends of $S^{'}$ are all not planar. This surfaces are well-known as the \emph{Cantor tree} (see Figure \ref{cantor}- a.) and the \emph{Blooming Cantor tree} (see Figure \ref{cantor}- b.), respectively (see \cite{Ghys}).

\textbf{Cantor binary tree.} One of the fundamental objects we use in the proof of the theorems \ref{T:0.2} and \ref{T:0.3} is the infinite 3-regular tree. The ends spaces of this graph plays a distinguished role for the topological proof of our surfaces. We give binary coordinates to the vertex set of the infinite 3-regular tree, for we use these in a systematic way during the proofs of our main results.

For every $n\in\mathbb{N}$ let $2^{n}:=\prod_{i=1}^n\{0,1\}_i$ and let $\pi_i:2^n \to \{0,1\}$ be the projection onto the $i$-th coordinate. We define $V:=\{ D_s: D_s\in 2^n \text{ for some } n\in \mathbb{N}\}$ and $E$ as the union of $((0),(1))$ with the set $\{ (D_s,D_t) : D_s\in 2^{n}$ and $D_t\in 2^{n+1}$ for some $n\in\mathbb{N}$, and $\pi_i(D_s) =\pi_i(D_t)$  for every $i\in\{1,...,n\}\}$. The infinite 3-regular tree with vertex set $V$ and edge set $E$ will be called the \emph{Cantor binary tree} and denoted by $T2^\omega$, (see Figure \ref{Figure3}).
\begin{figure}[h!]
 \centering
 \includegraphics[scale=0.5]{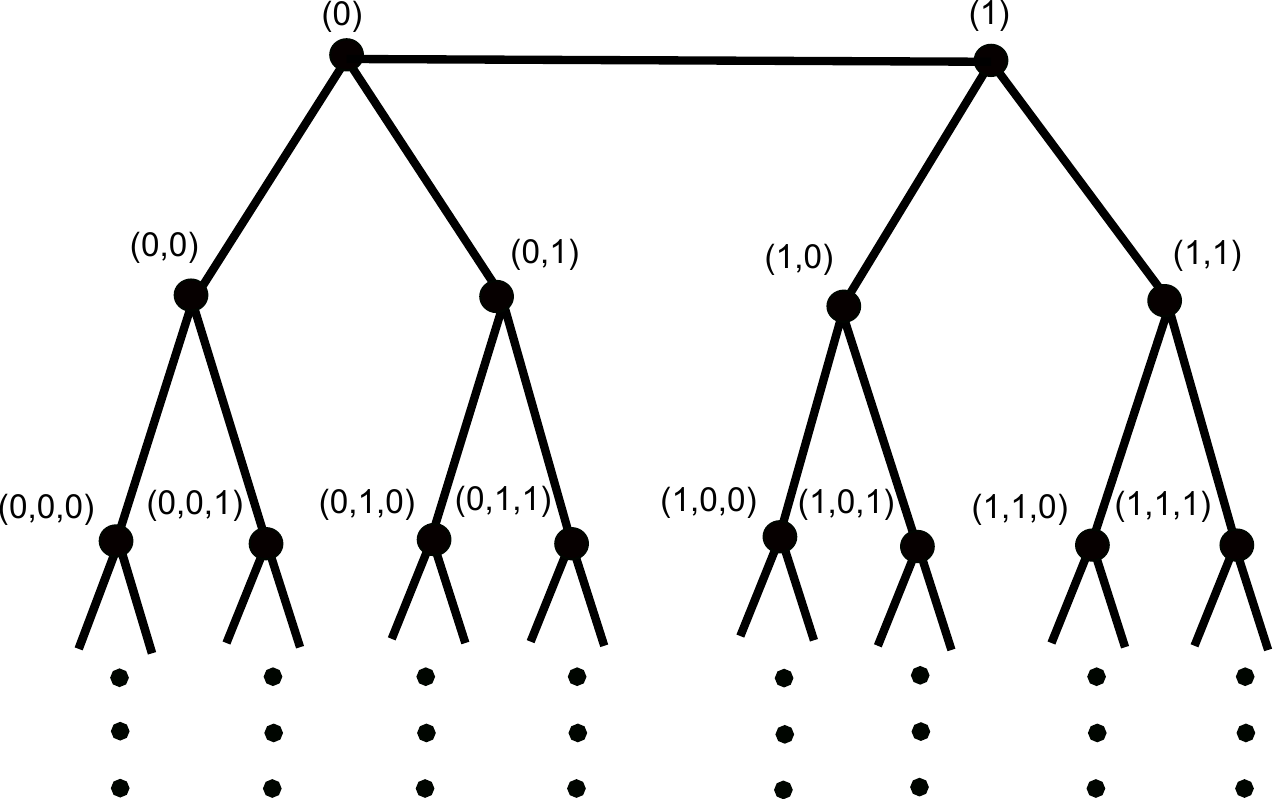}\\
 \caption{\emph{Cantor binary tree  $T2^{\omega}$.}}
 \label{Figure3}
\end{figure}

\begin{remark}\label{r:2.7}
Let $(v_n)_{n\in\mathbb{N}}$, where $v_n\in 2^\omega$ be an infinite simple path in $T2^\omega$. If we define $V_n$ as the connected component of $T2^{\omega}\setminus\{v_{n}\}$ such that $v_{n+1}\in V_n$, then $[V_n]\in Ends(T2^\omega)$ is completely determined by $(v_n)_{n\in\mathbb{N}}$. Hence, if we endow $\{0,1\}$ and $2^\omega:=\prod\limits_{i\in\mathbb{N}} \{0,1\}_{i} $ with the discrete and product topologies respectively, the map:
\begin{equation}\label{eq:3}
f :  \prod\limits_{i\in\mathbb{N}} \{0,1\}_{i}  \to   Ends(T2^{\omega}), \quad (x_{n})_{n\in\mathbb{N}}\mapsto ( v_n:=(x_1,\ldots,x_{n}))_{n\in\mathbb{N}},
\end{equation}
is a homeomorphism between the standard binary Cantor set and the space of ends of $T2^\omega$.
\end{remark}
Note that each end $[V_n]_{n\in \mathbb{N}}\in Ends(T2^{\omega})$ is determined by an infinite path $(v_n)_{n\in \mathbb{N}}\subset T2^{\omega}$ such that $v_n\in 2^{n}$ for every $n\in \mathbb{N}$ and viceversa.

\begin{remark}
\label{r:2.8}
Sometimes we will abuse of notation to denote by $f((x_n)_{n\in\mathbb{N}})$	both: the end defined by the infinite path $( v_n:=(x_1,\ldots,x_{n}))_{n\in\mathbb{N}}$ and the infinite path itself in $T2^\omega$.
\end{remark}


\subsection{Hyperbolic plane}\label{section2.3}
 
 Let $\mathbb{C}$ be the complex plane. Namely the upper half-plane $\mathbb{H}:=\{z\in\mathbb{C}: Im(z)>0\}$ equipped with the riemannian metric $ds=\frac{\sqrt{dx^2+dy^2}}{y}$ is well known as either the \emph{hyperbolic} or \emph{Lobachevski} plane. It comes with a group of transformation called the \emph{isometries} of $\mathbb{H}$ denoted by $Isom(\mathbb{H})$, which preserves the  hyperbolic distance on $\mathbb{H}$ defined by $ds$. Strictly, the group $PSL(2, \mathbb{R})$ is a subgroup $Isom(\mathbb{H})$ of index 2, where $PSL(2,\mathbb{R})$ is composed by all \emph{fractional linear transformations} or \emph{M\"{o}bius transformations}
\begin{equation}\label{eq:4}
f:\mathbb{C} \to \mathbb{C}, \quad z \mapsto \dfrac{az+b}{cz+d},
\end{equation}
where $a,b,c$ and $b$ are real numbers satisfying $ad-bc=1$. The group $PSL(2,\mathbb{R})$  could also be thought as the set of all real matrices having determinant one. Throughout this paper, unless specified in a different way, we shall always write the elements of $PSL(2,\mathbb{R})$ as \emph{M\"{o}bius transformations}.



\vspace{2mm}
\textbf{Half-circles.} Remind that the \emph{hyperbolic geodesics} of $\mathbb{H}$ are the half-circles and straight lines orthogonal to the real line $\mathbb{R}:=\{z\in \mathbb{C}: Im(z)=0\}$. Given a half-circle $C$ which center and radius are $\alpha\in\mathbb{R}$ and $r>0$ respectively, then the
set $\check{C}(f):=\{z\in\mathbb{H}: |z-\alpha|<r\}$ is called the \emph{inside} of $C$. Contrary, the set $\hat{C}(f):=\{z\in\mathbb{H}: |z-\alpha|>r\}$ is called the \emph{outside} of $C$. See the Figure \ref{Figure4}.

 \begin{figure}[h!]
\begin{center}
\includegraphics[scale=0.45]{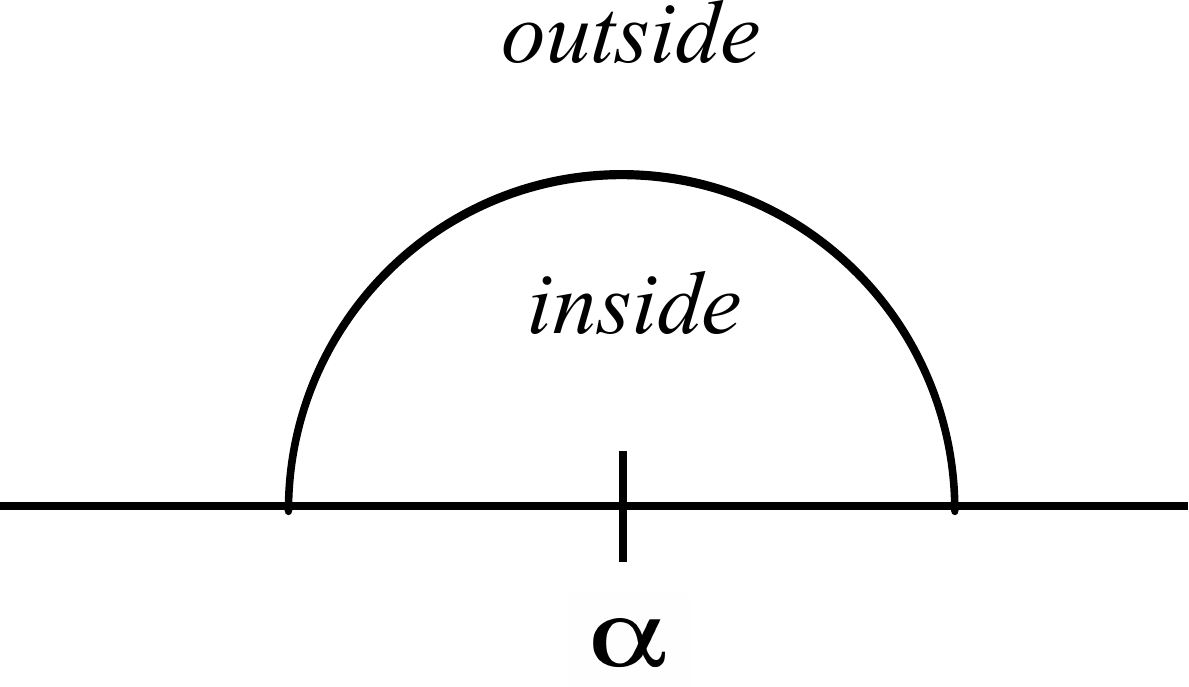}\\
  \caption{\emph{The inside and the outside of the half-circle $C$.}}
   \label{Figure4}
\end{center}
\end{figure}


Given the M\"{o}bius transformation $f\in PSL(2,\mathbb{R})$ as in equation (\ref{eq:4}) with $c\neq 0$, 
then the half-circle
\begin{equation}
\label{eq:5}
C(f):=\{z\in\mathbb{H}:| cz + d |^{-2}=1\}
\end{equation}
will be called the \emph{isometric circle of} $f$ (see \emph{e.g.}, \cite[p. 9]{MB}). We note that $\dfrac{-d}{c}\in\mathbb{R}$ the center of $C(f)$ is mapped by $f$ onto the infinity point  $\infty$. Further, $f$ sends the half-circle $C(f)$ onto $C(f^{-1})$ the isometric circle of the M\"{o}bius transformation $f^{-1}$, as such:
\begin{equation}
\label{eq:6}
C(f^{-1})=\{z\in\mathbb{H}:| -cz + a |^{-2}=1\}.
\end{equation}

%

\begin{remark}
\label{r:2.9}
The isometric circles $C(f)$ and $C^{-1}(f)$ have the same radius $r= | c |^{-1}$, and their respective centers are $\alpha=\dfrac{-d}{c}$ and $\alpha^{-1}=\dfrac{a}{c}$. 
\end{remark}

Given the half-circle $C$, then the M\"{o}bius transformation $f_C\neq Id\in PSL(2,\mathbb{R})$ reflecting with $C$ as the set of fixed points will be called the \emph{reflection respect to} $C$. We note that $f_C$ exchanges the ends points of $C$ and maps $\check{C}$ the inside of $C$ (respectively, $\hat{C}$ the outside of $C$) onto $\hat{C}$ the outside of $C$ (respectively, $\check{C}$ the inside of $C$). If $\alpha\in\mathbb{R}$ and $r>0$ are the center and the radius of $C$ then the relection $f_{C} : \mathbb{H}  \to  \mathbb{H}$ is defined as follows
\begin{equation}
\label{eq:8}
z \mapsto \dfrac{\dfrac{\alpha }{r}z +\left(\dfrac{- \alpha^2}{r} - r\right)}{\dfrac{z}{r}-\dfrac{\alpha}{r}}.
\end{equation}



\begin{remark}
 \label{r:2.10}
We let $L_{\alpha-2r}$ and $L_{\alpha+2r}$ be the two straight,  orthogonal lines  to the real axis $\mathbb{R}$ through the points $\alpha-2r$ and $\alpha+2r$, respectively. Then the reflection $f_C$ sends $L_{\alpha-2r}$ (analogously, $L_{\alpha+2r}$) onto the half-circle whose ends points are $\alpha+\dfrac{r}{2}$ and $\alpha$ (respectively, $\alpha-\dfrac{r}{2}$ and $\alpha$). See the Figure \ref{Figure5}.
\begin{figure}[h!]
\begin{center}
\includegraphics[scale=0.5]{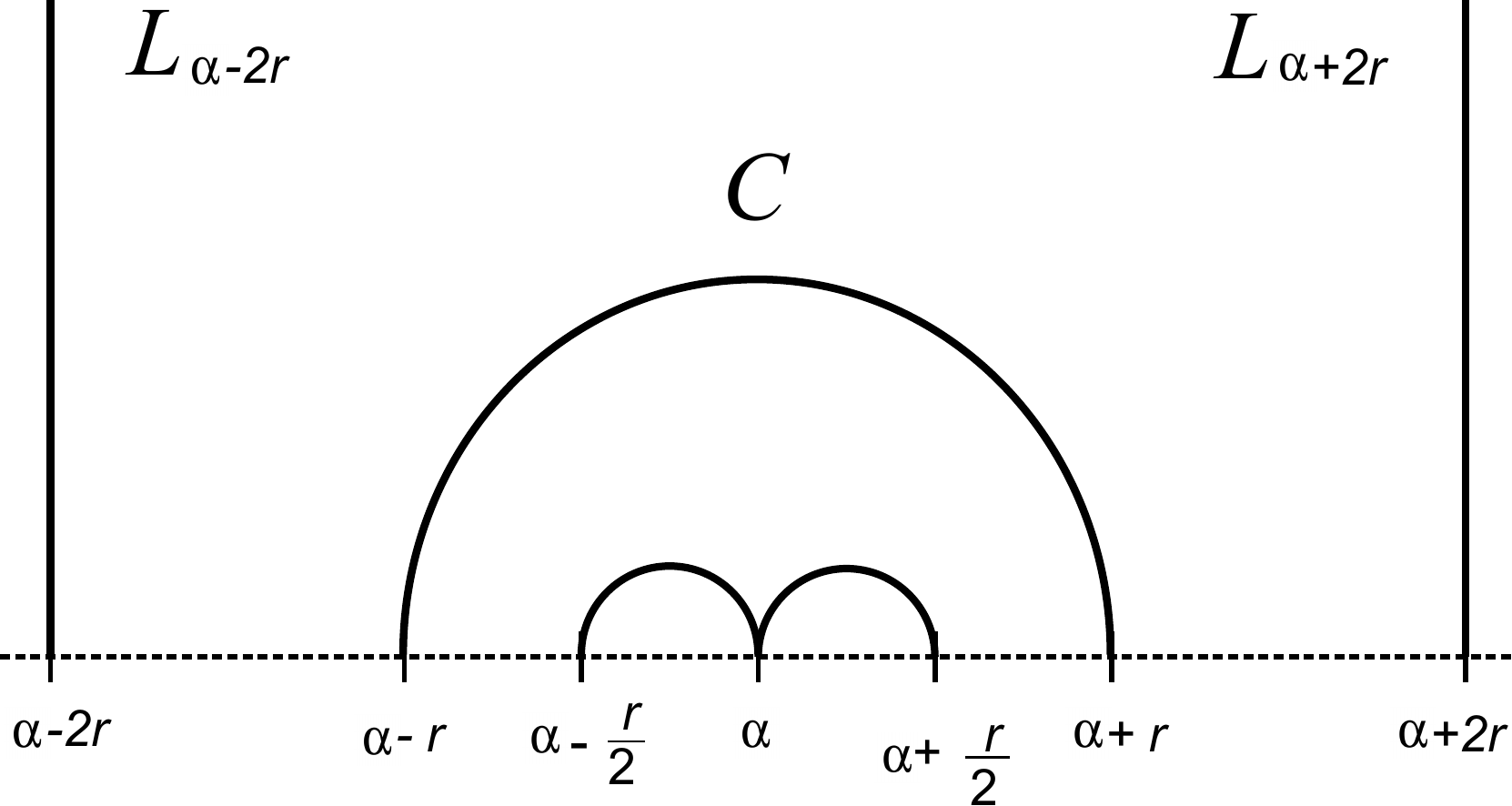}\\
  \caption{\emph{Reflection  respect to the half-circle $C$.}}
   \label{Figure5}
\end{center}
\end{figure}
Given that $f_C$ is an element of $PSL(2,\mathbb{R})$, then for every $\epsilon < \dfrac{r}{2}$ the closed hyperbolic $\epsilon$-neighborhood of the half-circle $C$ does not intersect any of the hyperbolic geodesics $L_{\alpha-2r}$, $f_{C}(L_{\alpha-2r})$, $L_{\alpha+2r}$, and $f_{C}(L_{\alpha+2r})$.
\end{remark}

\begin{lemma}\label{l:2.11}
 Let  $C_1$ and $C_2$ be two disjoint half-circles having centers and radius $\alpha_1, \alpha_2 \in\mathbb{R}$, and $r_1, r_2 > 0$, respectively. Suppose that the complex norm $|\alpha_1 -\alpha_2 |> (r_1 + r_2)$ then for every $\epsilon<\dfrac{\max\{r_1, r_2\}}{2}$ the closed hyperbolic $\epsilon$-neighborhoods of the half-circles $C_1$ and $C_{2}$ are disjoint.
\end{lemma}

\begin{proof}



By hypothesis $|\alpha_1 -\alpha_2 |> (r_1 + r_2)$, then the open strips $S_1$ and $S_2$ are disjoint, where
\[
S_{1}:=\{z\in\mathbb{H}: \alpha_1-2r_1 < Re(z)<\alpha_1+2r_1 \} \quad \text{and} \quad S_2:=\{z\in\mathbb{H}: \alpha_{2}-2r_{2} < Re(z)<\alpha_{2}+2r_{2}\}.
\]
We remark that the half-circle $C_i$ belongs to the open strip $S_i$, for every $i\in\{1, 2\}$. Further, the transformation
\begin{equation}\label{eq:9}
f : \mathbb{H}  \to \mathbb{H}, \quad z \mapsto \dfrac{\dfrac{r_2z}{\sqrt{r_1 r_2}} +\dfrac{(r_1C_2-r_2C_1)}{\sqrt{r_1 r_2}}}{\dfrac{r_1}{\sqrt{r_1 r_2}}}
\end{equation}
of $PSL(2,\mathbb{R})$ sends the open strip $S_1$ onto the open strip $S_2$ and vice versa. We must suppose without generality that $r_1=\max\{r_1, r_2\}$. Then, using the remark \ref{r:2.10} above, we have that for every $\epsilon <\frac{r_1}{2}$ the closed hyperbolic $\epsilon$-neighborhood of the half-circle $C_1$ is contained in the open strip $S_1$. Now, given that $f$ is an element of $PSL(2,\mathbb{R})$, then the closed hyperbolic $\epsilon$-neighborhood of the half-circle $C_2$ is contained in the open strip $S_2$. Since $S_1 \cap S_2 =\emptyset$ that implies that for every $\epsilon<\dfrac{\max\{r_1, r_2\}}{2}$ the closed hyperbolic $\epsilon$-neighborhoods of the half-circles $C_1$ and $C_{2}$ are disjoint.
\end{proof}




\subsection{Fuchsian groups and Fundamental region}\label{section2.4}
The group $PSL(2,\mathbb{R})$ comes with a topological structure from the quotient space between a group composed by all the real matrices $g:=\left(
                                                                                                       \begin{array}{cc}
                                                                                                         a & b \\
                                                                                                         c & d \\
                                                                                                       \end{array}
                                                                                                     \right)
$ with determinant exactly $det(g)=1$ under $\{ \pm Id\}$. A subgroup $\Gamma$ of $PSL(2,\mathbb{R})$ is called \emph{Fuchsian} if $\Gamma$ is discrete. We shall denote as $PSL(2,\mathbb{Z})$ the Fuchsian group of all M\"obius transformation with entries in the integers numbers.

\begin{definition}\label{d:2.12}
\cite{KS} Given a Fuchsian group $\Gamma < PSL(2,\mathbb{R})$. A closed region $R$ of the hyperbolic plane $\mathbb{H}$ is said to be \emph{a fundamental region} for $\Gamma$ if it satisfies the following facts:

\textbf{(i)} The union $\bigcup\limits_{f\in \Gamma} f(R)=\mathbb{H}$.

\textbf{(ii)} The intersection of the interior sets $Int(R) \cap f(Int (R)) = \emptyset$ for each $f\in \Gamma \setminus \{Id\}$.

\noindent The different set $\partial R= R \setminus Int(R)$ is called the \emph{boundary} of $R$ and the family $\mathfrak{T}:=\{f(R): f\in \Gamma\}$ is called the \emph{tessellation} of $\mathbb{H}$.
\end{definition}

 If $\Gamma$ is a Fuchsian group and each one of its elements are described as the equation \ref{eq:4} such that $c\neq 0$, then the subset $R_0$ of $\mathbb{H}$ defined as follows

\begin{equation}
\label{eq:10}
R_0 :=\bigcap\limits_{f\in \Gamma} \overline{\hat{C}(f)}\subseteq \mathbb{H},
\end{equation}
is a fundamental domain for the group $\Gamma$ (see \emph{e.g.}, \cite{Ford}, \cite[Theorem H.3 p. 32]{MB},  \cite[Theorem 3.3.5]{KS}). The fundamental domain $R_0$ is well-known as the \emph{Ford region for} $\Gamma$.

On the other hand, we can get a Riemann surface from any Fuchsian group $\Gamma$. It is only necessary to define the action as follows
\begin{equation}\label{eq:11}
 \alpha : \Gamma \times \mathbb{H} \to  \mathbb{H}, \quad (f,z) \mapsto f(z),
\end{equation}
which is proper and discontinuous (see \cite[Theorem 8. 6]{KS2}). Now, we define the subset
\begin{equation}
\label{eq:12}
K:=\{w\in\mathbb{H}: f(w)=w \text{ for any } f\in \Gamma-\{Id\}\}\subseteq\mathbb{H}.
\end{equation}
We note that
\begin{itemize}
\item \emph{The subset $K$ is countable and discrete.} Given that each element contained in $\Gamma$ fixes finitely many points of $\mathbb{H}$ and every Fuchsian group is countable. Thus we conclude that the set $K$ is countable. Contrarily, if $K$ is not discrete, then there is a point $w\in K$ and $\varepsilon >0$ such that the ball $B_{\varepsilon}(w)$ contains infinitely many  points of $K$. In other words, the set $\{f\in \Gamma : f(B_{\varepsilon}(w))\cap B_{\varepsilon}(w) \neq \emptyset\}$ is infinite. Clearly, it is a contradiction to proper discontinuity of the group $\Gamma$ on $\mathbb{H}$.
\item \emph{The action $\alpha$ leaves invariant the subset $K$.} If $w$ is a point contained in $K$, which is fixed by the element $f\in \Gamma$, then for any $g\in \Gamma$ the point $f(w)$ is fixed by the composition of isometries $f \circ g \circ f^{-1}(z)$.
\end{itemize}
Then the action $\alpha$ restricted to the hyperbolic plane $\mathbb{H}$ removing the subset $K$ is free, proper and discontinuous. 
Therefore, the quotient space (also called the space of the $\Gamma$-orbits)
\begin{equation}
\label{eq:13}
S:= (\mathbb{H}-K)/\Gamma
\end{equation}
is well-defined and via the projection map
\begin{equation}\label{eq:14}
\pi : (\mathbb{H}-K) \to  S, \quad  z \mapsto [z].
\end{equation}
It comes with a hyperbolic structure, it means, $S$ is a Riemann surface (see \emph{e.g.}, \cite{LJ}).

\begin{remark}\label{r:2.13}
If $R$ is a locally finite\footnote{In the sense due to Beardon  on \cite[Definition 9.2.3]{Bear}.} fundamental domain for the Fuchsian group $\Gamma$, then the quotient space $(\mathbb{H}-K)/\Gamma$ is homeomorphic to $R/\Gamma$ (see the Theorem 9.2.4 on \cite{Bear}).
\end{remark}




\subsection{Classical Schottky groups and Geometric Schottky groups}\label{section2.5}

In general, by a classical Schottky group it is understood a finitely generated subgroup of $PSL(2, \C)$, generated by isometries sending the exterior of one circular disc to the exterior of a different circular disc, both of them disjoint. From the various definitions, we consider the one given in \cite{MB}, but there are alternative and similar definitions that can be found in \cite{BJ, Carne, TM}.

Let $C_1, C_1^{\prime}, \ldots, C_n, C_n^{\prime}$ be a set of disjoint countably of circles in the extended complex plane $\widehat{\mathbb{C}}$, for any $n\in\mathbb{N}$, bounding a common region $D$. For every $j\in\{1,\ldots,n\}$, we consider the  M\"{o}bius transformation $f_j$, which sends the circle $C_j$ onto the circle $C^{'}_j$, \emph{i.e.}, $f_j(C_j)= C_j^{\prime}$ and $f_j(D)\cap D=\emptyset$. The group $\Gamma$ generated by the set $\{f_j, f_j^{-1}: j\in \{1,\ldots, n\}\}$ is called a \emph{classical Schottky group}.

The \emph{Geometric Schottky groups} can be acknowledged as a nice generalization of the \emph{Classical Schottky groups} because the definition of the first group is extended to the second one, in the sense that the Classical Schottky groups are  finitely generated by definition, while the Geometric Schottkky group can be infinitely generated. This geometric groups were done thanks to \emph{Anna, Z.} (see \cite[Section 3]{Ziel}) and they are the backbone of our main result \ref{T:0.2} and \ref{T:0.3}.



A subset $I\subseteq \mathbb{Z}$ is called \emph{symmetric} if it satisfies that $0\notin I$ and for every $k\in I$ implies $-k\in I$.

\begin{definition}\label{d:2.14}
\cite[Definition 2. p. 28]{Ziel} Let $\{A_k:k\in I\}$ be a family of straight segments in the real line $\mathbb{R}$, where $I$ is a symmetric subset of $\mathbb{Z}$ and let $\{f_k:k\in I\}$ be a subset of $PSL(2, \mathbb{R})$. The pair
\begin{equation}\label{eq:15}
\mathfrak{U}(A_k,f_k,I):=(\{A_k\}, \{f_k\})_{k\in I}
\end{equation}
 is called a \emph{Schottky description}\footnote{The writer gives this definition to the Poincar\'e disc and we use its equivalent to the half plane.} if it satisfies the following conditions:
\begin{enumerate}
  \item The closures subsets $\overline{A_k}$ in $\mathbb{C}$ are mutually disjoint.
  \item None of the $\overline{A_k}$ contains a closed half-circle.
  \item For every $k\in I$, we denote as $C_k$ the half-circle whose ends points coincide to the ends points of $\overline{A_k}$, which  is the isometric circle of $f_k$. Analogously, the half-circle $C_{-k}$ is the isometric circle of $f_{-k}:=f^{-1}_k$.
  \item For each $k\in I$, the M\"obius transformation $f_k$ is hyperbolic.
  \item There is an $\epsilon> 0$ such that the closed hyperbolic $\epsilon$-neighborhood of the half-circles $C_{k}$, $k\in I$ are pairwise disjoint.
\end{enumerate}
\end{definition}

\begin{definition}\label{d:2.12}
\cite[Definition 3. p. 29]{Ziel} A subgroup $\Gamma$ of $PSL(2,\mathbb{R})$ is called \emph{Schottky type} if there exists a Schottky description $\mathfrak{U}(A_k,f_k,I)$ such that the generated group by the set $\{f_k:k\in I\}$ is equal to $\Gamma$, \emph{i.e.} we have $\Gamma=\langle f_k : k\in I \rangle$.
\end{definition}

We note that any Schottky description $\mathfrak{U}(A_k,f_k,I)$ defines a Geometric Schottky group.


\begin{proposition}\label{p:2.16}
\cite[Proposition 4. p.34]{Ziel} Every Geometric Schottky group $\Gamma$ is a Fucshian group.
\end{proposition}

The \emph{standard fundamental domain} for the Geometric Schottky group $\Gamma$ having Schottky description  $\mathfrak{U}(A_k,f_k, I)$ is the intersection of all outside of the half-circle associated to the transformations  $f_k$'s, \emph{i.e.},
\begin{equation}\label{eq:17}
F(\Gamma):=\bigcap_{k\in I} \overline{\hat{C}_{k}}\subset \mathbb{H}.
\end{equation}

\begin{proposition}\label{p:2:17}
\cite[Proposition 2. p. 33]{Ziel} The standard fundamental domain $F(\Gamma)$ is a fundamental domain for the Geometric Schottky group $\Gamma$.
\end{proposition}


\section{Main result}\label{section3}


In this section we present the proof of our main results, which are based on the following sketch. First, we shall build explicitly a suitable family of mutually disjoint half circle $\mathcal{C}$ (the geometrical construction of the Cantor set plays an important role in the building of this family during the proof of the theorems \ref{T:0.2}, and \ref{T:0.3}). Then we shall define the set $J$ composed by the M\"{o}bius transformation having as isometric circles the elements of $\mathcal{C}$. Immediately, we could prove that $\Gamma$ the subgroup of $PSL(2,\mathbb{R})$ generated by $J$ is a Fuchsian group. Also, we show up that the quotient space $\mathbb{H} / \Gamma$ is the desired non-compact surface.


\subsection{Proof Theorem \ref{T:0.1}}\label{section3.1}

\hfill \break

 \textbf{Step 1. Building the group $\Gamma$.} Given $\mathcal{C}$ the family composed by the whole half-circle with center the even numbers in the real line $\mathbb{R}$ and radius one (see the Figure \ref{Figure6}), we consider  $\Gamma$ the subgroup of $PSL(2,\mathbb{Z})$, which is generated by the set $\{f_n(z), g_n(z), f^{-1}_n(z), g^{-1}_n(z): n\in \mathbb{Z}\}$, where

\begin{align*}
f_{n}(z)      & := \frac{(8n+4)z-(1+8n(8n+4))}{z-8n}, \\
g_{n}(z)      & := \frac{(8n+6)z +(-1-(8n+2)(8n+6))}{z-(8n+2)},\\
f_{n}^{-1}(z) & := \frac{-8nz+(1+8n(8n+4))}{-z+(8n+4)}, \\
g_{n}^{-1}(z) & := \frac{-(8n+2)z +(1+(8n+2)(8n+6))}{-z+(8n+6)}.
\end{align*}
\begin{figure}[h!]
\begin{center}
\includegraphics[scale=0.7]{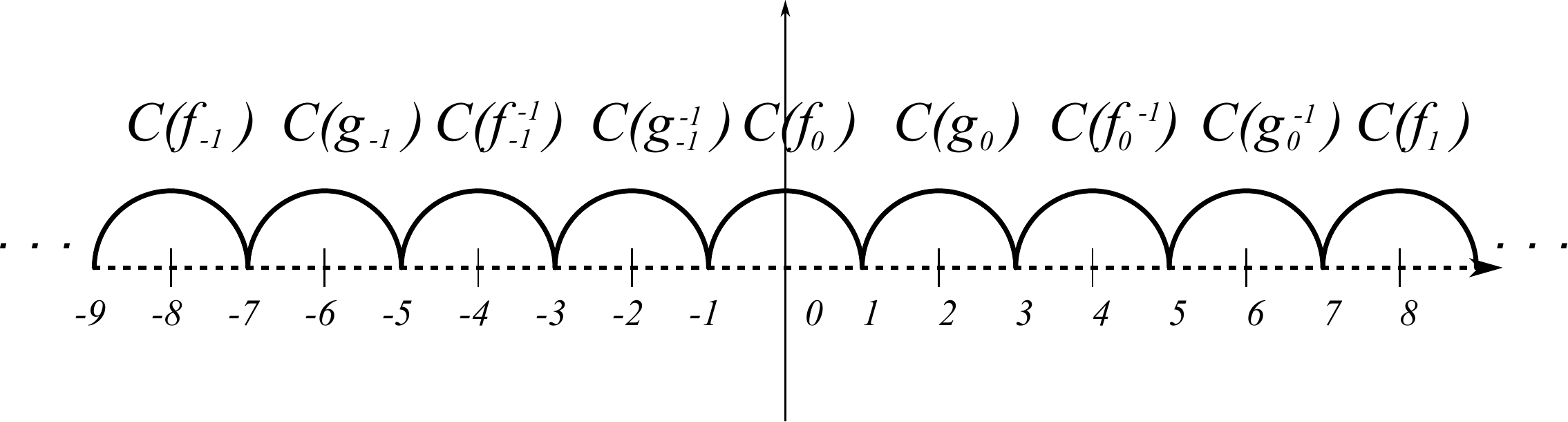}\\
\caption{\emph{Family of half-circle $\mathcal{C}$.}}
 \label{Figure6}
\end{center}
\end{figure}

We remark that for any $n\in\mathbb{Z}$, the M\"obius transformations $f_{n}(z)$ and $f^{-1}_{n}(z)$ of $PSL(2,\mathbb{R})$ have as isometric circle the two half-circles $C(f_n)$ and $C(f^{-1}_n)$ of $\mathcal{C}$, whit centers at the points $8n$ and $8n+4$ respectively, \emph{i.e.},
\begin{align*}
  C(f_{n})     & :=  \{z\in \mathbb{H}: |z-8n|=1\}, \\
  C(f_{n}^{-1})& := \{z\in\mathbb{H}: |-z+(8n+4)|=1\}.
\end{align*}
Analogously, the M\"obius transformations $g_{n}(z)$ and $g^{-1}_{n}(z)$ of $PSL(2,\mathbb{R})$ have as isometric circle the two half-circles $C(g_n)$ and $C(g^{-1}_n)$ of $\mathcal{C}$, whit centers at the points $8n+2$ and $8n+6$ respectively, \emph{i.e.},
\begin{align*}
C(g_{n})&:= \{z\in\mathbb{H}: |z-(8n+2)|=1\}, \\
C(g^{-1}_n) &:=\{z\in\mathbb{H}: |-z +(8n+6)|=1\}.\\
\end{align*}

Since $\Gamma$ is a subgroup of the Fuchsian group $PSL(2,\mathbb{Z})$, then $\Gamma$ is an infinitely generated \emph{Fuchsian group} composed by M\"obius transformations having integers coefficients. Moreover, the cardinality of the subgroup $\Gamma$ is countable. Now, we define the subset $K\subseteq \mathbb{H}$ as the equation \ref{eq:12}, then the Fuchsian group $\Gamma$ acts freely and properly discontinuously on the open subset $\mathbb{H}-K$. It means that the quotient space
\[
S:= (\mathbb{H}-K)/\Gamma
\]
is  a well-defined hyperbolic surface via the projection map $\pi: (\mathbb{H}-K)\to S$.

We note that to this case $K=\emptyset$ because the intersection of any two different elements belonged to $\mathcal{C}$ is either: empty or at infinity, that means, they meet in the same point in the real line $\mathbb{R}$.

\vspace{2mm}
\textbf{Step 2. The desired surface.} To end the proof we must prove that $S$ is the Infinite Loch Ness monster, \emph{i.e.}, it has infinite genus and only one end. The introduction of the following remark is necessary.

\begin{remark}\label{r:3.1}
The following facts come from the definition of the family $\mathcal{C}$ and the group $\Gamma$.


1. The family $\mathcal{C}$ can be written as $\mathcal{C}:=\{C(f_{n}),C(g_{n}), C(f_{n}^{-1}), C(g^{-1}_n): n\in\mathbb{Z} \}$. Remember that the intersection of any two different elements belonged to $\mathcal{C}$ is either: empty or at infinity, that is, they meet in the same point in the real line $\mathbb{R}$.

2. The Ford region $R_0$ associated to $\Gamma$  (see the Figure \ref{Figure7})
\begin{equation}\label{eq:15}
\begin{array}{ccl}
  R_0 & =&\bigcap\limits_{n\in\mathbb{Z}} \left( \overline{\hat{C}(f_{n})}\, \cap \overline{\hat{C}(f^{-1}_{n})}\, \cap\, \overline{\hat{C}(g_{n})} \, \cap \, \overline{\hat{C}(g^{-1}_{n})}\right),  \\
      &                           =                            &  \bigcap\limits_{n\in\mathbb{Z}} \big(\{z\in\mathbb{H}:|z-8n|^{-2} \leq 1\}\cap \{z\in\mathbb{H}:|-z+(8n+4)|^{-2} \leq 1\}\\
      &                                                       &\hspace{7mm} \cap \{z\in\mathbb{H}:|z-(8n+2)|^{-2} \leq 1\}\cap \{z\in\mathbb{H}:|-z+(8n+6)|^{-2} \leq 1\}),
\end{array}
\end{equation}
is a fundamental domain for the Fuchsian group $\Gamma$.
\begin{figure}[h!]
\begin{center}
\includegraphics[scale=0.7]{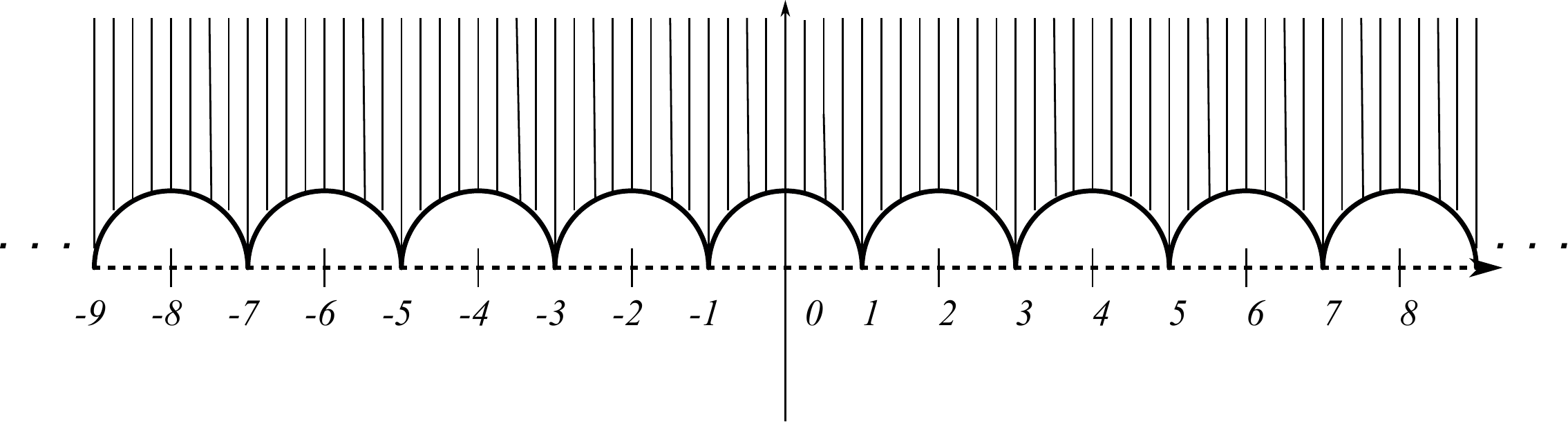}\\
\caption{\emph{Ford region $R_0$ associated to $\Gamma$.}}
 \label{Figure7}
\end{center}
\end{figure}
\end{remark}

Note that the fundamental domain $R_0$ given in the equation \ref{eq:15} is connected and locally finite having infinite hyperbolic area. Further, its boundary is the family of half-circles $\mathcal{C}$ \emph{i.e.}, it consists of infinite hyperbolic geodesic with ends points at infinite and mutually disjoint.


Given that $R_0$ the fundamental domain of $\Gamma$ is  non-compact Dirichlet region having infinite hyperbolic area, then the quotient space $S$ is also a non-compact hyperbolic surface with infinite hyperbolic area (see \cite[Theorem 14.3 p. 283]{KS2}). 

\vspace{2mm}
\textbf{The hyperbolic surface $S$ has only one end.} Let $K$ be a compact subset of $S$. We must prove that  there is a compact subset $K\subset K^{'}\subset S$ such that the difference $S \setminus K^{'}$ is connected. Seeing that the quotient $R_0/\Gamma$ is homeomorphic to $S$ we must suppose that there is a compact subset $B\subset R_0$ such that the projection map $\pi$ (see the equation \ref{eq:14}) sends the intersection $B\cap R_0$ to $K$ \emph{i.e.}, $\pi(B\cap R_0)=K$. Given that the hyperbolic plane $\mathbb{H}$ has exactly one end, then there exist two closed intervals $I_1, I_2\subset\mathbb{R}$ such that  $B\subset I_1 \times I_2\subset \mathbb{H}$, and the difference $\mathbb{H}\setminus (I_1 \times I_2)$ is connected. The projection map $\pi$ sends the intersection $B\cap R_0$ into a compact subset of $S$, which we denote by
\begin{equation}\label{eq:19}
\pi(B\cap R_0):= K^{'}\subset S.
\end{equation}
We note that by construction $K\subseteq K^{'}$. We claim that $S\setminus K^{'}$ is connected. Let $[z]$ and $[w]$ be two different points belonged to $S\setminus K^{'}$ we shall build a path in $S\setminus K^{'}$ joining both points.  On the other hand, for every $x\in \mathbb{R}\setminus I_1$ we consider the geodesic $\gamma_x$ of hyperbolic plane $\mathbb{H}$, which is perpendicular line to the real line and one of its ends points is $(x,0)$. Similarly, for every $y>0$ with $y\notin I_2$ we consider the connected subset $\gamma^y:=\{z \in \mathbb{H}: Im(z)=y\}$ of the hyperbolic plane $\mathbb{H}$.
\begin{remark}
\label{rem:3.3}
 The subsets $\gamma_x$ and $\gamma_y$ have the following properties:
 \begin{itemize}
 \item The intersection $\gamma_x \cap (I_1 \times I_2)=\emptyset$, and the projection map $\pi$ sends the set $\gamma_x\cap R_0$ into a connected subset of $S$.

 \item The intersection $\gamma^y \cap (I_1 \times I_2 )=\emptyset$, and the projection map $\pi$ sends the set $\gamma^y\cap R_0$ into a connected subset of $S$.
 \end{itemize}
\end{remark}
Given the two different equivalent classes $[z]$ and $[w]$ of $S\setminus K^{' }$ without loss of generality we can assume that $z,w\in R_0\setminus (I_{1}\times I_2)$, then there exist two connected subsets $\gamma$ and $\gamma^{'}$ as shown above such that $z\in \gamma$ and $w\in \gamma^{'}$. If $\gamma\cap \gamma^{'}\neq \emptyset$ then the image of $(\gamma\cup \gamma^{'})\cap R_0$ under $\pi$ is a connected subset belonged to $S\setminus K^{'}$ containing the points  $[z]$ and $[w]$. Oppositely, if $\gamma\cap \gamma^{'}= \emptyset$ then there is a connected subset $\gamma^{\ast}$ as shown above such that $\gamma\cap \gamma^{\ast}\neq \emptyset$ and $\gamma^{'}\cap \gamma^{\ast}\neq \emptyset$. Consequently, the image of $(\gamma\cup \gamma^{'}\cup \gamma^{\ast})\cap R_0$ under $\pi$  is a connected subset belonged to $S\setminus K^{'}$ containing the points $[z]$ and  $[w]$. This proves that the subset $S\setminus K^{'}$ is connected.

\vspace{2mm}
\textbf{The hyperbolic surface $S$ has infinite genus.} For every $n\in\mathbb{Z}$, we defined the subset
\[
S_n:=\{z\in\mathbb{H}: 0< Im(z) < 3 \text{ and } -1+8n < Re(z) < 7+8n\}\subset \mathbb{H}.
\]

The projection map $\pi$ sends the intersection $S_n \cap R_0$ into a subsurface with boundary $\widehat{S}_n \subset S$, which is homeomorphic to the torus punctured by only one point (see the Figure \ref{Figure8}). Furthermore, for any two different integers $m\neq n\in\mathbb{Z}$ the subsurfaces $\widehat{S}_n$ and $\widehat{S}_m$ are disjoint. Thus, we conclude that the hyperbolic surface $S$ has infinite genus.
\qed

\begin{figure}[h!]
\begin{center}
\includegraphics[scale=0.7]{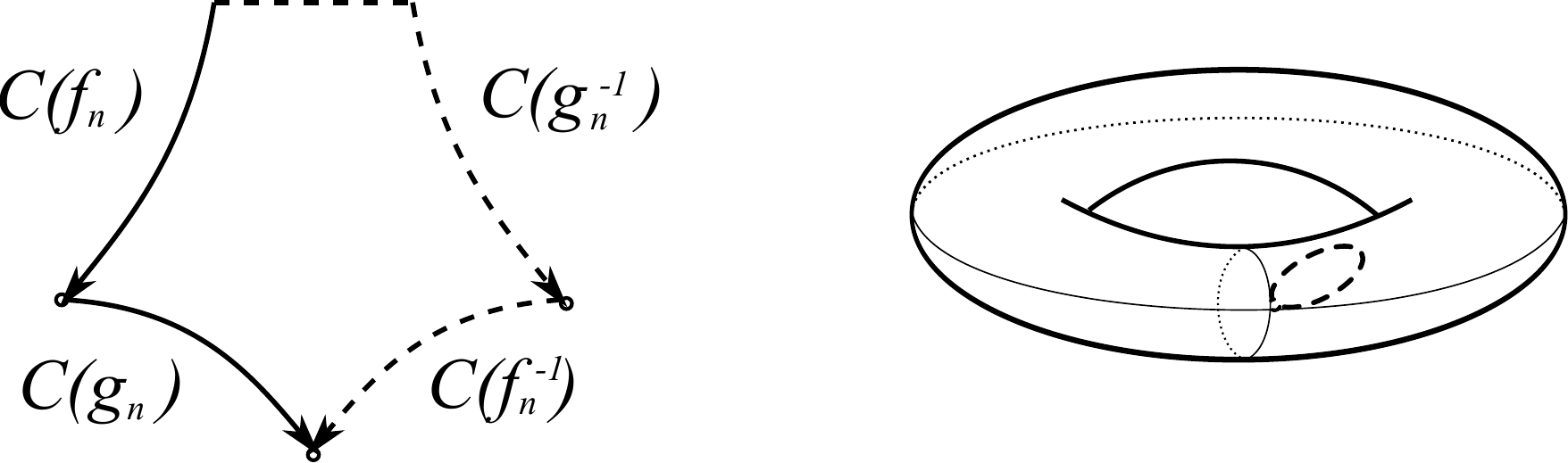}\\
\caption{\emph{Subsurface $\widehat{S}_n\subset S$}.}
 \label{Figure8}
\end{center}
\end{figure}

So, from Theorem \ref{T:uni} we can conclude that:

\begin{corollary}
The fundamental group of the Infinite Loch Ness monster is isomorphic to $\Gamma$.
\end{corollary}


\subsection{Proof Theorem \ref{T:0.2}}\label{section3.2}

\hfill \break

\textbf{Step 1. Build the group $\Gamma$.}

\vspace{2mm}
\noindent \textbf{For $n=1$. Build the set $J_1$ containing exactly two M\"obius transformations and the set $\mathcal{C}_1$ composed by its respective isometric circles.}  We consider the closed interval $I_0=[1,2]$ and its symmetrical with respect to the imaginary axis $-I_0=[-2,-1]$ (see the theorem \ref{T:2.1}). We let $\hat{I}_{1}, -\hat{I}_{1}$ be the middle third  of $I_0$ and $-I_0$, respectively. By the remark \ref{r:2.2} we have
\begin{align*}
\hat{I}_1  & =\left[\dfrac{4+s_0}{3}, \dfrac{3+s_1}{3} \right],\\
-\hat{I}_1 & =\left[-\dfrac{3+s_1}{3},-\dfrac{4+s_0}{3}  \right].\\
\end{align*}
Given that $s_0=0$ and $s_1=2$ then we get $\hat{I}_1=\left[\dfrac{4}{3}, \dfrac{5}{3}\right]$ and $-\hat{I}_1:=\left[-\dfrac{5}{3}, -\dfrac{4}{3}\right]$. See the Figure \ref{Figure9}.
\begin{figure}[h!]
\begin{center}
\includegraphics[scale=0.5]{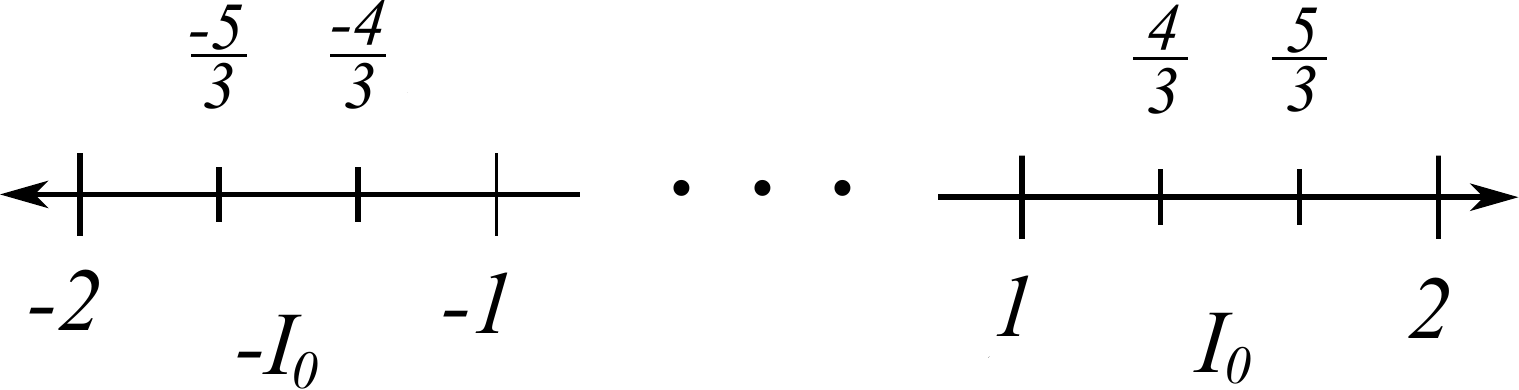}\\
\caption{\emph{The middle thirds of the closed intervals $[1,2]$ and $[-2,-1]$}.}
 \label{Figure9}
\end{center}
\end{figure}

\begin{remark}
If we remove $\hat{I}_1$ and $-\hat{I}_{1}$ of the closed intervals $I_0$ and $-I_0$ respectively, we hold the sets
\[
\begin{array}{rclclc}
I_0\setminus \hat{I}_{1} & = &  I_1 
                                      &  = & \left[1, \dfrac{4+s_0}{3}\right]\bigcup \left[\dfrac{3+s_1}{3}, \dfrac{4+s_1}{3}\right],\\
                           &   &      &  = & \left[1, \dfrac{4}{3}\right]\bigcup \left[\dfrac{5}{3}, 2\right],\\
(-I_0)\setminus(-\hat{I}_{1})  & = &-I_1 
                                         & = & \left[-\dfrac{4+s_1}{3},-\dfrac{3+s_1}{3} \right]\bigcup \left[-\dfrac{4+s_0}{3},-1 \right],\\
                               &   &     & = & \left[-2,-\dfrac{5}{3} \right]\bigcup \left[-\dfrac{4}{3},-1 \right].\\
\end{array}
\]
\end{remark}

On the other hand, the length of closed intervals $\hat{I}_{1,1}$ and $-\hat{I}_{1,1}$ is $\dfrac{1}{3}$ then we choose their respective middle points, which are given by
\begin{equation}\label{eq:17}
\begin{array}{cccl}
\alpha_{1,1}      & := & \dfrac{4+s_0}{3}+\dfrac{1}{3\cdot 2}=\dfrac{9+2 s_0}{3\cdot 2}   & = \dfrac{3}{2}, \\
\alpha_{1,1}^{-1} & := &-\dfrac{4+s_0}{3}-\dfrac{1}{3\cdot 2}=- \dfrac{9+2 s_0}{3\cdot 2} &= -\dfrac{3}{2}.\\
 \end{array}
\end{equation}
By construction the points $\alpha_{1,1}$ and $\alpha^{-1}_{1,1}$ are symmetrical with respect to the imaginary axis \emph{i.e.}, we have the equality $\alpha^{-1}_{1,1}=-\alpha_{1,1}$. Then we let $C(f_{1,1}),  C(f^{-1}_{1,1})$ be the two half-circles having centers $\alpha_{1,1}$ and $\alpha_{1,1}^{-1}$ respectively, and the same radius $r(1)=\dfrac{1}{3 \cdot 2^2}$ (see the Figure \ref{Figure10}). They are given by the formula
\begin{equation}\label{eq:18}
\begin{array}{rcl}
C(f_{1,1})      &=&\left\{z\in \mathbb{H}: \left| 3 \cdot 2^2 z -2\cdot(9+2 s_0)\right|^{-2}=1\right\},\\
                &=&\left\{z\in \mathbb{H}: \left| 12 z -18\right|^{-2}=1\right\},\\
C(f^{-1}_{1,1}) &=&\left\{z\in \mathbb{H}: \left| - 3 \cdot 2^2 z-2\cdot(9+2 s_0)\right|^{-2}=1\right\}, \\
                &=&\left\{z\in \mathbb{H}: \left| - 12 z-18\right|^{-2}=1\right\}. \\
\end{array}
\end{equation}

\begin{figure}[h!]
\begin{center}
\includegraphics[scale=0.5]{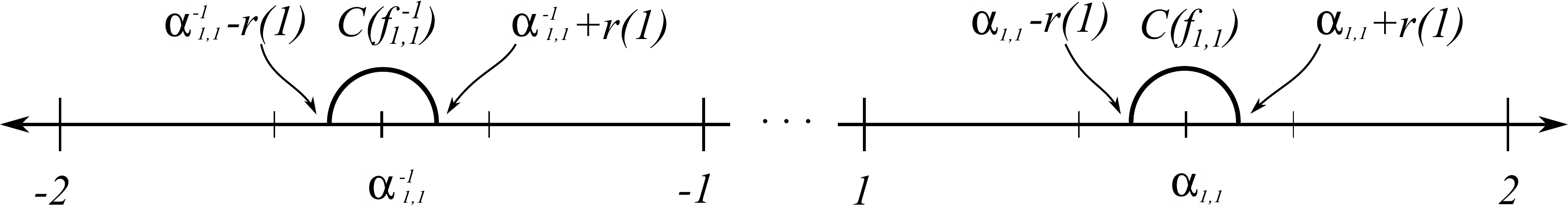}\\
\caption{\emph{Half-circles $C(f_{1,1})$ and $C(f^{-1}_{1,1})$. The points $\alpha_{1,1}+r(1)$ and $\alpha_{1,1}-r(1)$ of $\hat{I}_{1,1}$ are the end points of the half-circle $C(f_{1,1})$. Similarly, the points $\alpha_{1,1}^{-1}+r(1)$ and $\alpha_{1,1}^{-1}-r(1)$ of $-\hat{I}_{1,1}$ are the end points of the half-circle $C(f^{-1}_{1,1})$.}}
 \label{Figure10}
\end{center}
\end{figure}


Now, we calculate the M\"obius transformation and its respective inverse
\begin{equation}\label{eq:19}
\begin{array}{rl}
f_{1,1}(z)     & =\dfrac{a_{1,1}z+b_{1,1}}{c_{1,1}z+d_{1,1}},\\
f_{1,1}^{-1}(z)&=\dfrac{d_{1,1}z-b_{1,1}}{-c_{1,1}z+a_{1,1}},\\
\end{array}
\end{equation}
which have as isometric circle $C(f_{1,1})$  and $C(f^{-1}_{1,1})$, respectively. Using the remark \ref{r:2.9} we have
\[
\begin{array}{ccccccl}
c_{1,1} & = & \dfrac{1}{r(1)}                 & = & 3\cdot 2^2 & = & 12    \text{,} \\
d_{1,1} & = & - c_{1,1}\cdot \alpha_{1,1}     & = & -(18+4s_0) & = &-18    \text{,} \\
a_{1,1} & = & c_{1,1} \cdot \alpha^{-1}_{1,1} & = & -(18+4s_0) & = & -18   \text{.}
\end{array}
\]
Now, we substitute these values in the equation $a_{1,1}\cdot d_{1,1}- c_{1,1}\cdot b_{1,1}=1$ and computing we hold
\[
\begin{array}{ccl}
b_{1,1} & = & \dfrac{(18+4s_0)^2 -1 }{3\cdot 2^2} =\dfrac{323}{12}.
\end{array}
\]
Thus, we know the explicit form of the M\"obius transformations of  the equations \ref{eq:19}  \emph{i.e.},
\[
\begin{array}{ccccc}
f_{1,1}(z)     & = &\dfrac{-(18 + 4s_0)z+ \dfrac{(18+4s_0)^2 -1 }{3\cdot 2^2}}{3\cdot 2^2z -(18+4s_0)} & = & \dfrac{-18z+ \dfrac{323}{12}}{12z -18},\\
f^{-1}_{1,1}(z)& = &\dfrac{-(18 + 4s_0)z- \dfrac{(18+4s_0)^2 -1 }{3\cdot 2^2}}{-3\cdot 2^2z -(18+4 s_0)}& = &\dfrac{-18z- \dfrac{323}{12}}{-12z -18}.\\
\end{array}
\]
Finally, we define the sets $J_1$ and $\mathcal{C}_1$ composed by M\"{o}bius transformations and half-circles, respectively, as such
\begin{equation}
\begin{array}{ccl}
J_1           &:=& \{f_{1,1}(z), \, f^{-1}_{1,1}(z)\}, \\
\mathcal{C}_1 &:=&\{C(f_{1,1}), \, C(f^{-1}_{1,1})\}.\\
\end{array}
\end{equation}

We remark that by construction, each M\"{o}bius transformations of $J_1$ is hyperbolic and the half-circles of $\mathcal{C}_1$ are pairwise disjoint.

\begin{remark}\label{r:3.5}
We let $\Gamma_1$ be the classical Schottky group of rank one generated by the set $J_1$. If we use the same ideas as the step 2 of the Infinite Loch Ness monster case, it is easy to check that the quotient space $S_1:=\mathbb{H} / \Gamma_1$ (see the equation \ref{eq:13}) is a hyperbolic surface homeomorphic to the cylinder. Moreover, $S_1$ has infinite hyperbolic area.
\end{remark}

\noindent \textbf{For $n$. Build a set $J_n$ containing $2^{n}$ M\"obius transformations and the set $\mathcal{C}_n$ composed by its respective isometric circles.}  We consider the closed subset $I_{n-1}=\bigcup\limits_{k=0}^{2^{n-1} -1} \left[\dfrac{3^{n-1}+s_k}{3^{n-1}}, \dfrac{3^{n-1}-s_k +1}{3^{n-1}}\right]\subseteq I_0$ and its symmetrical with respect to the imaginary axis $-I_{n-1}=\bigcup\limits_{k=0}^{2^{n-1} -1} \left[-\dfrac{3^{n-1}-s_k +1}{3^{n-1}}, - \dfrac{3^{n-1}+s_k}{3^{n-1}}\right]\subseteq -I_0$ (see the theorem \ref{T:2.1}). See the Figure \ref{Figure11}. For each $k\in \{0,\ldots,2^{n-1}-1\}$ we let $\hat{I}_{n,k},-\hat{I}_{n,k}$ be the middle third of the closed intervals

\begin{align*}
\left[ \dfrac{ 3^{n-1}+s_{k-1} }{3^{n-1}},\dfrac{3^{n-1}+s_{k-1}+1}{3^{n-1}}\right] &  & \text{ and } & &\left[-\dfrac{3^{n-1}+s_{k-1}+1}{3^{n-1}}, -\dfrac{3^{n-1}+s_{k-1}}{3^{n-1}}\right],\\
\end{align*}
respectively. By the remark \ref{r:2.2} we hold
\begin{align*}
\hat{I}_{n,k}   & =\left[\dfrac{3^{n} + s_{2k-2} +1}{3^{n}}, \dfrac{3^{n} + s_{2k-1}}{3^{n}}\right],\\
-\hat{I}_{n,k}  & =\left[-\dfrac{3^{n} + s_{2k-1}}{3^{n}}, -\dfrac{3^{n} + s_{2k-2} +1}{3^{n}} \right].\\
\end{align*}

\begin{figure}[h!]
\begin{center}
\includegraphics[scale=0.6]{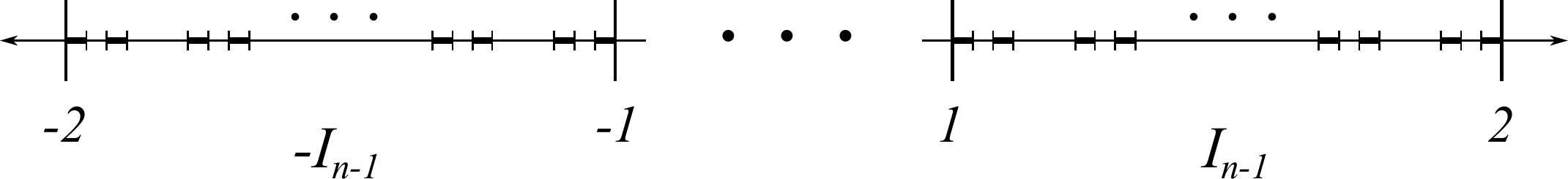}\\
\caption{\emph{ The closed subsets $I_{n-1}$ and $-I_{n-1}$}.}
 \label{Figure11}
\end{center}
\end{figure}

\begin{remark}
If we remove $\bigcup\limits_{k=0}^{2^{n-1}-1} \hat{I}_{n,k}$ and $\bigcup\limits_{k=0}^{2^{n-1}-1} -\hat{I}_{n,k}$ of the closed subsets $I_{n-1}$ and $-I_{n-1}$ accordingly, we hold the sets
\[
\begin{array}{ccccccl}
I_{n-1}\setminus\left(\bigcup\limits_{k=0}^{2^{n-1}-1} \hat{I}_{n,k}\right) &=& I_n   & =  & \bigcup\limits_{k=0}^{2^{n}-1} \left[1+\dfrac{s_k}{3^n},1+ \dfrac{s_k +1}{3^n}\right],\\
      && & =  & \bigcup\limits_{k=0}^{2^n -1} \left[\dfrac{3^n+s_k}{3^n}, \dfrac{3^n+s_k +1}{3^n}\right],\\
(-I_{n-1})\setminus\left(\bigcup\limits_{k=0}^{2^{n-1}-1}- \hat{I}_{n,k}\right) &=&-I_n  & =  & \bigcup\limits_{k=0}^{2^{n }-1} \left[-\left(1+ \dfrac{s_k +1}{3^n}\right), -\left(1+\dfrac{s_k}{3^n}\right)\right],\\
           & &      & =  & \bigcup\limits_{k=0}^{2^{n} -1} \left[- \dfrac{3^n+s_k +1}{3^n} , -\dfrac{3^n+ s_k}{3^n}\right].
\end{array}
\]
\end{remark}

On the other hand, for every $k\in\{0,\ldots,2^{n-1}-1\}$ the length of the closed intervals $\hat{I}_{n,k}$ and $-\hat{I}_{n,k}$ is $\dfrac{1}{3^n}$, then, we choose their respective middle points, which are given by
\begin{equation}\label{eq:21}
\begin{array}{ccccc}
\alpha_{n,k}      &=&1+\dfrac{1+s_{2k-1}}{3^n}+\dfrac{1}{3^n \cdot 2}&=&\dfrac{3^n\cdot 2 + 3 + 2s_{2k-1}}{3^n \cdot 2},\\ \alpha^{-1}_{n,k}&=&-1-\dfrac{1+s_{2k-1}}{3^n}-\dfrac{1}{3^n \cdot 2}&=&-\dfrac{3^n\cdot 2 + 3 + 2s_{2k-1}}{3^n \cdot 2}.\\
\end{array}
\end{equation}
%
%
By construction, the points $\alpha_{n,k}$ and $\alpha^{-1}_{n,k}$ are symmetrical with respect to the imaginary axis \emph{i.e.}, we have $\alpha^{-1}_{n,k}=-\alpha_{n,k}$. Then we let $C(f_{n,k}), C(f^{-1}_{n,k})$ be the two half-circles having centers $\alpha_{n,k}$ and $\alpha^{-1}_{n,k}$ correspondingly, and the same radius $r(n)=\dfrac{1}{3^n\cdot 2^2}$ (see the Figure \ref{Figure12}). They are given by the formula
\begin{equation}\label{eq:22}
\begin{array}{rcl}
C(f_{n,k})     &=&\{z\in \mathbb{H}:|2\cdot 3^n z-2\cdot(3^n\cdot 2 + 3 + 2s_{2k-1})|=1\},\\
C(f^{-1}_{n,k})&=&\{z\in \mathbb{H}:|-2\cdot 3^nz -2\cdot(3^n\cdot 2 + 3 + 2s_{2k-1})|=1\}.
\end{array}
\end{equation}

\begin{figure}[h!]
\begin{center}
\includegraphics[scale=0.5]{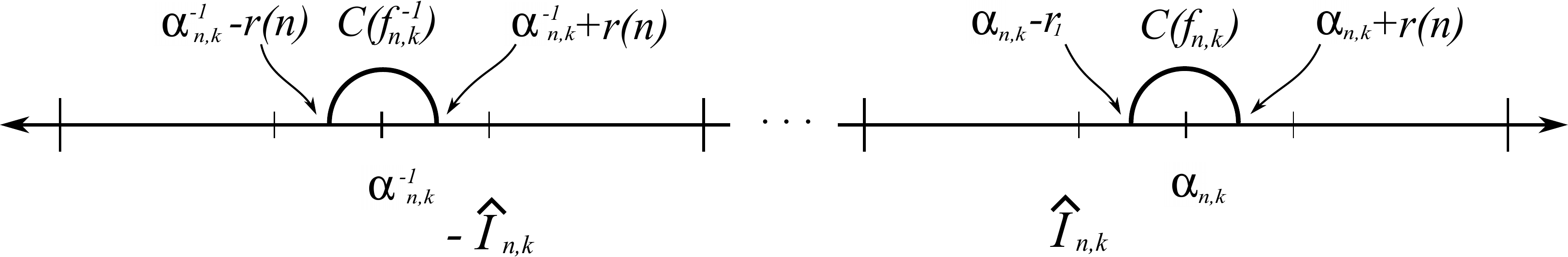}\\
\caption{\emph{Half circles $C(f_{n,k})$ and  $C(f^{-1}_{n,k})$  having centers $\alpha_{n,k}$ and $\alpha^{-1}_{n,k}$ respectively with  radius $r(n)=\frac{1}{3^n\cdot 2^2}$. For each $k\in\{0,\ldots, 2^{n-1}-1\}$, the points $\alpha_{n,k}+r(n)$ and $\alpha_{n,k}-r(n)$ of $\hat{I}_{n,k}$ are the end points of the half-circle $C(f_{n,k})$. Equally, the points $\alpha_{n,k}^{-1}+r(n)$ and $\alpha_{n,k}^{-1}-r(n)$ of $-\hat{I}_{n,k}$ are the end points of the half-circle $C(f^{-1}_{n,k})$.}}
 \label{Figure12}
\end{center}
\end{figure}


Now, for each $k\in \{0,\ldots,2^{n-1}-1\}$ we must calculate the M\"obius transformation and its respective inverse
\begin{equation}\label{eq:23}
\begin{array}{rcl}
f_{n,k}(z) & = & \dfrac{a_{n,k}z+b_{n,k}}{c_{n,k}z+d_{n,k}},\\
f^{-1}_{n,k}(z) &= &\dfrac{d_{n,k}z-b_{n,k}}{-c_{n,k}z+a_{n,k}},
\end{array}
\end{equation}

which have as isometric circles $C(f_{n,k})$ and $C(f^{-1}_{n,k})$, respectively.
%
%
Using the remark \ref{r:2.9} we have
\[
\begin{array}{ccccl}
c_{n,k} & = & \dfrac{1}{r(n)} &=&3^n\cdot 2^2 ,              \\
d_{n,k} & = & - c_{n,k}\cdot \alpha_{n,k}& = &-2\cdot(3^n\cdot 2 + 3 + 2s_{2k-1}), \\
a_{n,k} & = & c_{n,k} \cdot \alpha^{-1}_{n,k}&=& -2\cdot(3^n\cdot 2 + 3 + 2s_{2k-1}).\\
\end{array}
\]
Now, we substitute these values in the equation $a_{n,k}\cdot d_{n,k}- c_{n,k}\cdot b_{n,k}=1$ and computing we hold
\[
\begin{array}{ccl}
b_{n,k} & = & \dfrac{2^2\cdot(3^n\cdot 2 + 3 + 2s_{2k-1})^2 -1 }{3^n \cdot 2^2}.
\end{array}
\]
Thus, for each $k\in\{0,\ldots, 2^{n-1}-1\}$ we know the explicit  form of the M\"obius transformations of equations \ref{eq:23}  \emph{i.e.},
\[
\begin{array}{rcl}
f_{n,k}(z) & = &\dfrac{-2\cdot(3^n\cdot 2 + 3 + 2s_{2k-1})z+\dfrac{2^2\cdot(3^n\cdot 2 + 3 + 2s_{2k-1})^2 -1 }{3^n \cdot 2^2}}{3^n\cdot 2^2z-2\cdot(3^n\cdot 2 + 3 + 2s_{2k-1})},\\
f^{-1}_{n,k}(z) & =&\dfrac{-2\cdot(3^n\cdot 2 + 3 + 2s_{2k-1})z-\dfrac{2^2\cdot (3^n\cdot 2 + 3 + 2s_{2k-1})^2 -1 }{3^n \cdot 2^2}}{-3^n\cdot 2^2 z-2\cdot(3^n\cdot 2 + 3 + 2s_{2k-1})}.\\
\end{array}
\]
Finally, we define the sets $J_n$ and $\mathcal{C}_n$ composed by M\"{o}bius transformations and half-circles, respectively, just as:
\begin{equation}\label{eq:24}
\begin{array}{rcl}
J_n & := & \{f_{n,k}(z), \, f^{-1}_{n,k}(z): k\in\{0,\ldots,2^{n-1}-1\}\},\\
\mathcal{C}_n & := &\{C(f_{n,k}), \, C(f^{-1}_{n,k}): k\in\{0,\ldots,2^{n-1}-1\}\}.\\
\end{array}
\end{equation}

We note that by construction each M\"{o}bius transformation of $J_n$ is hyperbolic and the half-circles of $\mathcal{C}_n$ are pairwise disjoint.

\begin{remark}\label{r:3.7}
If we consider the classical Schottky group of rank one $\Gamma_n$, which is generated by the set $J_n$. Using the same ideas as the Infinite Loch Ness monster case, it is easy to check that the quotient space $S_n:= \mathbb{H}/ \Gamma_n$ (see the equation \ref{eq:13}) is a hyperbolic surface having infinite area and homeomorphic to the sphere punctured by $2n$ points.
\end{remark}

From the previous recursive construction of M\"{o}bius transformations and half-circles (see the equation \ref{eq:24}), we define the sets
\begin{equation}\label{eq:25}
\begin{array}{rcl}
J & := & \bigcup\limits_{n\in \N} J_n,\\
\mathcal{C} & := & \bigcup\limits_{n\in N} \mathcal{C}_n,\\
\end{array}
\end{equation}
and we denote $\Gamma$ as the subgroup of $PSL(2,\mathbb{R})$ generated by the set $J$. We observe that by construction each M\"{o}bius transformation of $J$ is hyperbolic and the half-circles of $\mathcal{C}$ are pairwise disjoint.

\vspace{2mm}
\noindent \textbf{Step 2. The group $\Gamma$ is a Fuchsian group.} In order to $\Gamma$ will be a Geometric Schottky group, we shall define a Schottky description for it. Hence, by the proposition \ref{p:2.16} we will conclude that $\Gamma$ is Fuchsian.

We note that the elements belonging to the family $J$ in the equations \ref{eq:25} can be indexed by a symmetric subset of $\mathbb{Z}$. Merely, we consider $P:=\{p_n\}_{n\in\mathbb{N}}$ the subset of all the prime numbers. Then it is easy to check that the map $\psi: J\to \mathbb{Z}$ defined by
\[
f_{n,k}(z)\mapsto p_n^k, \quad f_{n,k}^{-1}(z)\mapsto -p_n^k, \quad \text{(for every $n\in\mathbb{N}$ and $k\in\{0,\ldots,2^{n-1}-1\}$),}
\]
is well-defined and injective. The image of $J$ under $\psi$ is a symmetric subset of $\mathbb{Z}$, which we denote as $I$. Given that for each element  $k$ belonged to $I$ there is a unique transformation $f\in J$ such that $\psi(f)=k$, then we label the map $f$ as $f_k$ and its respective isometric circle as $C(f_k)$. Hence, we re-write the sets $J$ and $\mathcal{C}$ as follows
\[
\begin{array}{rcl}
J & := &\left\{f_i(z)\right\}_{i\in I},\\
\mathcal{C} & := &\{C(f_i)\}_{i\in I}.
\end{array}
\]



On the other hand, we define the set $\{A_i\}_{i\in I}$ where $A_i$ is the straight segment in the real line $\mathbb{R}$ whose ends points coincide to the ends points at the infinite of the half-circle $C(f_{i})$ (see the equation \ref{eq:25}). In other words, it is the straight segment joining the ends points of $C(f_i)$ the isometric circle of $f_i(z)$.

We claim that the pair $(\{A_i\}, \{f_i\})_{i\in I}$ is a Schottky description. By the inductive construction  of the family $J:=\{f_i\}_{i\in I}$ described above, it is immediate that the pair $(\{A_i\},\{f_i\})_{i\in I}$ satisfies the conditions 1 to 4 of definition \ref{d:2.14}. Thus, we must only prove that the condition 5 of definition \ref{d:2.14} is done.

For any $i\in I$, we consider the transformation $f_i$ and its respective isometric circle $C(f_i)$ . Now, we denote as $\alpha_i$ and $r_i$ the center and radius of $C(f_i)$. Hence, we define the open strip associated to $f_i$ as follows (see the Figure \ref{Figure5})
\[
S_i:=\{z\in\mathbb{H}\, :\,  \alpha_{i}-2r_{i}\, < \, Re(z)\, < \, \alpha_{i}+2r_{i}\, \}.
\]
We note that $C(f_i)\subset S_i$ and for any two different transformations $f_i\neq f_j\in J$ its respective open strips associated $S_i$ and $S_{j}$ are disjoint. Further, by construction of the family  $\mathcal{C}$ (see the equation \ref{eq:25}) there exist a transformation $f_k\in J$ such that $r_k$, the radius of its isometric circle $C(f_k)$ satisfies  $r_i\leqslant r_k$ for all $i\in I$.




Moreover, by the remark \ref{r:2.10} it follows that $B_{\epsilon}^k$ the closed hyperbolic $\epsilon$-neighborhood of the half-circle $C(f_k)$  belongs to the open strip $S_k$, choosing $\epsilon$ less than $\dfrac{r_k}{2}$. Since, for all $i\in I$ we have $| \alpha_k -\alpha_i |> (r_k + r_i)$,  by the lemma \ref{l:2.11} for every $\epsilon < \dfrac{\max\{r_k,r_i\}}{2}=\dfrac{r_k}{2}$ the closed hyperbolic $\epsilon$-neighborhoods $B^k_{\epsilon}$ and $B^i_{\epsilon}$ of the half circle $C(f_k)$ and $C(f_i)$ are disjoint. Even more, this lemma assures that $B^i_{\epsilon}\subseteq S_i$, for all $i\in I$. This implies that condition 5 is done. Therefore, we conclude that the pair $(\{A_i\}, \{f_i\})_{i\in I}$ is a Schottky description.






\vspace{2mm}
\noindent \textbf{Step 3. Holding the surface called the Cantor tree.} The Geometric Schottky group $\Gamma$ acts freely and properly discontinuously on the open subset $\mathbb{H}-K$, where the subset $K\subset \mathbb{H}$ is defined as in the equation \ref{eq:12}. In this case, we note that $K$ is empty because the intersection of any two different elements of $\mathcal{C}$ is empty. Then the quotient space
\begin{equation}\label{eq.26}
S:= \mathbb{H}/\Gamma
\end{equation}
is well-defined and through the projection map
\begin{equation}\label{eq:27}
\pi:\mathbb{H}\to S, \quad z\mapsto[z],
\end{equation}
it is a hyperbolic surface. To end the proof we shall show that $S$ is the Cantor tree \emph{i.e.}, its ends space is homeomorphic to the Cantor set having all ends planar. In other words, $S$ is homeomorphic to the Cantor tree. To prove this, we will describe the ends space of $S$ using the property of $\sigma$-compact of $S$ and showing that there is a homeomorphism $F$ from the ends spaces of the Cantor binary tree $Ends(T2^\omega)$ onto  the ends space $Ends(S)$. 

The following remark is necessary.
\begin{remark}\label{r:3.8}
We let $F(\Gamma)$ be the standard fundamental domain of  $\Gamma$, as such
\begin{equation}\label{eq:28}
F(\Gamma):=\bigcap_{i\in I} \overline{\hat{C}(f_i)}\subset \mathbb{H}.
\end{equation}
Regarding the proposition \ref{p:2:17} it is a fundamental domain for $\Gamma$ having the  following properties.
\begin{enumerate}
\item It is connected and locally finite having infinite hyperbolic area. Further, its boundary is composed by the family of half-circle $\mathcal{C}$ (see the equation \ref{eq:25}). In other words, it consists of infinitely many hyperbolic geodesic with ends points at infinite and mutually disjoint.

\item It is  a non-compact Dirichlet region and the quotient $F(\Gamma)/\Gamma$ is homeomorphic to $S$, then the quotient space $S$ is also a non-compact hyperbolic surface with infinite hyperbolic area (see \cite[Theorem 14.3 p. 283]{KS2}). We note that by construction, the surface $S$ does not have genus \emph{i.e.}, its ends are planar.
\end{enumerate}
\end{remark}

Since surfaces are $\sigma$-compact space, for $S$ there is an exhaustion of $S=\bigcup_{n\in\mathbb{N}} K_n$ by compact sets whose complements define the  ends spaces of the surface $S$. More precisely;

\vspace{1.5mm}
\noindent \textbf{For $n=1$.} We consider the radius $r(1)=\dfrac{1}{3\cdot 2^2}$ given in the recursive construction of $\Gamma$ and define the compact subset $\tilde{K}_1$ of the hyperbolic plane $ \mathbb{H}$ as follows
$$
\tilde{K}_1:=\{ z\in \mathbb{H}: -2 \leq Re(z) \leq 2, \, \text{ and } \, r(1) \leq Im (z) \leq 1\}.
$$
The image of the intersection $\tilde{K}_1\cap F(\Gamma)$ under the projection map $\pi$ described in the equation \ref{eq:27}
$$
\pi ({\tilde{K}_1 \cap F(\Gamma)}):=K_1\subset S,
$$
is a compact subset of $S$. By definition of $\tilde{K}_1$ the difference $S\setminus K_1$ consists of two connected components whose closure in $S$ is non-compact, and they have compact boundary. Hence, we can write
\[
S\setminus K_1:=U_0 \sqcup U_1.
\]

\begin{remark}\label{r:3.9}
The set of connected components of $S\setminus K_1$ and the set defined as $2^1:=\{0,1\}$ are equipotent. In other words, they are in one-to-one relation.
\end{remark}

\vspace{1mm}
\noindent \textbf{For $n=2$.} We consider the radius $r(2)=\dfrac{1}{3^2 \cdot2^2}$ given in the recursive construction of $\Gamma$ and define the compact subset $\tilde{K}_2$ of the hyperbolic plane $\mathbb{H}$ as follows
$$
\tilde{K}_2:=\{ z\in \mathbb{H}: -3 \leq Re(z) \leq 3, \, \text{ and } \, r(2)\leq Im(z) \leq 2\}.
$$
By construction $\tilde{K}_1\subset \tilde{K}_2$ and the image of the intersection $\tilde{K}_2\cap F(\Gamma)$ under the projection map $\pi$ (see the equation \ref{eq:27})
$$
\pi ({\tilde{K}_2 \cap F(\Gamma)}):=K_2\subset S,
$$
is a compact subset of $S$ just as $K_1\subset K_2$. By definition of $\tilde{K}_2$ the difference $S\setminus K_2$ consists of $2^2$ connected components whose closure in $S$ is non-compact, and they have compact boundary. Moreover, for every $l$ in the set $2^1$ there exist exactly two connected components of $S\setminus K_2$ contained in $U_l\subset S\setminus K_1$. Hence, we can write
\[
S\setminus K_2:=U_{0,\, 0}\sqcup U_{0,\, 1} \sqcup U_{1,\, 0}\sqcup U_{1,\, 1}=\bigsqcup\limits_{l\in 2^1 } (U_{l,\, 0}\sqcup U_{l,\, 1}),
\]
being  $U_{l, \, 0}, U_{l, \,1}\subset U_l$ for every $l\in 2^1$.

\begin{remark}\label{r:3.10}
The set of connected components of $S\setminus K_2$ and the set defined as $2^2:=\prod\limits_{i=1}^{2}\{0,1\}_i$ are equipotent. In other words, they are in one-to-one relation.
\end{remark}

\vspace{1mm}
Following the recursive construction above, for $n$ we consider the radius $r(n)=\dfrac{1}{3^n \cdot2^2}$ given in the recursive construction of $\Gamma$ and define the compact subset $\tilde{K}_n$ of the hyperbolic plane $\mathbb{H}$ as follows
$$\tilde{K}_n:=\{ z\in \mathbb{H}: -(n+1) \leq Re(z) \leq n+1, \, \text{ and } \, r(n)\leq Im(z) \leq n\}.$$
By construction $\tilde{K}_{n-1}\subset \tilde{K}_n$ and the image of the intersection $\tilde{K}_n\cap F(\Gamma)$ under the projection map $\pi$ (see the equation \ref{eq:27})
$$\pi ({\tilde{K}_n \cap F(\Gamma)}):=K_n\subset S,$$
is a compact subset of $S$ just as $K_{n-1}\subset K_n$. By definition of $\tilde{K}_n$ the difference $S\setminus K_n$ consists of $2^n$ connected components whose closure in $S$ is non-compact, and they have compact boundary. Moreover, for every $l$ in the set $2^{n-1}$ there exist exactly two connected components of $S\setminus K_{n-1}$ contained in $U_l \subset S\setminus K_{n-1}$. Hence, we can write
\[
S\setminus K_n:= \bigsqcup_{l\in 2^{n-1}} (U_{l,\,0} \sqcup U_{l,\, 1}),
\]
being $U_{l,\, 0}, U_{l,\, 1}\subset U_l$ for every $l\in 2^{n-1}$.

\begin{remark}\label{r:3.11}
The set of connected components of $S\setminus K_n$ and the set defined as $2^n:=\prod\limits_{i=1}^{n}\{0,1\}_i$ are equipotent. In other words, they are in one-to-one relation.
\end{remark}

This recursive construction induces the desired numerable family of increasing compact subset $\{K_n\}_{n\in\mathbb{N}}$ covering the surface $S$, 
$$
S=\bigcup_{n\in\mathbb{N}} K_n.
$$
Thus, the ends space of $S$ is composed by all sequences $(U_{l_n})_{n\in \N}$ such that $U_{l_n} \subset S\setminus K_n$ and $U_{l_n}\supset U_{l_{n+1}}$, for each $n\in\mathbb{N}$. Further, $l_n\in 2^n$, $l_{n+1}\in 2^{n+1}$ such that $\pi_{i}(l_n)=\pi_{i}(l_{n+1})$ for every $i\in\{1,\ldots,n\}$, (see subsection Cantor binary tree).

It is easy  to check that the  map
\[
\begin{array}{ccc}
f :  Ends(T2^\omega) \to  Ends(S), &  &          (v_i)_{i\in\mathbb{N}}  \mapsto  [U_{v_i}]_{i\in \N},
\end{array}
\]
is well-defined. We now prove that is a closed continuous bijection, hence a homeomorphism.

\vspace{1.5mm}
\textbf{Injectivity.} Consider two different infinite sequences $(v_{n})_{n\in\mathbb{N}}\neq (v_{n\in\mathbb{N}}^{'})$ of the graph $T2^{\omega}$. Then there exist $N\in\mathbb{N}$ such that for all $m> N $ we have $v_m \neq v^{'}_m$. This implies that $U_{v_m}\cap U_{v_m^{'}}=\emptyset$, hence $f((v_n)_{n\in\mathbb{N}})=[U_{v_n}]_{n\in\mathbb{N}}\neq [U_{v_n^{'}}]_{n\in\mathbb{N}}=f((v_n^{'})_{n\in\mathbb{N}})$.

\vspace{1.5mm}
\textbf{Surjectivity.} Consider an end $[W_n]_{n\in\mathbb{N}}$ of $Ends(S)$, then for each $l \in \mathbb{N}$ there is a $W_{n_l}$ of the sequence $(W_n)_{n\in\mathbb{N}}$ such that $K_l \cap W_{n_l}=\emptyset$. Hence, $W_{n_l}$ is contained to $U_{s_l}$ a connected component of $S\setminus K_l$ such that $s_l\in 2^l$. Given that $W_{n_l}$ is a connected open such that $W_{n_{l-1}}\supseteq W_{n_l}$ then $U_{s_{l-1}}\supset U_{s_l}$ for every $l\in\mathbb{N}$ \emph{i.e.}, the sequences $(W_n)_{n\in\mathbb{N}}$ and $(U_{s_l})_{l\in\mathbb{N}}$ define the same end in $S$ \emph{i.e.}, $[W_n]_{n\in\mathbb{N}}=[U_{s_l}]_{l\in\mathbb{N}}$. This implies that the sequence $(s_{l})_{l\in\mathbb{N}}$ is an infinite path in $T2^{\omega}$ defining an end of $T2^{\omega}$ and $f$ sends the sequence $(s_{l})_{l\in\mathbb{N}}$ into the ends $[W_n]_{n\in\mathbb{N}}=[U_{s_l}]_{l\in\mathbb{N}}$.

\vspace{1.5mm}
\textbf{The map $f$ is continuous.} Given an end $(v_n)_{n\in \mathbb{N}}\in Ends(T2^{\omega})$ and any open subset $W\subset S$ whose boundary is compact such that the basic open $W^{\ast}$ contains the point $f((v_n)_{n\in\mathbb{N}})=[U_{v_n}]_{n\in\mathbb{N}}$, we must prove that there is an open subset $Z\subset T2^{\omega}$ such that the basic open $Z^{\ast}$ contains the point $(v_n)_{n\in\mathbb{N}}$ and $f(Z^{\ast})\subset W^{\ast}$. Hence $[U_{v_n}]_{n\in\mathbb{N}}\in W^{\ast}$ and there is $m\in\mathbb{N}$ such that $U_{v_m}\subset W$. If we remove the vertex $v_m$ of the graph $T2^{\omega}$ the subset  $T2^{\omega}\setminus\{v_m\}$ consists of exactly three open connected components whose closure in $T2^{\omega}$ are non-compact, but has compact boundary. We denote as $Z$ the connected component of $T2^{\omega}$ having the infinite sequence $(v_{m+i})_{i\in\mathbb{N}}$. Then the subset $Z^{\ast}\subset Ends(T2^{\omega})$ is the required basic open, which contains the points $(v_n)_{n\in\mathbb{N}}$ and $f(Z^{\ast})\subset W^{\ast}$.

The map $f$ is closed because it is a continuous bijective  from a compact space into a Hausdorff space (see \cite[Theorem 2.1 p.226]{Dugu}. Hence $f$ is a homeomorphism.

\qed

\begin{corollary}\label{c:3.10}
For every $n\in\mathbb{N}$, there is a classical Schottky subgroup $\Gamma_n$ of $\Gamma$ having rank $n$, such that the quotient space $S_n:=\mathbb{H}/\Gamma_n$ is a hyperbolic surface homeomorphic to the sphere punctured by $2n$ points. The fundamental group of $S_n$ is isomorphic to $\Gamma_n$. Furthermore, $S$ has infinite hyperbolic area.
\end{corollary}

Indeed, we consider $\Gamma_n$ as the Fuchsian group generated by the set $J_n$ (see the equation \ref{eq:24}) and proceed verbatim.

On the other hand, from Theorem \ref{T:uni} we can conclude that:

\begin{corollary}
The fundamental group of the Cantor tree is isomorphic to $\Gamma$.
\end{corollary}


\subsection{Proof Theorem \ref{T:0.3}}\label{section3.3}

\hfill \break

 \textbf{Step 1. Building the group $\Gamma$.} 

\vspace{2mm}
\textbf{For $n=1$. Building the set $J_1$ containing infinitely countable M\"obius transformation and the set $\mathcal{C}_1$ composed by its respective isometric circles.} We consider the closed interval $I_0=[1,2]$ and its symmetric with respect to the imaginary axis $-I_0=[-2,-1]$ (see the theorem \ref{T:2.1}). By the remark \ref{r:2.2} we have that $\hat{I}_{1,1}=\left[\dfrac{4+s_0}{3}, \dfrac{3+ s_1}{3}\right]$ and $-\hat{I}_{1,1}=\left[-\dfrac{3+ s_1}{3}, -\dfrac{4+s_0}{3} \right]$ are the middle thirds of the closed intervals $I_0$ and $-I_0$, respectively. Given that $s_0=0$ and $s_1=2$ then we have $\hat{I}_{1,1}=\left[\dfrac{4}{3}, \dfrac{5}{3}\right]$ and $-\hat{I}_{1,1}=\left[ -\dfrac{5}{3}, -\dfrac{4}{3}\right]$. See the Figure \ref{Figure13}.
\begin{figure}[h!]
\begin{center}
\includegraphics[scale=0.5]{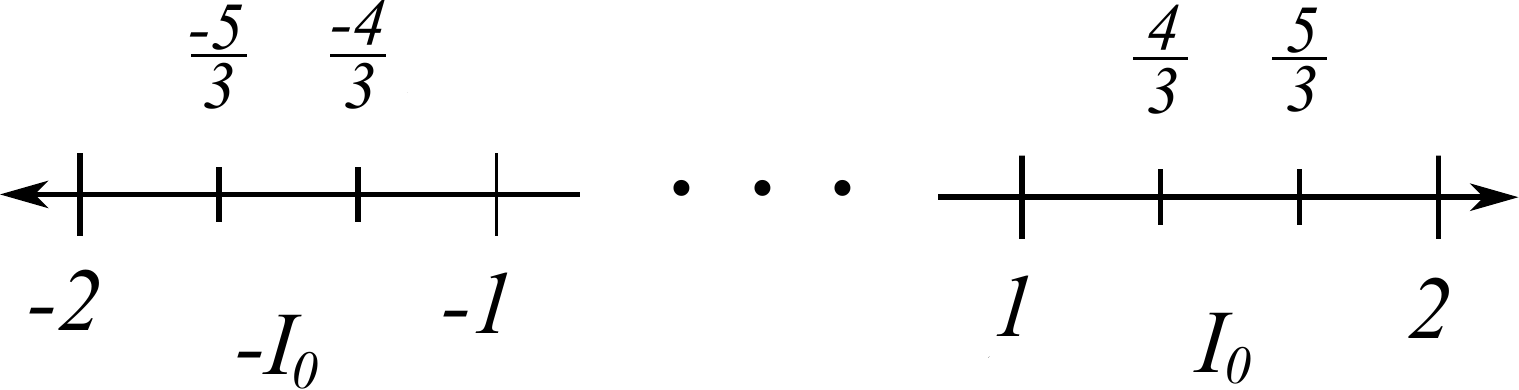}\\
\caption{\emph{The middle thirds of the closed intervals $[1,2]$ and $[-2,-1]$}.}
 \label{Figure13}
\end{center}
\end{figure}

We remark that the length of the closed intervals $\hat{I}_{1,1}$ and $-\hat{I}_{1,1}$ is $\dfrac{1}{3}$ and their respective middle points are $\alpha_{1,1}$ and $\alpha^{-1}_{1,1}$ as in the equation \ref{eq:17}.
Then we let $C(f_{1,1}), C(f_{1,1}^{-1})$ be two half-circles having centers $\alpha_{1,1}$ and $\alpha_{1,1}^{-1}$ respectively, and the same radius $r(1)=\dfrac{1}{6\cdot 3 \cdot 2}$ (see the Figure \ref{Figure14}). They are given by the formula
\begin{equation}\label{eq:29}
\begin{array}{rl}
C(f_{1,1})      & =\left\{z\in \mathbb{H}: \left| 6\cdot 3 \cdot 2 z -6\cdot(9+2 s_0)\right|^{-2}=1\right\},\\
                & =\left\{z\in \mathbb{H}: \left| 36 z - 54 \right|^{-2}=1\right\},\\
C(f_{1,1}^{-1}) & =\left\{z\in \mathbb{H}: \left| - 6\cdot 3 \cdot 2 z-6\cdot(9+2 s_0)\right|^{-2}=1\right\},\\
                & =\left\{z\in \mathbb{H}: \left| -36 z - 54\right|^{-2}=1\right\}.\\
\end{array}
\end{equation}

We note that the half-circles given by the equations \ref{eq:22} and \ref{eq:29} are different because their radius are distinct. Moreover, by construction the centers of the half-circles $C(f_{1,1})$ and $C(f_{1,1}^{-1})$ are symmetrical with respect to the imaginary axis, \emph{i.e.}, we have $-\alpha_{1,1}=\alpha^{-1}_{1,1}$. The points $\alpha_{1,1}+r(1)$ and $\alpha_{1,1}-r(1)$ of $\hat{I}_{1,1}$ are the two end points at infinite of $C(f_{1,1})$. Similarly, the points $\alpha_{1,1}^{-1}+r(1)$ and $\alpha_{1,1}^{-1}-r(1)$ of $-\hat{I}_{1,1}$ are the two end points at infinite of $C(f_{1,1}^{-1})$.

\begin{remark}
As the isometric circles are not necessary to give the explicit expression of the group $\Gamma$, from now on we will avoid writing the equations of $C(f_{i,j})$ for the rest of the proof.
\end{remark}

Now, we calculate the M\"obius transformation and its respective inverse
\begin{equation}\label{eq:30}
\begin{array}{ccl}
f_{1,1}(z)      &=&\dfrac{a_{1,1}z+b_{1,1}}{c_{1,1}z+d_{1,1}},\\
f_{1,1}^{-1}(z) &=&\dfrac{d_{1,1}z-b_{1,1}}{-c_{1,1}z+a_{1,1}},
\end{array}
\end{equation}
which has as isometric circles $C(f_{1,1})$  and $C(f^{-1}_{1,1})$, respectively. By the remarks \ref{r:2.9} we have
\[
\begin{array}{cclcl}
a_{1,1} & = &-6\cdot(9 + 2 s_0) &= &-54,\\
c_{1,1} & = & 6\cdot 3\cdot 2 &= &36,\\
d_{1,1} & = &-6\cdot(9+2 s_0) &= &-54.\\
\end{array}
\]
Now, we substitute these values in the equation $a_{1,1}\cdot d_{1,1}-b_{1,1}\cdot c_{1,1}=1$ and computing we hold
\[
b_{1,1} = \dfrac{6^2\cdot(9+2 s_0)^2 -1 }{6\cdot 3\cdot 2}= \dfrac{2915}{36}.
\]
Thus, we know the explicit form of the M\"obius transformations of the equation \ref{eq:30}
\[
\begin{array}{cclcl}
f_{1,1}(z)     &=&\dfrac{-6\cdot(9 + 2 s_0)z+ \dfrac{6^2\cdot(9+2 s_0)^2 -1 }{6\cdot3\cdot 2}}{6\cdot 3\cdot 2z -6\cdot(9+2 s_0)}&=&\dfrac{-54z+\dfrac{2915}{36}}{36z-54},\\
f^{-1}_{1,1}(z)&=&\dfrac{-6\cdot(9 + 2 s_0)z- \dfrac{6^2\cdot(9+2 s_0)^2 -1 }{6\cdot 3\cdot 2}}{-6\cdot 3\cdot 2z -6\cdot(9+2 s_0)}&=&\dfrac{-54z-\dfrac{2915}{36}}{-36z-54}.
\end{array}
\]

Then we define the sets $J_{1,1}$ and $C_{1,1}$ composed by M\"{o}bius transformations and half-circles, as such
\begin{equation}\label{eq:31}
\begin{array}{ccl}
J_{1,1}           & :=& \{f_{1,1}(z), \, f^{-1}_{1,1}(z)\}, \\
\mathcal{C}_{1,1} & := &\{C(f_{1,1}), \, C(f^{-1}_{1,1})\}.\\
\end{array}
\end{equation}

We remark that by construction each M\"{o}bius transformation of $J_{1,1}$ is hyperbolic and the half-circle of $C_{1,1}$ are pairwise disjoint. We let $\Gamma_{1,1}$ be the classical  Schottky group of rank one generated by $J_{1,1}$. Using the same ideas as the Infinite Loch Ness monster case, it is easy to check that the quotient space $\mathbb{H}/\Gamma_{1,1}$ is a hyperbolic surface homeomorphic to the cylinder.

\begin{remark}
Now, we shall build eight sequences of half-circles whose radius converge to zero, each sequence will have associated a suitable sequence of M\"{o}bius transformations. Two of those sequences will be in the second sixth closed subinterval of $\hat{I}_{1,1}$ and two more in the fifth sixth closed subinterval of $-\hat{I}_{1,1}$. Likewise, two sequences will be in the fifth sixth closed subinterval of $\hat{I}_{1,1}$ and the last two in the second sixth closed subinterval of $-\hat{I}_{1,1}$. In other words, it will be sequences of half-circles converged to zero at the left and at the right of $C(f_{1,1})$ and $C(f^{-1}_{1,1})$.
\end{remark}

Now, we shall build sequences of half-circles at the left and at the right of $C(f_{1,1})$, $C(f^{-1}_{1,1})$, whose radius converge to zero. Each sequence will have associated a suitable sequence of M\"{o}bius transformations as follows.

\textbf{Part I. Building sequences of half-circles at the left of $C(f_{1,1})$ and at right of $C(f^{-1}_{1,1})$.} We divide the closed interval $\hat{I}_{1,1}$ into six as shown in Figure \ref{Figure14}. Then we consider the second sixth closed subinterval of its, which is given by
\begin{figure}[h!]
\begin{center}
\includegraphics[scale=0.6]{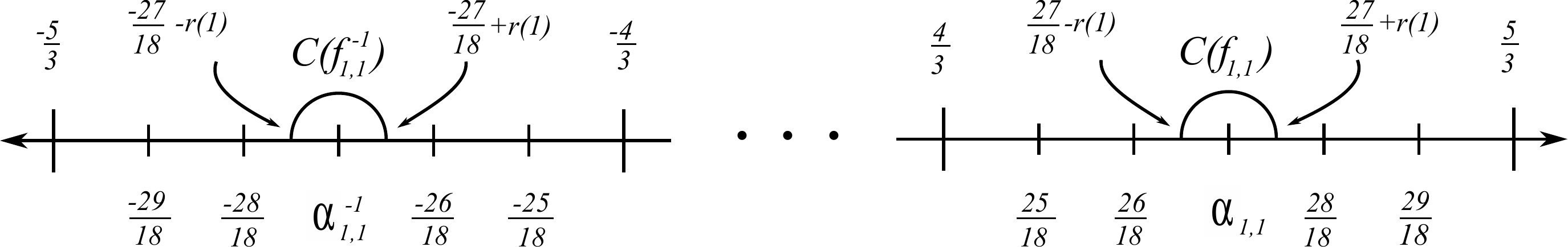}\\
\caption{\emph{The closed intervals $\hat{I}_{1,1}$ and $-\hat{I}_{1,1}$ are divided into six closed subintervals. For  $\hat{I}_{1,1}$ the second sixth subinterval is  $\left[\frac{25}{18}, \frac{26}{18}\right]$ denoted as $L_{1,1}$ and for  $-\hat{I}_{1,1}$ the fifth sixth subinterval is  $\left[-\frac{26}{18}, -\frac{25}{18}\right]$ denoted as $-L_{1,1}$.  }}
 \label{Figure14}
\end{center}
\end{figure}
\[
L_{1,1}:=\left[\dfrac{4+s_0}{3} +\dfrac{1}{3\cdot 6}, \dfrac{4+s_0}{3} +\dfrac{2}{3\cdot 6}\right] = \left[\dfrac{25 +6s_0}{3\cdot 6}, \dfrac{26 + 6s_0}{3\cdot 6}\right] = \left[\dfrac{25}{18}, \dfrac{26}{18}\right].
\]
We denote as $-L_{1,1}$ the closed interval symmetric with respect to the imaginary axis of $L$, \emph{i.e.},
\[
-L_{1,1} = \left[-\dfrac{26 + 6s_0}{3\cdot 6}, -\dfrac{25 +6s_0}{3\cdot 6}\right] = \left[-\dfrac{26}{18},-\dfrac{25}{18} \right].
\]

Then we write $L_{1,1}$ and $-L_{1,1}$ as an union of closed subintervals $\bigcup\limits_{m\in\mathbb{N}}\Bigl(L_{1,1}\Bigr)_m$ and $\bigcup\limits_{m\in\mathbb{N}}\Bigl(-L_{1,1}\Bigl)_m$,  which are given respectively as
\[
\begin{array}{ccl}
  \Bigl(L_{1,1}\Bigr)_m & = &\left[\dfrac{25+ 6s_0}{3\cdot 6}+\dfrac{1}{3\cdot 6 \cdot 2^m},\dfrac{25 +6s_0}{3\cdot 6}+\dfrac{1}{3\cdot 6 \cdot 2^{m-1}} \right], \\
      &= &\left[ \dfrac{(25+6s_0)\cdot 2^m +1}{3\cdot 6\cdot 2^m}, \dfrac{(25+6s_0)\cdot 2^{m-1} +1}{3\cdot 6\cdot 2^{m-1}}\right],\\
      & =& \left[\dfrac{25\cdot 2^m +1}{18\cdot 2^m}, \dfrac{25\cdot 2^{m-1}+1}{18\cdot 2^{m-1}} \right],\\
\Bigl(-L_{1,1}\Bigr)_m&=& \left[- \dfrac{(25+6s_0)\cdot 2^{m-1} +1}{3\cdot 6\cdot 2^{m-1}},- \dfrac{(25+6s_0)\cdot 2^m +1}{3\cdot 6\cdot 2^m} \right],\\
      & =& \left[- \dfrac{25\cdot 2^{m-1}+1}{18\cdot 2^{m-1}},-\dfrac{25\cdot 2^m +1}{18\cdot 2^m} \right].\\
\end{array}
\]
By definition, the length of $\Bigl(L_{1,1}\Bigr)_m$ and $\Bigl(-L_{1,1}\Bigr)_m$ is $\dfrac{1}{3\cdot 6\cdot 2^{m}}$ for each $m\in\mathbb{N}$, then  we choose the points $\Bigl(\alpha_{1,1}\Bigr)_{1,m}, \Bigl(\alpha_{1,1}\Bigr)_{2,m}$ belonged to $\Bigl(L_{1,1}\Bigr)_m$, as such
\begin{equation}\label{eq:32}
\begin{array}{cll}
\Bigl(\alpha_{1,1}\Bigr)_{1, m}& =\dfrac{(25+6s_0)\cdot 2^m +1}{3\cdot 6\cdot 2^m} +\dfrac{7}{10}\left( \dfrac{1}{3\cdot 6\cdot 2^{m}}\right) & =\dfrac{(250+60 s_0)\cdot 2^{m} +17}{10\cdot 3\cdot 6\cdot 2^{m}},\\
           &=\dfrac{250\cdot 2^{m} +17}{180 \cdot 2^{m}},&\\
\Bigl(\alpha_{1,1}\Bigr)_{2, m}& = \dfrac{(25+6s_0)\cdot 2^m +1}{3\cdot 6\cdot 2^m} +\dfrac{3}{10}\left( \dfrac{1}{3\cdot 6 \cdot 2^{m}}\right)&=\dfrac{(250+60s_0)\cdot 2^{m} +13}{10\cdot 3\cdot 6\cdot 2^{m}},\\
&=\dfrac{250 \cdot 2^{m} +13}{180 \cdot 2^{m}},&\\
\end{array}
\end{equation}
and the points $\Bigl(\alpha_{1,1}\Bigr)^{-1}_{1m}, \Bigl(\alpha_{1,1}\Bigr)^{-1}_{2,m}$ belonged to $\Bigl(-L_{1,1}\Bigr)_m$, as a result
\begin{equation}\label{eq:33}
\begin{array}{cl}
\Bigl(\alpha_{1,1}\Bigr)_{1, m}^{-1} &=-\dfrac{(250+60s_0)\cdot 2^{m} +13}{10\cdot 3\cdot 6\cdot 2^{m}},\\
                   &=-\dfrac{250 \cdot 2^{m} +13}{180 \cdot 2^{m}},\\
\Bigl(\alpha_{1,1}\Bigr)_{2, m}^{-1} & =-\dfrac{(250+60 s_0)\cdot 2^{m} +17}{10\cdot 3\cdot 6\cdot 2^{m}},\\
                   & =-\dfrac{250\cdot 2^{m} +17}{180 \cdot 2^{m}}.\\
\end{array}
\end{equation}
Then we let $C\Bigl(f_{1,1} \Bigr)_{1,m}, C\Bigl(f_{1,1} \Bigr)^{-1}_{1,m},C\Bigl(f_{1,1} \Bigr)_{2,m}, C\Bigl(f_{1,1} \Bigr)^{-1}_{2,m}$ be the four half-circles having centers $\Bigl(\alpha_{1,1}\Bigr)_{1,m}$, $\Bigl(\alpha_{1,1}\Bigr)_{1, m}^{-1}$, $\Bigl(\alpha_{1,1} \Bigr)_{2,m}$ and $\Bigl(\alpha_{1,1}\Bigr)_{2, m}^{-1}$ respectively (see the equations \ref{eq:32} and \ref{eq:33}), and the same radius $\Bigl(r(1)\Bigr)_m=\dfrac{1}{10}\left(\dfrac{1}{3\cdot 6 \cdot 2^{m}}\right)$ (see the Figure \ref{Figure15}).
 \begin{figure}[h!]
\begin{center}
\includegraphics[scale=0.6]{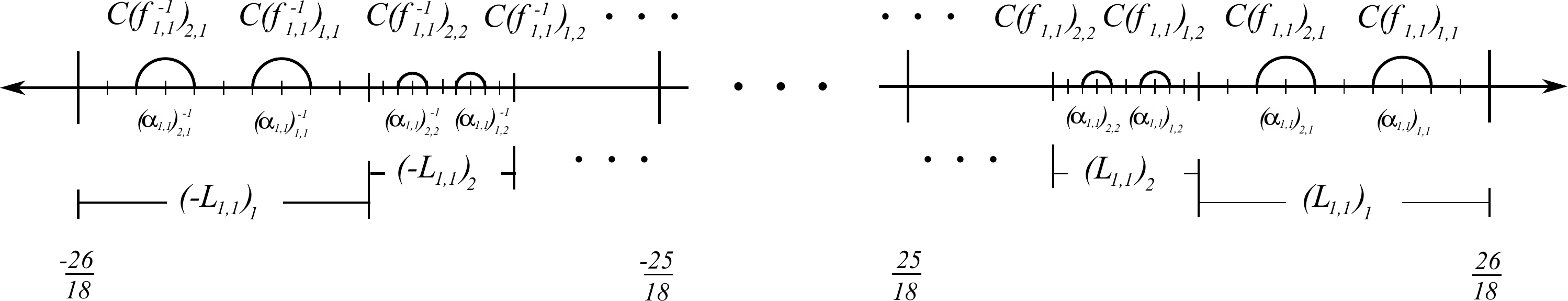}\\
\caption{\emph{The closed intervals $L_{1,1}$ and $-L_{1,1}$  are expressed as an union of closed subintervals $\bigcup\limits_{m\in\mathbb{N}}\Bigl(L_{1,1}\Bigr)_m$ and $\bigcup\limits_{m\in\mathbb{N}}\Bigl(-L_{1,1}\Bigl)_m$. Every  $\Bigl(L_{1,1}\Bigr)_m$ is built  with the addition of a term of the sequence $\dfrac{1}{3\cdot 6\cdot 2^m}$. On each $\Bigl(L_{1,1}\Bigr)_m$ there are two half-circles of radius $(r(1))_m$ }}
 \label{Figure15}
\end{center}
\end{figure}

\begin{remark}
By construction the centers of the half-circles $C\Bigl(f_{1,1} \Bigr)_{1,m}$ and $C\Bigl(f_{1,1} \Bigr)^{-1}_{2,m}$ (analogously, $C\Bigl(f_{1,1} \Bigr)^{-1}_{1,m}$ and $C\Bigl(f_{1,1} \Bigr)_{2,m}$) are symmetrical with respect to the imaginary axis \emph{i.e.}, we have $-\Bigl(\alpha_{1,1}\Bigr)_{1,m}=\Bigl(\alpha_{1,1}\Bigr)^{-1}_{2,m}$ (respectively, $-\Bigr(\alpha_{1,1}\Bigl)_{2,m}=\Bigl(\alpha_{1,1}\Bigr)^{-1}_{1,m}$).
\end{remark}

Now, for every $n\in\mathbb{N}$ we must calculate the M\"obius transformations and its respective inverse
\begin{equation}\label{eq:34}
\begin{array}{ccl}
\Bigl(f_{1,1} \Bigr)_{1,m}(z) &=&\dfrac{\Bigl(a_{1,1}\Bigr)_{1,m} z+ \Bigl(b_{1,1}\Bigr)_{1,m}}{\Bigl(c_{1,1}\Bigr)_{1,m}z+\Bigl(d_{1,1}\Bigr)_{1,m}},\\
\Bigl(f_{1,1}\Bigr)_{1,m}^{-1}(z)&=&\dfrac{\Bigr(d_{1,1}\Bigr)_{1,m}z-\Bigr(b_{1,1}\Bigr)_{1,m}}{-\Bigl(c_{1,1}\Bigr)_{1,m}z+\Bigl(a_{1,1}\Bigr)_{1,m}},
\end{array}
\end{equation}
having as isometric circles $C\Bigl(g_{1,1} \Bigr)_{1,m}$ and $C\Bigl(g_{1,1} \Bigr)^{-1}_{1,m}$, respectively.
By the remarks \ref{r:2.9} we have
\[
\begin{array}{cll}
\Bigl(a_{1,1}\Bigr)_{1,m} & = -((250+60s_0)\cdot 2^{m} +13) &=-(250\cdot 2^{m}+13),\\
\Bigl(c_{1,1}\Bigr)_{1,m} & = 10\cdot 3\cdot 6\cdot 2^{m} &=180\cdot 2^{m},  \\
\Bigl(d_{1,1}\Bigr)_{1,m} & = -((250+60 s_0)\cdot 2^{m} +17)&=-(250\cdot 2^m +17).\\
\end{array}
\]
Now, we substitute these values in the equation $a_{1,m}\cdot d_{1,m}-b_{1,m}\cdot c_{1,m}=1$ and computing we hold
\[
\begin{array}{cl}
\Bigl(b_{1,1}\Bigr)_{1,m} & =\dfrac{\left((250+60s_0)\cdot 2^{m} +13\right)((250+60 s_0)\cdot 2^{m} +17)-1}{10\cdot 3\cdot 6\cdot 2^{m}},\\
        &=\dfrac{(250\cdot 2^{m}+13)(250\cdot 2^m +17)-1}{180\cdot 2^{m}}.
\end{array}
\]
Hence, we can easily write the explicit form of the M\"{o}bius transformations of  the equations \ref{eq:34}.

On the other hand, for every $m\in \mathbb{N}$ we shall calculate the M\"obius transformations and its respective inverse
\begin{equation}\label{eq:35}
\begin{array}{ccl}
\Bigl(f_{1,1} \Bigr)_{2,m}(z) &=&\dfrac{\Bigl(a_{1,1}\Bigr)_{2,m}z+\Bigl(b_{1,1}\Bigr)_{2,m}}{\Bigl(c_{1,1}\Bigr)_{2,m}z+\Bigl(d_{1,1}\Bigr)_{2,m}}, \\
\Bigl(f_{1,1}\Bigr)_{2,m}^{-1}(z)&=&\dfrac{\Bigl(d_{1,1}\Bigr)_{2,m}z-\Bigl(b_{1,1}\Bigr)_{2,m}}{-\Bigl(c_{1,1}\Bigr)_{2,m}z+\Bigl(a_{1,1}\Bigr)_{2,m}},\\
\end{array}
\end{equation}
having as isometric circles $C\Bigl(f_{1,1} \Bigr)_{2,m}$  and $C\Bigl(f_{1,1} \Bigr)^{-1}_{2,m}$, respectively. By remarks \ref{r:2.9} we get
\[
\begin{array}{cll}
\Bigl(a_{1,1}\Bigr)_{2,m} & = -((250+60 s_0)\cdot 2^{m} +17)&=-(250\cdot 2^{m}+17) ,\\
\Bigl(c_{1,1}\Bigr)_{2,m} & = 10\cdot 3\cdot 6\cdot 2^{m}&=180\cdot 2^{m},\\
\Bigl(d_{1,1}\Bigr)_{2,m} & = -((250+60s_0)\cdot 2^{m} +13) &=-(250\cdot 2^{m}+13).\\
\end{array}
\]
Now, we substitute these values in the equation $\Bigl(a_{1,1}\Bigr)_{2,m}\cdot \Bigl(d_{1,1}\Bigr)_{2,m}-\Bigl(b_{1,1}\Bigr)_{2,m}\cdot \Bigl(c_{1,1}\Bigr)_{2,m}=1$ and computing we hold
\[
\begin{array}{cl}
\Bigl(b_{1,1}\Bigr)_{2,m} & = \dfrac{((250+60 s_0)\cdot 2^{m} +17)((250+60s_0)\cdot 2^{m} +13)-1}{10\cdot 3\cdot 6\cdot 2^{m}},\\
        & =\dfrac{(250\cdot 2^{m}+17)(250\cdot 2^{m}+13)-1}{180\cdot 2^{m}}.\\
\end{array}
\]
Hence, we can easily write the explicit form of the M\"{o}bius transformation of  the equations \ref{eq:35}. Then, we define the sets $\Bigl(J_{1,1}\Bigr)_{left-right}$ and $\Bigl(\mathcal{C}_{1,1}\Bigr)_{left-right}$ composed by M\"{o}bius transformation and half-circle, respectively, as such
\begin{equation}\label{eq:36}
\begin{array}{cl}
\Bigl(J_{1,1}\Bigr)_{left-right} &:=\left\{\Bigl(f_{1,1}\Bigr)_{1,m}(z), \, \Bigl(f_{1,1}\Bigr)^{-1}_{1,m}(z), \,  \Bigl(f_{1,1}\Bigr)_{2,m}(z), \Bigl(f_{1,1}\Bigr)_{2,m}^{-1}(z):m\in\mathbb{N} \right\},\\
\Bigl(\mathcal{C}_{1,1}\Bigr)_{left-right} & := \left\{ C\Bigl(f_{1,1}\Bigr)_{1,m}, \, C\Bigl(f_{1,1}\Bigr)^{-1}_{1,m}, \, C\Bigl(f_{1,1}\Bigr)_{2,m}, \, C\Bigl(f_{1,1}\Bigr)_{2,m}^{-1}:m\in\mathbb{N} \right\}. \\
\end{array}
\end{equation}

We remark that by construction each M\"{o}bius transformation of $\Bigl(J_{1,1}\Bigr)_{left-right}$ is hyperbolic and the half-circles of $\Bigl(\mathcal{C}_{1,1}\Bigr)_{left-right}$ are pairwise disjoint.

\textbf{Part II. Building sequences of half-circle at the right of $C(f_{1,1})$ and at the left of $C(f^{-1}_{1,1})$.} We consider the fifth sixth closed subinterval of $\hat{I}_{1,1}$, which is given by
\[
S_{1,1}=\left[\dfrac{3+s_1}{3} -\dfrac{2}{3\cdot 6}, \dfrac{3+s_1}{3} +\dfrac{1}{3\cdot 6}\right] = \left[\dfrac{16+6s_1}{3\cdot6},\dfrac{17+6s_1}{3\cdot6}\right]=\left[\dfrac{28}{18},\dfrac{29}{18}\right].
\]
Then we denote as $-S_{1,1}$ the closed interval symmetric with respect to the imaginary axis \emph{i.e.},
\[
-S_{1,1}=\left[-\dfrac{17+6s_1}{3\cdot6},- \dfrac{16+6s_1}{3\cdot6} \right] =\left[-\dfrac{29}{18},-\dfrac{28}{18}\right].
\]

See Figure \ref{Figure16}. Then we write $S_{1,1}$ and $-S_{1,1}$ as the union of closed subintervals $\bigcup\limits_{m\in\mathbb{N}}\Bigl(S_{1,1}\Bigr)_m$ and $\bigcup\limits_{m\in\mathbb{N}}\Bigr(-S_{1,1}\Bigl)_m$ respectively, as such
\[
\begin{array}{clccc}
  \Bigr(S_{1,1}\Bigl)_m & =\left[\dfrac{17+6s_1}{3\cdot 6} -\dfrac{1}{3\cdot 6 \cdot 2^{m-1}}, \dfrac{17+6s_1}{3\cdot 6}- \dfrac{1}{3\cdot 6 \cdot 2^m} \right], \\
      &=\left[ \dfrac{(17+6s_1)\cdot 2^{m-1}-1}{3\cdot 6\cdot 2^{m-1}}, \dfrac{(17+6s_1)\cdot 2^m -1}{3\cdot 6\cdot 2^m}\right],\\
      & =\left[\dfrac{29\cdot 2^{m-1}-1}{3\cdot6\cdot 2^{m-1}}, \dfrac{29\cdot 2^{m}-1}{3\cdot6\cdot 2^{m}} \right],\\
  \Bigl(-S_{1,1}\Bigr)_m&=\left[ -\dfrac{(17+6s_1)\cdot 2^m -1}{3\cdot 6\cdot 2^m}, -\dfrac{(17+6s_1)\cdot 2^{m-1}-1}{3\cdot 6\cdot 2^{m-1}} \right],\\
      & =\left[-\dfrac{29\cdot 2^{m-1}-1}{3\cdot6\cdot 2^{m-1}}, -\dfrac{29\cdot 2^{m}-1}{3\cdot6\cdot 2^{m}} \right].\\
\end{array}
\]
By definition  the length of $\Bigl(S_{1,1}\Bigr)_m$ and $\Bigl(-S_{1,1}\Bigr)_m$ is $\dfrac{1}{3\cdot 6\cdot 2^{m}}$ for each $m\in\mathbb{N}$, then  we choose the points $\Bigl(\alpha_{1,1}\Bigr)_{3,m}$ and $\Bigl(\alpha_{1,1}\Bigr)_{4,m}$  belonging to $\Bigl(S_{1,1}\Bigr)_m$, as such
\begin{equation}\label{eq:37}
\begin{array}{cll}
\Bigl(\alpha_{1,1}\Bigr)_{3, m}& = \dfrac{(17+6s_1)\cdot 2^{m-1}-1}{3\cdot 6\cdot 2^{m-1}}+\dfrac{7}{10}\left( \dfrac{1}{3\cdot 6\cdot 2^{m}}\right) & =\dfrac{(170+60s_1)\cdot 2^{m}-13}{10\cdot 3\cdot 6\cdot 2^{m}},\\
            & =\dfrac{290\cdot2^m-13}{180\cdot2^m},&\\
\Bigl(\alpha_{1,1}\Bigr)_{4, m}& = \dfrac{(170+60s_1)\cdot 2^{m-1}-1}{3\cdot 6\cdot 2^{m-1}} +\dfrac{3}{10}\left( \dfrac{1}{3\cdot 6 \cdot 2^{m}}\right) &=\dfrac{(170+60s_1)\cdot 2^{m}-17}{10\cdot 3\cdot 6\cdot 2^{m}},\\
             &=\dfrac{290\cdot 2^m-17}{180\cdot 2^m},&\\
\end{array}
\end{equation}
and the points $\Bigl(\alpha_{1,1}\Bigr)^{-1}_{3,m}$ and $\Bigl(\alpha_{1,1}\Bigl)^{-1}_{4,m}$ belonging to $\Bigl(-S_{1,1}\Bigr)_m$, as such
\begin{equation}\label{eq:38}
\begin{array}{cl}
\Bigl(\alpha_{1,1}\Bigr)_{3, m}^{-1} &= - \dfrac{(170+60s_1)\cdot 2^{m}-17}{10\cdot 3\cdot 6\cdot 2^{m}},\\
                  & =-\dfrac{290\cdot 2^m-17}{180\cdot 2^m},\\
\Bigl(\alpha_{1,1}\Bigr)_{4, m}^{-1} &=- \dfrac{(170+60s_1)\cdot 2^{m}-13}{10\cdot 3\cdot 6\cdot 2^{m}},\\
                  &=-\dfrac{290\cdot2^m-13}{180\cdot2^m}.\\
\end{array}
\end{equation}

Then we let $C\Bigl(f_{1,1} \Bigr)_{3,m}, C\Bigl(f_{1,1} \Bigr)^{-1}_{3,m}, C\Bigl(f_{1,1} \Bigr)_{4,m}, C\Bigl(f_{1,1} \Bigr)^{-1}_{4,m}$ be the $4m$ half-circles having centers $\Bigl(\alpha_{1,1}\Bigr)_{3,m}$, $\Bigl(\alpha_{1,1}\Bigr)_{3, m}^{-1}$, $\Bigl(\alpha_{1,1}\Bigr)_{4,m}$ and $\Bigl(\alpha_{1,1}\Bigr)_{4, m}^{-1}$ respectively (see the equations \ref{eq:37} and \ref{eq:38}), and the same radius $\Bigl(r(1)\Bigr)_m=\dfrac{1}{10}\left(\dfrac{1}{3\cdot 6 \cdot 2^{m}}\right)$ (see the Figure \ref{Figure16}).
\begin{figure}[h!]
\begin{center}
\includegraphics[scale=0.6]{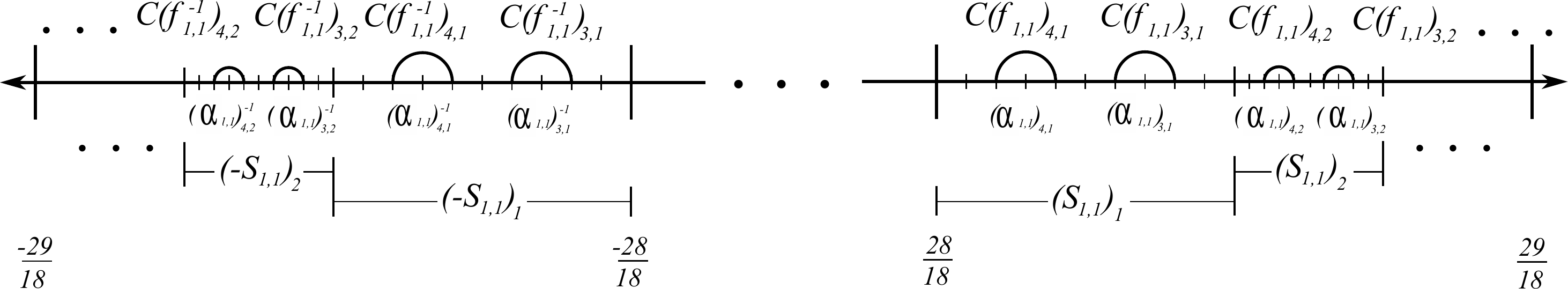}\\
\caption{\emph{The closed intervals $S_{1,1}$ and $-S_{1,1}$  are expressed as an union of closed subintervals $\bigcup\limits_{m\in\mathbb{N}}\Bigl(S_{1,1}\Bigr)_m$ and $\bigcup\limits_{m\in\mathbb{N}}\Bigl(-S_{1,1}\Bigl)_m$. Every  $\Bigl(S_{1,1}\Bigr)_m$ is built  with the subtraction of a term of the sequence $\dfrac{1}{3\cdot 6\cdot 2^m}$. On each $\Bigl(S_{1,1}\Bigr)_m$ there are two half-circles of radius $(r(1))_m$.}}
 \label{Figure16}
\end{center}
\end{figure}

\begin{remark}
By construction, the centers of the half-circles $C\Bigl(f_{1,1} \Bigr)_{3,m}$ and $C\Bigl(f_{1,1} \Bigr)^{-1}_{4,m}$ (analogously, $C\Bigl(f_{1,1} \Bigr)^{-1}_{3,m}$ and $C\Bigl(f_{1,1} \Bigr)_{4,m}$) are symmetrical  with respect to the imaginary axis \emph{i.e.}, we have $-\Bigl(\alpha_{1,1}\Bigr)_{3,m}=\Bigl(\alpha_{1,1}\Bigr)^{-1}_{4,m}$ (respectively, $-\Bigl(\alpha_{1,1}\Bigr)_{4,m}=\Bigl(\alpha_{1,1}\Bigr)^{-1}_{3,m}$).
\end{remark}

Now, for every $m\in\mathbb{N}$ we must calculate the M\"obius transformations and its respective inverse
\begin{equation}\label{eq:39}
\begin{array}{ccl}
\Bigl(f_{1,1} \Bigr)_{3,m}(z) &=&\dfrac{\Bigl(a_{1,1}\Bigr)_{3,m}z+\Bigl(b_{1,1}\Bigr)_{3,m}}{\Bigl(c_{1,1}\Bigr)_{3,m}z+\Bigl(d_{1,1}\Bigr)_{3,m}},\\
\Bigl(f_{1,1}\Bigr)_{3,m}^{-1}(z)&=&\dfrac{\Bigl(d_{1,1}\Bigr)_{3,m}z-\Bigl(b_{1,1}\Bigr)_{3,m}}{-\Bigl(c_{1,1}\Bigr)_{3,m}z+\Bigl(a_{1,1}\Bigr)_{3,m}},
\end{array}
\end{equation}
having as isometric circles $C\Bigl(f_{1,1} \Bigr)_{3,m}$ and $C\Bigl(f_{1,1} \Bigr)^{-1}_{3,m}$, respectively. By remarks \ref{r:2.9} we have
\[
\begin{array}{cll}
\Bigl(a_{1,1}\Bigr)_{3,m} & =-((170+60s_1)\cdot 2^{m}-17) & =-(290\cdot 2^m-17) ,\\
\Bigl(c_{1,1}\Bigr)_{3,m} & = 10\cdot 3\cdot 6\cdot 2^m & =180\cdot 2^m, \\
\Bigl(d_{1,1}\Bigr)_{3,m} & = -((170+60s_1)\cdot 2^{m}-13) &=-(290\cdot2^m-13).\\
\end{array}
\]
Now, we substitute these values in the equation $\Bigl(a_{1,1}\Bigr)_{3,m}\cdot \Bigl(d_{1,1}\Bigr)_{3,m}-\Bigl(b_{1,1}\Bigr)_{3,m}\cdot \Bigl(c_{1,1}\Bigr)_{3,m}=1$ and computing we hold
\[
\begin{array}{cl}
\Bigl(b_{1,1}\Bigr)_{3,m} & = \dfrac{(170+60s_1)\cdot 2^{m}-17)((170+60s_1)\cdot 2^{m}-13)-1}{10\cdot 3\cdot 6\cdot 2^m},\\
        &=\dfrac{(290\cdot 2^m-17)(290\cdot2^m-13)-1}{180\cdot2^m}.\\
\end{array}
\]
Hence, we can easily write the explicit form of the M\"{o}bius transformations of the equations \ref{eq:39}.

In a similar way, for every $n\in\mathbb{N}$ we shall calculate the M\"obius transformations and its respective inverse
\begin{equation}\label{eq:40}
\begin{array}{ccl}
\Bigl(f_{1,1}\Bigr)_{4,m}(z) & = & \dfrac{\Bigl(a_{1,1}\Bigr)_{4,m}z+\Bigl(b_{1,1}\Bigr)_{4,m}}{\Bigl(c_{1,1}\Bigr)_{4,m}z+\Bigl(d_{1,1}\Bigr)_{4,m}},\\
\Bigl(f_{1,1}\Bigr)_{4,m}^{-1}(z) & = &\dfrac{\Bigl(d_{1,1}\Bigr)_{4,m}z-\Bigl(b_{1,1}\Bigr)_{4,m}}{-\Bigl(c_{1,1}\Bigr)_{4,m}z+\Bigl(a_{1,1}\Bigr)_{4,m}},\\
\end{array}
\end{equation}
having as isometric circles $C\Bigl(f_{1,1} \Bigr)_{4,m}$  and $C\Bigl(f_{1,1} \Bigr)^{-1}_{4,m}$, respectively. By remarks \ref{r:2.9} we get

\[
\begin{array}{cll}
\Bigl(a_{1,1}\Bigr)_{4,m} & = -((170+60s_1)\cdot 2^{m}-13)&=-(290\cdot2^m-13) ,\\
\Bigl(c_{1,1}\Bigr)_{4,m} & = 10\cdot 3\cdot 6\cdot 2^m & =180\cdot 2^m , \\
\Bigl(d_{1,1}\Bigr)_{4,m} & = -((170+60s_1)\cdot 2^{m}-17)&=-(290\cdot 2^m-17).\\
\end{array}
\]
Now, we substitute these values in the equation $\Bigl(a_{1,1}\Bigr)_{4,m}\cdot \Bigl(d_{1,1}\Bigr)_{4,m}-\Bigl(b_{1,1}\Bigr)_{4,m}\cdot \Bigl(c_{1,1}\Bigr)_{4,m}=1$ and computing we hold
\[
\begin{array}{cl}
\Bigl(b_{1,1}\Bigr)_{4,m} & = \dfrac{((170+60s_1)\cdot 2^{m}-13)((170+60s_1)\cdot 2^{m}-17)-1}{10\cdot 3\cdot 6\cdot 2^m},\\
        &=\dfrac{(290\cdot2^m-13)(290\cdot 2^m-17)-1}{180\cdot2^m}.\\
\end{array}
\]
Hence, we can easily write the explicit form of the M\"{o}bius transformations of  equations \ref{eq:40}. Then, we define the sets $\Bigl(J_{1,1}\Bigr)_{right-left}$ and $\Bigl(\mathcal{C}_{1,1}\Bigr)_{right-left}$ composed by M\"{o}bius transformations and half-circles, respectively, as such
\begin{equation}\label{eq:41}
\begin{array}{cl}
\Bigl(J_{1,1}\Bigr)_{right-left} &:=\left\{\Bigl(f_{1,1}\Bigr)_{3,m}(z), \, \Bigl(f_{1,1}\Bigr)^{-1}_{3,m}(z), \,  \Bigl(f_{1,1}\Bigr)_{4,m}(z), \Bigl(f_{1,1}\Bigr)_{4,m}^{-1}(z):m\in\mathbb{N} \right\},\\
\Bigl(\mathcal{C}_{1,1}\Bigr)_{right-left} & := \left\{ C\Bigl(f_{1,1}\Bigr)_{3,m}, \, C\Bigl(f_{1,1}\Bigr)^{-1}_{3,m}, \, C\Bigl(f_{1,1}\Bigr)_{4,m}, \, C\Bigl(f_{1,1}\Bigr)_{4,m}^{-1}:m\in\mathbb{N} \right\}. \\
\end{array}
\end{equation}

We remark that by construction each M\"{o}bius transformation of $\Bigl(J_{1,1}\Bigr)_{right-left}$ is hyperbolic and the half-circles of $\Bigl(\mathcal{C}_{1,1}\Bigr)_{right-left}$ are pairwise disjoint.

Finally, from the equations \ref{eq:31}, \ref{eq:36}, and \ref{eq:41} we define the sets of the M\"{o}bius transformations and their respective isometric circles to the step 1, as such
\begin{equation}\label{eq:42}
\begin{array}{cl}
J_{1}
     & := \left\{f_{1,1}(z), \, f_{1,1}^{-1}(z), \, \Bigl( f_{1,1}\Bigr)_{s,m}(z), \, \Bigl( f_{1,1}\Bigr)_{s,m}^{-1}(z): s\in \{1,\ldots, 4\}, \, m\in\mathbb{N}\right\},\\
\mathcal{C}_1 
              & := \left\{C(f_{1,1}), \, C(f_{1,1}^{-1}), \, C\Bigl( f_{1,1}\Bigr)_{s,m}, \, C\Bigl( f_{1,1}\Bigr)_{s,m}^{-1}: s\in \{1,\ldots, 4\}, \, m\in\mathbb{N}\right\}.\\
\end{array}
\end{equation}

We remark that by construction each M\"{o}bius transformation of $J_{1}$ is hyperbolic and the half-circles of $\mathcal{C}_{1}$ are pairwise disjoint.

\begin{remark}\label{r:3.17}
If we consider the classical Schottky group of rank one $\Gamma_1$, which is generated by the set $J_1$ (see the equation \ref{eq:42}). Using the same ideas as the Infinite Loch Ness monster case, it is easy to check that the quotient space $S_1:=\mathbb{H}/ \Gamma_1$ (see the equation \ref{eq:13}) is a hyperbolic surface having infinite area and homeomorphic to the ladder of Jacob \emph{i.e.}, $S_1$ has two ends and each one having infinite genus. See the Figure \ref{Figure17}.
\begin{figure}[h!]
\begin{center}
\includegraphics[scale=0.5]{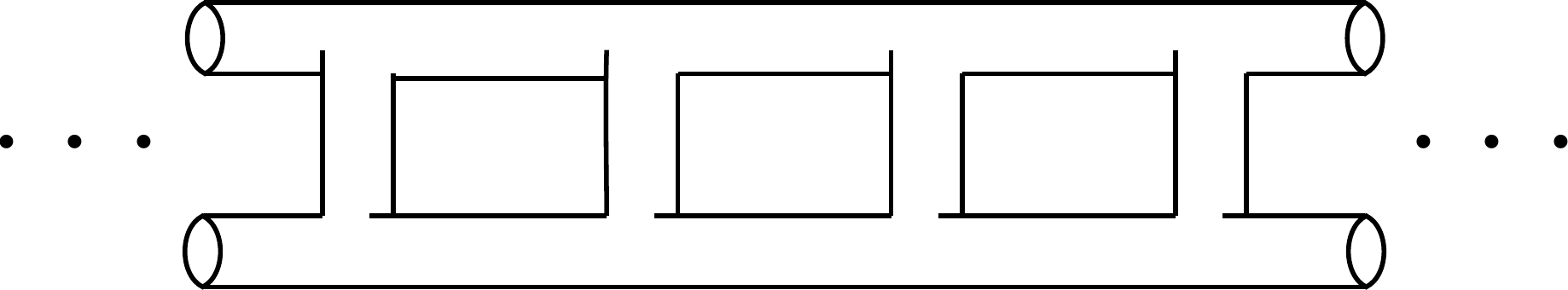}\\
\caption{\emph{The ladder of Jacob}.}
 \label{Figure17}
\end{center}
\end{figure}
\end{remark}

\textbf{For $n$. Building the set $J_n$ containing infinitely countable M\"obius transformation and the set $\mathcal{C}_n$ composed by its respective isometric circles.} We consider the closed subset  (see the theorem \ref{T:2.1})
\[
I_{n-1}=\bigcup\limits_{k=0}^{2^{n-1} -1} \left[\dfrac{3^{n-1}+s_k}{3^{n-1}}, \dfrac{3^{n-1}-s_k +1}{3^{n-1}}\right]\subset I_0
 \]
and its symmetrical with respect to the imaginary axis. See the Figure \ref{Figure18}.
\[
-I_{n-1}=\bigcup\limits_{k=0}^{2^{n-1} -1} \left[-\dfrac{3^{n-1}-s_k +1}{3^{n-1}}, - \dfrac{3^{n-1}+s_k}{3^{n-1}}\right]\subset -I_0.
 \]
Then for each $k\in \{0 ,\ldots,2^{n-1}-1\}$ we let $\hat{I}_{n,k},-\hat{I}_{n,k}$ be the middle third of the closed intervals
\begin{align*}
\left[ \dfrac{3^{n-1}+s_{k-1}}{3^{n-1}},\dfrac{3^{n-1}+s_{k-1}+1}{3^{n-1}}\right] &  & \text{ and } & &\left[-\dfrac{3^{n-1}+s_{k-1}+1}{3^{n-1}}, -\dfrac{3^{n-1}+s_{k-1}}{3^{n-1}}\right],\\
\end{align*}
respectively. By the remark \ref{r:2.2} we hold
\begin{align*}
\hat{I}_{n,k}   & =\left[\dfrac{3^n + s_{2k-2} +1}{3^n}, \dfrac{3^n + s_{2k-1}}{3^n}\right],\\
-\hat{I}_{n,k}  & =\left[-\dfrac{3^n + s_{2k-1}}{3^n}, -\dfrac{3^n + s_{2k-2} +1}{3^n} \right].\\
\end{align*}

\begin{figure}[h!]
\begin{center}
\includegraphics[scale=0.6]{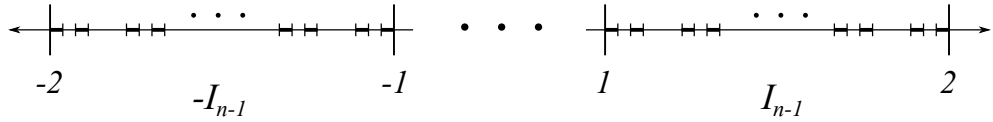}\\
\caption{\emph{The closed intervals $\hat{I}_{n,k},-\hat{I}_{n,k}$}.}
 \label{Figure18}
\end{center}
\end{figure}

On the other hand, for every $k\in\{0,\ldots, 2^{n-1}-1\}$ the length of the closed intervals $\hat{I}_{n,k}$ and $-\hat{I}_{n,k}$ is $\dfrac{1}{3^n}$ and their respective middle points are $\alpha_{n,k}$ and $\alpha_{n,k}^{-1}$ as in the equation \ref{eq:21}. Then we let $C(f_{n,k}), C(f^{-1}_{n,k})$ be the half-circles having centers $\alpha_{n,k}$ and $\alpha^{-1}_{n,k}$ respectively, and the same radius $r(n)=\dfrac{1}{3^n\cdot 6\cdot 2}$.

Now, for each $k\in\{0,\ldots, 2^{n-1}-1\}$ we shall calculate the M\"obius transformations and its respective inverse for each $n$

\begin{equation}\label{eq:43}
\begin{array}{ccl}
f_{n,k}(z)      &=&\dfrac{a_{n,k}z+b_{n,k}}{c_{n,k}z+d_{n,k}},\\
f_{n,k}^{-1}(z) &=&\dfrac{d_{n,k}z-b_{n,k}}{-c_{n,k}z+a_{n,k}},
\end{array}
\end{equation}
which has as isometric circles $C(f_{n,k})$  and $C(f^{-1}_{n,k})$, respectively (see the Figure \ref{Figure19}).

\begin{figure}[h!]
\begin{center}
\includegraphics[scale=0.6]{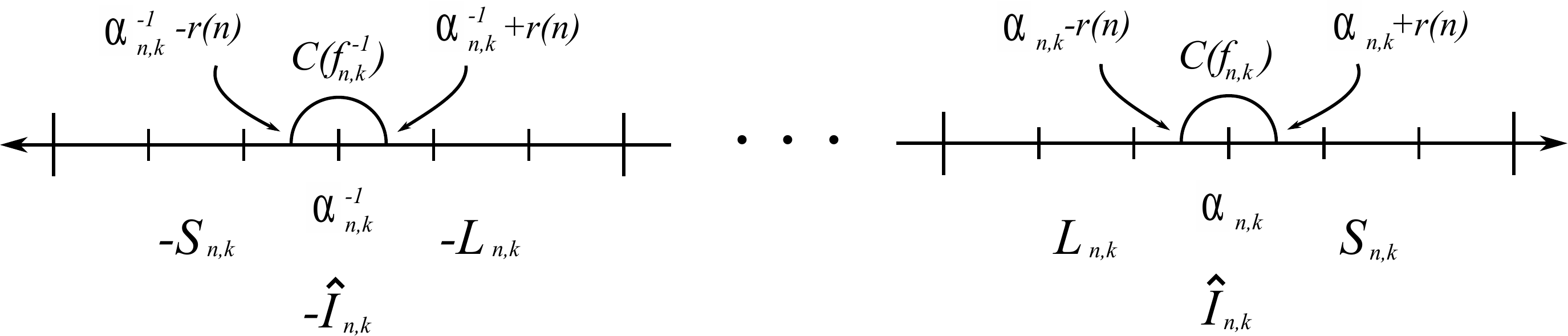}\\
\caption{\emph{$L_{n,k}$ represents the second sixth closed subinterval of $\hat{I}_{n,k}$ and $-L_{n,k}$ represents their symmetric with respect to the imaginary axis.}}
 \label{Figure19}
\end{center}
\end{figure}
By remarks \ref{r:2.9} we have
\[
\begin{array}{cclcl}
a_{n,k} & = &-6\cdot(3^n \cdot 2+2(1+ s_{2k-1})+1),\\
c_{n,k} & = & 6\cdot 3^n\cdot 2 ,\\
d_{n,k} & = &-6\cdot(3^n \cdot 2+2(1+ s_{2k-1})+1).\\
\end{array}
\]
Now, we substitute these values in the equation $a_{1,1}\cdot d_{n,k}-b_{n,k}\cdot c_{n,k}=1$ and computing we hold
\[
b_{n,k} = \dfrac{6^2\cdot(3^n\cdot 2+2(1+s_{2k-1})+1)^2 -1 }{6\cdot 3^n\cdot 2}.
\]
Thus, we know the explicit form of the M\"obius transformations of the equations \ref{eq:43}. Then for each $k\in \{0,\ldots, 2^{n-1}-1\}$ we define the sets
\begin{equation}\label{eq:44}
\begin{array}{ccc}
J_{n,k}           &:=& \{f_{n,k}(z), \, f^{-1}_{n,k}(z) \},\\
\mathcal{C}_{n,k} &:=& \{C(f_{n,k}), \, C(f^{-1}_{n,k}) \}.
\end{array}
\end{equation}
By construction each M\"{o}bius transformation of $J_{n,k}$ is hyperbolic and the half-circle of $C_{n,k}$ are pairwise disjoint.

Now, for each $k\in\{0,\ldots, 2^{n-1}-1\}$ we shall build sequences of half-circles at the left and at the right of $C(f_{n,k})$, $C(f^{-1}_{n,k})$, whose radius converge to zero, each sequence will have associated a suitable sequence of M\"{o}bius transformations as follows.

\textbf{Part I. Building sequences of half-circles at the left of $C(f_{n,k})$ and at the right of $C(f^{-1}_{n,k})$.} For each $k\in\{0,\ldots,2^{n-1}-1\}$ we consider the second sixth closed subinterval of each
$$\hat{I}_{n,k}=\left[\dfrac{3^n+s_{2k-2}+1}{3^n}, \dfrac{3^n+ s_{2k-1}}{3^n}\right],$$
which are given by

\[
L_{n,k}:=\left[\dfrac{3^n+s_{2k-2}+1}{3^n} +\dfrac{1}{3^n\cdot 6}, \dfrac{3^n+s_{2k-2}+1}{3^n} +\dfrac{2}{3^n\cdot 6}\right].
\]
We denote as $-L_{n,k}$ the closed interval symmetric with respect to the imaginary axis of $L_{n,k}$ \emph{i.e.},
\[
-L_{n,k} = \left[- \dfrac{3^n+s_{2k-2}+1}{3^n} -\dfrac{2}{3^n\cdot 6}, -\dfrac{3^n+s_{2k-2}+1}{3^n} -\dfrac{1}{3^n\cdot 6}\right].
\]

Then we write $L_{n,k}$ and $-L_{n,k}$ as an union of closed subintervals $\bigcup\limits_{m\in\mathbb{N}}\Bigl(L_{n,k}\Bigr)_m$ and $\bigcup\limits_{m\in\mathbb{N}}\Bigl(-L_{n,k}\Bigl)_m$,  which are given respectively as

\[
 \Bigl(L_{n,k} \Bigr)_m:=\left[\dfrac{3^n+s_{2k-2}+1}{3^n} +\dfrac{1}{3^n\cdot 6\cdot 2^m}, \dfrac{3^n+s_{2k-2}+1}{3^n} +\dfrac{1}{3^n\cdot 6\cdot 2^{m+1}}\right],
\]

\[
 \Bigl(-L_{n,k} \Bigr)_m:=\left[- \dfrac{3^n+s_{2k-2}+1}{3^n} -\dfrac{1}{3^n\cdot 6\cdot 2^{m+1}}, -\dfrac{3^n+s_{2k-2}+1}{3^n} -\dfrac{1}{3^n\cdot 6\cdot 2^m}\right].
\]
\begin{figure}[h!]
\begin{center}
\includegraphics[scale=0.6]{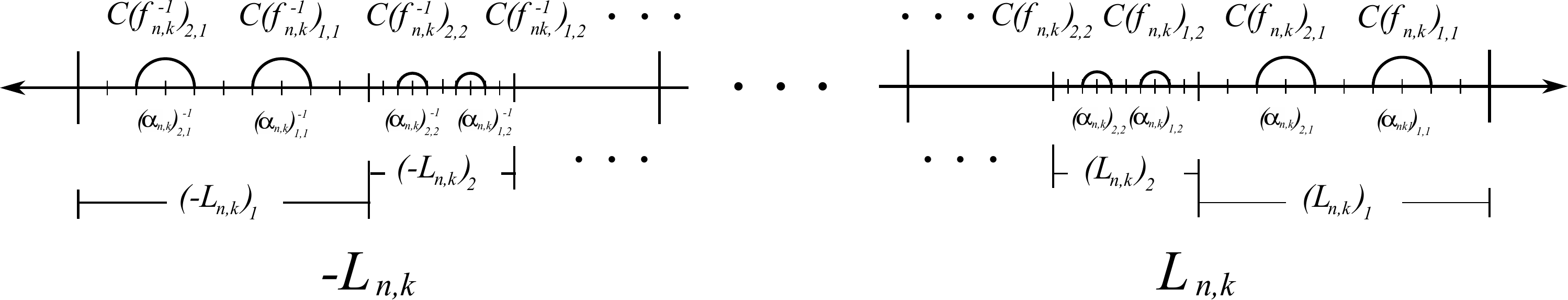}\\
\caption{\emph{The closed intervals $-L_{n,k}$  and $L_{n,k}$ written as an union of closed subintervals $\bigcup\limits_{m\in\mathbb{N}}\Bigl(L_{n,k}\Bigr)_m$ and $\bigcup\limits_{m\in\mathbb{N}}\Bigl(-L_{n,k}\Bigl)_m$.}}
 \label{Figure20}
\end{center}
\end{figure}
By definition the length of $ \Bigl(L_{n,k} \Bigr)_m$ and $ \Bigl(-L_{n,k} \Bigr)_m$ is $\dfrac{1}{3^n\cdot 6\cdot 2^{m}}$ for each $m\in\mathbb{N}$, then  we choose the points $(\alpha_{n,k})_{1,m}, (\alpha_{n,k})_{2,m}$ belonged to $ \Bigl(L_{n,k} \Bigr)_m$ as

\begin{equation}\label{eq:45}
\begin{array}{cl}
\Bigl(\alpha_{n,k}\Bigr)_{1,m}& =\dfrac{3^n+s_{2k-2}+1}{3^n} +\dfrac{1}{3^n\cdot 6\cdot 2^m}+\dfrac{7}{10}\left(\dfrac{1}{3^n\cdot 6\cdot 2^m}\right),\\
& =\dfrac{10\cdot 6\cdot 2^m(3^n+s_{2k-2}+1)+17}{3^n\cdot 10 \cdot 6\cdot 2^m},\\
\Bigr(\alpha_{n,k}\Bigl)_{2,m} &= \dfrac{3^n+s_{2k-2}+1}{3^n} +\dfrac{1}{3^n\cdot 6\cdot 2^m}+\dfrac{3}{10}\left(\dfrac{1}{3^n\cdot 6\cdot 2^m}\right),\\
& =\dfrac{10\cdot 6\cdot 2^m(3^n+s_{2k-2}+1)+13}{3^n\cdot 10 \cdot 6\cdot 2^m},
\end{array}
\end{equation}
and the points $\alpha^{-1}_{1m}, \alpha^{-1}_{2,m}$ belonged to $ \Bigl(-L_{n,k} \Bigr)_m$ as
\begin{equation}\label{eq:46}
\begin{array}{cl}
\Bigl(\alpha_{n,k}\Bigr)_{1,m}^{-1}& =-\dfrac{3^n+s_{2k-2}+1}{3^n} -\dfrac{1}{3^n\cdot 6\cdot 2^m}-\dfrac{3}{10}\left(\dfrac{1}{3^n\cdot 6\cdot 2^m}\right),\\
& =-\dfrac{10\cdot 6\cdot 2^m(3^n+s_{2k-2}+1)+13}{3^n\cdot 10 \cdot 6\cdot 2^m},\\
\Bigl(\alpha_{n,k}\Bigr)_{2,m}^{-1} &= -\dfrac{3^n+s_{2k-2}+1}{3^n} -\dfrac{1}{3^n\cdot 6\cdot 2^m}-\dfrac{7}{10}\left(\dfrac{1}{3^n\cdot 6\cdot 2^m}\right),\\
& =-\dfrac{10\cdot 6\cdot 2^m(3^n+s_{2k-2}+1)+17}{3^n\cdot 10 \cdot 6\cdot 2^m}.
\end{array}
\end{equation}

Then we let $C\Bigl(f_{n,k} \Bigr)_{1,m}, C\Bigl(f_{n,k} \Bigr)^{-1}_{1,m},C\Bigl(f_{n,k} \Bigr)_{2,m}, C\Bigl(f_{n,k} \Bigr)^{-1}_{2,m}$ be the $m$ half-circles having centers $\Bigl(\alpha_{n,k}\Bigr)_{1,m}$, $\Bigl(\alpha_{n,k}\Bigr)_{1, m}^{-1}$, $\Bigl(\alpha_{n,k} \Bigr)_{2,m}$ and $\Bigl(\alpha_{n,k}\Bigr)_{2, m}^{-1}$ respectively (see the equation \ref{eq:45} and \ref{eq:46}), and the same radius $\Bigl(r(1)\Bigr)_m=\dfrac{1}{10}\left(\dfrac{1}{3^n\cdot 6 \cdot 2^{m}}\right)$ (see the Figure \ref{Figure20}). Now, for every $m\in\mathbb{N}$ we shall calculate the M\"obius transformations and its respective inverse

\begin{equation}\label{eq:47}
\begin{array}{ccl}
\Bigl(f_{n,k} \Bigr)_{1,m}(z) &=&\dfrac{(a_{n,k})_{1,m}z+(b_{n,k})_{1,m}}{(c_{n,k})_{1,m}z+(d_{n,k})_{1,m}},\\
\Bigl(f_{n,k}\Bigr)_{1,m}^{-1}(z)&=&\dfrac{d_{1,m}z-b_{1,m}}{-c_{1,m}z+a_{1,m}},
\end{array}
\end{equation}
having as isometric circles $C\Bigl(f_{n,k} \Bigr)_{1,m}$ and $C\Bigl(f_{n,k} \Bigr)^{-1}_{1,m}$, respectively. By remarks \ref{r:2.9} we have
\[
\begin{array}{cll}
(a_{n,k})_{1,m} & = -(10\cdot 6\cdot 2^m(3^n+s_{2k-2}+1)+13),\\
(c_{n,k})_{1,m} & = 10\cdot 3^n\cdot 6\cdot 2^{m},  \\
(d_{n,k})_{1,m} & = -(10\cdot 6\cdot 2^m(3^n+s_{2k-2}+1)+17).\\
\end{array}
\]
Now, we substitute these values in the equation $(a_{n,k})_{1,m}\cdot (d_{n,k})_{1,m}-(b_{n,k})_{1,m}\cdot (c_{n,k})_{1,m}=1$ and computing we hold
\[
\begin{array}{cl}
(b_{n,k})_{1,m}=\dfrac{(10\cdot 6\cdot 2^m(3^n+s_{2k-2}+1)+13)(10\cdot 6\cdot 2^m(3^n+s_{2k-2}+1)+17)-1}{ 10\cdot 3^n\cdot 6\cdot 2^{m}}.
\end{array}
\]
Hence, we can easily write the explicit form of the M\"{o}bius transformations of the equations \ref{eq:47}. Then for every $k\in\{0,\ldots,2^{n-1}-1\}$ we define the sets
\begin{equation}\label{eq:48}
\begin{array}{cl}
\Bigl(J_{n,k}\Bigr)_{left-right} &:=\left\{\Bigl(f_{n,k}\Bigr)_{1,m}(z), \, \Bigl(f_{n,k}\Bigr)^{-1}_{1,m}(z), \, \Bigl(f_{n,k}\Bigr)_{2,m}(z), \Bigl(f_{n,k}\Bigr)_{2,m}^{-1}(z): m\in\mathbb{N} \right\},\\
\Bigl(\mathcal{C}_{n,k}\Bigr)_{left-right} & := \left\{ C\Bigl(f_{n,k}\Bigr)_{3,m}, \, C\Bigl(f_{n,k}\Bigr)^{-1}_{3,m}, \, C\Bigl(f_{n,k}\Bigr)_{4,m}, \, C\Bigl(f_{n,k}\Bigr)_{4,m}^{-1}:m\in\mathbb{N} \right\}. \\
\end{array}
\end{equation}
By construction each M\"{o}bius transformation of $\Bigl(J_{n,k}\Bigr)_{left-right}$ is hyperbolic and the half-circle of $\Bigl(\mathcal{C}_{n,k}\Bigr)_{left-right}$ are pairwise disjoint.

\textbf{Part II. Building  sequences of half-circle on the right $C(f_{n,k})$ and at the left $C(f^{-1}_{n,k})$.} For each $k\in\{0,\ldots,2^{n-1}-1\}$ we consider the fifth sixth closed subinterval of each
$$\hat{I}_{n,k}=\left[\dfrac{3^n+s_{2k-2}+1}{3^n}, \dfrac{3^n+ s_{2k-1}}{3^n}\right],$$
which are given by
\[
\begin{array}{cl}
S_{n,k}&=\left[\dfrac{3^n+s_{2k-1}}{3^n}-\dfrac{2}{3^n\cdot 6}, \dfrac{3^n+ s_{2k-1}}{3^n}-\dfrac{1}{3^n\cdot 6}\right],\\
&=\left[\dfrac{6(3^n+s_{2k-1})-2}{3^n\cdot 6}, \dfrac{6(3^n+ s_{2k-1})-1}{3^n\cdot 6}\right],
\end{array}
\]
and  we denote as $-S_{n,k}$ the closed interval symmetric with respect to the imaginary axis \emph{i.e.}
\[
-S_{n,k}=\left[-\dfrac{6(3^n+s_{2k-1})-2}{3^n\cdot 6},- \dfrac{6(3^n+ s_{2k-1})-1}{3^n\cdot 6} \right].
\]

Then we write $S_{n,k}$ and $-S_{n,k}$ as the union of closed subintervals $\bigcup\limits_{m\in\mathbb{N}}\Bigl(S_{n,k}\Bigr)_m$ and $\bigcup\limits_{m\in\mathbb{N}}\Bigr(-S_{n,k}\Bigl)_m$ respectively as
\begin{equation}\label{eq:49}
\begin{array}{clccc}
  \Bigr(S_{n,k}\Bigl)_m & =\left[\dfrac{6(3^n+ s_{2k-1})-1}{3^n\cdot 6} -\dfrac{1}{3^n\cdot 6 \cdot 2^{m-1}}, \dfrac{6(3^n+ s_{2k-1})-1}{3^n\cdot 6}- \dfrac{1}{3^n\cdot 6 \cdot 2^m} \right], \\
      &=\left[ \dfrac{2^{m-1}(6(3^n+s_{2k-1})-1)-1}{3^n\cdot 6\cdot 2^{m-1}} , \dfrac{2^{m}(6(3^n+s_{2k-1})-1)-1}{3^n\cdot 6\cdot 2^{m}} \right],\\
  \Bigl(-S_{n,k}\Bigr)_m&=\left[ -\dfrac{6(3^n+ s_{2k-1})-1}{3^n\cdot 6}-\dfrac{1}{3^n\cdot 6 \cdot 2^m}, -\dfrac{6(3^n+ s_{2k-1})-1}{3^n\cdot 6} -\dfrac{1}{3^n\cdot 6 \cdot 2^{m-1}} \right],\\
      & =\left[-\dfrac{2^{m}(6(3^n+s_{2k-1})-1)-1}{3^n\cdot 6\cdot 2^{m}},  -\dfrac{6(3^n+ s_{2k-1})-1}{3^n\cdot 6} +\dfrac{1}{3^n\cdot 6 \cdot 2^{m-1}} \right].\\
\end{array}
\end{equation}
By definition  the length of $\Bigl(S_{n,k}\Bigr)_m$ and $\Bigl(-S_{n,k}\Bigr)_m$ is $\dfrac{1}{3^n\cdot 6\cdot 2^{m}}$ for each $m\in\mathbb{N}$, then  we choose the points

\begin{equation}\label{eq:50}
\begin{array}{cll}
\Bigl(\alpha_{n,k}\Bigr)_{3, m}& =\dfrac{2^{m-1}(6(3^n+s_{2k-1})-1)-1}{3^n\cdot 6\cdot 2^{m-1}}+\dfrac{7}{10}\left(\dfrac{1}{3^n\cdot 6\cdot 2^m}\right), \\
&=\dfrac{2^{m}(6(3^n+s_{2k-1})-1)-13}{3^n\cdot 6\cdot 10\cdot 2^{m}},\\
\Bigl(\alpha_{n,k}\Bigr)_{4, m}& = \dfrac{2^{m-1}(6(3^n+s_{2k-1})-1)-1}{3^n\cdot 6\cdot 2^{m-1}}+\dfrac{3}{10}\left(\dfrac{1}{3^n\cdot 6\cdot 2^m}\right),\\
&=\dfrac{2^{m}(6(3^n+s_{2k-1})-1)-17}{3^n\cdot 6\cdot 10\cdot 2^{m}},
\end{array}
\end{equation}
and the points $\Bigl(\alpha_{n,k}\Bigr)^{-1}_{3,m}$ and $\Bigl(\alpha_{n,k}\Bigl)^{-1}_{4,m}$ belonged to $\Bigl(-S_{1,1}\Bigr)_m$ as such

\[
\begin{array}{cl}
\Bigl(\alpha_{n,k}\Bigr)_{3, m}^{-1} &= -\dfrac{2^{m}(6(3^n+s_{2k-1})-1)-17}{3^n\cdot 6\cdot 10\cdot 2^{m}}, \\
\Bigl(\alpha_{n,k}\Bigr)_{4, m}^{-1} &=-\dfrac{2^{m}(6(3^n+s_{2k-1})-1)-13}{3^n\cdot 6\cdot 10\cdot 2^{m}}.
\end{array}
\]

Then we let $C\Bigl(f_{n,k} \Bigr)_{3,m}, C\Bigl(f_{n,k} \Bigr)^{-1}_{3,m}, C\Bigl(f_{n,k} \Bigr)_{4,m}, C\Bigl(f_{n,k} \Bigr)^{-1}_{4,m}$ be the $m$ half-circles having centers $\Bigl(\alpha_{n,k}\Bigr)_{3,m}$, $\Bigl(\alpha_{n,k}\Bigr)_{3, m}^{-1}$, $\Bigl(\alpha_{n,k}\Bigr)_{4,m}$ and $\Bigl(\alpha_{n,k}\Bigr)_{4, m}^{-1}$ respectively (see the equations \ref{eq:49} and \ref{eq:50}), and the same radius $\Bigl(r(n)\Bigr)_m=\dfrac{1}{10}\left(\dfrac{1}{3^n\cdot 6 \cdot 2^{m}}\right)$ (see the Figure \ref{Figure21}).
\begin{figure}[h!]
\begin{center}
\includegraphics[scale=0.6]{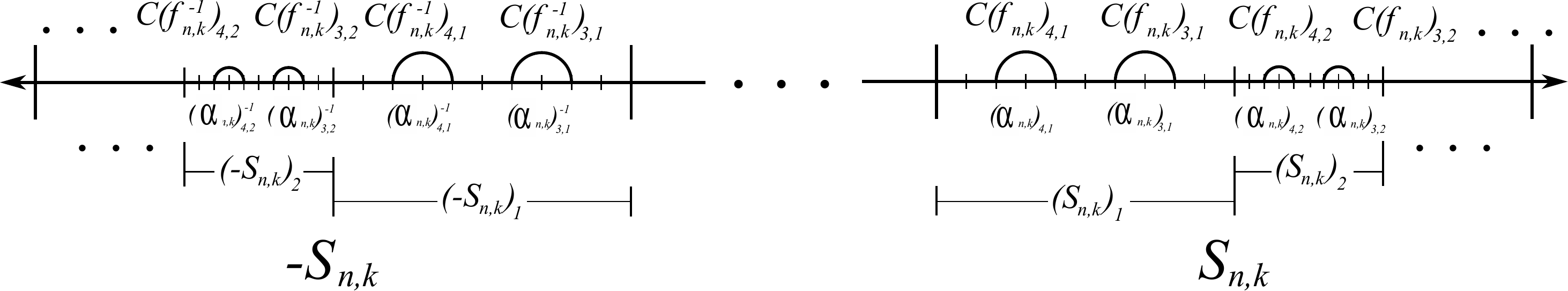}\\
\caption{\emph{The closed intervals $S_{n,k}$  and $-S_{n,k}$ written as an union of closed subintervals $\bigcup\limits_{m\in\mathbb{N}}\Bigl(S_{n,k}\Bigr)_m$ and $\bigcup\limits_{m\in\mathbb{N}}\Bigl(-S_{n,k}\Bigl)_m$.}}
 \label{Figure21}
\end{center}
\end{figure}

\begin{remark}
By construction the centers of the half-circles $C\Bigl(f_{n,k} \Bigr)_{3,m}$ and $C\Bigl(f_{n,k} \Bigr)^{-1}_{4,m}$ (corresponding, $C\Bigl(f_{n,k} \Bigr)^{-1}_{3,m}$ and $C\Bigl(f_{n,k} \Bigr)_{4,m}$) are symmetrical with respect to the imaginary axis \emph{i.e.}, we have $-\Bigl(\alpha_{n,k}\Bigr)_{3,m}=\Bigl(\alpha_{n,k}\Bigr)^{-1}_{4,m}$ (respectively, $-\Bigl(\alpha_{n,k}\Bigr)_{4,m}=\Bigl(\alpha_{n,k}\Bigr)^{-1}_{3,m}$).
\end{remark}

Now, we shall calculate the M\"obius transformations and their respective inverse
\begin{equation}\label{eq:51}
\begin{array}{ccl}
\Bigl(f_{n,k} \Bigr)_{3,m}(z) &=&\dfrac{\Bigl(a_{n,k}\Bigr)_{3,m}z+\Bigl(b_{n,k}\Bigr)_{3,m}}{\Bigl(c_{n,k}\Bigr)_{3,m}z+\Bigl(d_{n,k}\Bigr)_{3,m}},\\
\Bigl(f_{n,k}\Bigr)_{3,m}^{-1}(z)&=&\dfrac{\Bigl(d_{n,k}\Bigr)_{3,m}z-\Bigl(b_{n,k}\Bigr)_{3,m}}{-\Bigl(c_{n,k}\Bigr)_{3,m}z+\Bigl(a_{n,k}\Bigr)_{3,m}},
\end{array}
\end{equation}
having as isometric circles $C\Bigl(f_{n,k} \Bigr)_{3,m}$ and $C\Bigl(f_{n,k} \Bigr)^{-1}_{3,m}$, respectively. By remarks \ref{r:2.9} we have
\[
\begin{array}{cll}
\Bigl(a_{n,k}\Bigr)_{3,m} & =- (2^{m}(6(3^n+s_{2k-1})-1)-13) ,\\
\Bigl(c_{n,k}\Bigr)_{3,m} & = 10\cdot 3^n\cdot 6\cdot 2^m , \\
\Bigl(d_{n,k}\Bigr)_{3,m} & = -(2^{m}(6(3^n+s_{2k-1})-1)-17).\\
\end{array}
\]
Now, we substitute these values in the equation $\Bigl(a_{n,k}\Bigr)_{3,m}\cdot \Bigl(d_{n,k}\Bigr)_{3,m}-\Bigl(b_{n,k}\Bigr)_{3,m}\cdot \Bigl(c_{n,k}\Bigr)_{3,m}=1$ and computing we hold
\[
\begin{array}{cl}
\Bigl(b_{n,k}\Bigr)_{3,m} & = \dfrac{(2^{m}(6(3^n+s_{2k-1})-1)-13)(2^{m}(6(3^n+s_{2k-1})-1)-17)-1}{10\cdot 3^n\cdot 6\cdot 2^m}.\\
\end{array}
\]
Hence, we can easily write the explicit form of the M\"{o}bius transformations of equations \ref{eq:51}.

On the other hand, for every $m\in\mathbb{N}$ we shall calculate the M\"obius transformations and its respective inverse
\begin{equation}\label{eq:52}
\begin{array}{ccl}
\Bigl(f_{n,k}\Bigr)_{4,m}(z) & = & \dfrac{\Bigl(a_{n,k}\Bigr)_{4,m}z+\Bigl(b_{n,k}\Bigr)_{4,m}}{\Bigl(c_{n,k}\Bigr)_{4,m}z+\Bigl(d_{n,k}\Bigr)_{4,m}},\\
\Bigl(f_{n,k}\Bigr)_{4,m}^{-1}(z) & = &\dfrac{\Bigl(d_{n,k}\Bigr)_{4,m}z-\Bigl(b_{n,k}\Bigr)_{4,m}}{-\Bigl(c_{n,k}\Bigr)_{4,m}z+\Bigl(a_{n,k}\Bigr)_{4,m}},\\
\end{array}
\end{equation}
having as isometric circles $C\Bigl(f_{n,k} \Bigr)_{4,m}$  and $C\Bigl(f_{n,k} \Bigr)^{-1}_{4,m}$, respectively. By remarks \ref{r:2.9} we get

\[
\begin{array}{cll}
\Bigl(a_{n,k}\Bigr)_{4,m} & = -(2^{m}(6(3^n+s_{2k-1})-1)-13) ,\\
\Bigl(c_{n,k}\Bigr)_{4,m} & = 10\cdot 3^n\cdot 6\cdot 2^m  , \\
\Bigl(d_{n,k}\Bigr)_{4,m} & = -(2^{m}(6(3^n+s_{2k-1})-1)-17).\\
\end{array}
\]
Now, we substitute these values in the equation $\Bigl(a_{n,k}\Bigr)_{4,m}\cdot \Bigl(d_{n,k}\Bigr)_{4,m}-\Bigl(b_{n,k}\Bigr)_{4,m}\cdot \Bigl(c_{n,k}\Bigr)_{4,m}=1$ and computing we hold
\[
\begin{array}{cl}
\Bigl(b_{n,k}\Bigr)_{4,m} & = \dfrac{(2^{m}(6(3^n+s_{2k-1})-1)-13)(2^{m}(6(3^n+s_{2k-1})-1)-17)-1}{10\cdot 3^n \cdot 6\cdot 2^m}.\\
\end{array}
\]
Hence, we can easily write the explicit form of the M\"{o}bius transformations of the equations \ref{eq:52}. Then for each $k\in\{0,\ldots, 2^{n-1}-1\}$ we define the set
\begin{equation}\label{eq:53}
\begin{array}{cl}
\Bigl(J_{n,k}\Bigr)_{right-left} &:=\left\{\Bigl(f_{n,k}\Bigr)_{3,m}(z), \, \Bigl(f_{n,k}\Bigr)^{-1}_{3,m}(z), \, \Bigl(f_{n,k}\Bigr)_{4,m}(z), \Bigl(f_{n,k}\Bigr)_{4,m}^{-1}(z): m\in\mathbb{N} \right\},\\
\Bigl(\mathcal{C}_{n,k}\Bigr)_{right-left} & := \left\{ C\Bigl(f_{n,k}\Bigr)_{3,m}, \, C\Bigl(f_{n,k}\Bigr)^{-1}_{3,m}, \, C\Bigl(f_{n,k}\Bigr)_{4,m}, \, C\Bigl(f_{n,k}\Bigr)_{4,m}^{-1}:m\in\mathbb{N} \right\}. \\
\end{array}
\end{equation}
By construction each M\"{o}bius transformation of $\Bigl(J_{n,k}\Bigr)_{right-left}$ is hyperbolic and the half-circle of $\Bigl(\mathcal{C}_{n,k}\Bigr)_{right-left}$ are pairwise disjoint.

Finally, from  equations \ref{eq:44}, \ref{eq:48}, and \ref{eq:53} we define the sets of the M\"{o}bius maps and their respective isometric circles, as
\begin{equation}\label{eq:54}
\begin{array}{cl}
J_{n} 
     & := \left\{f_{n,k}(z), \, f_{n,k}^{-1}(z), \, \Bigl( f_{n,k}\Bigr)_{s,m}(z), \, \Bigl( f_{n,k}\Bigr)_{s,m}^{-1}(z): k\in\{0,\ldots,2^{n-1}-1\}, \, s\in \{1,\ldots, 4\}, \, m\in\mathbb{N}\right\},\\
\mathcal{C}_n & := \left\{C(f_{n,k}), \, C(f_{n,k}^{-1}), \, C\Bigl( f_{n,k}\Bigr)_{s,m}, \, C\Bigl( f_{n,k}\Bigr)_{s,m}^{-1}: k\in\{0,\ldots,2^{n-1}-1\}, \, s\in \{1,\ldots, 4\}, \, m\in\mathbb{N}\right\}.\\
\end{array}
\end{equation}
We remark that by construction each M\"{o}bius transformation of $J_{n}$ is hyperbolic and the half-circle of $\mathcal{C}_{n}$ are pairwise disjoint.

By the previous recursive construction of M\"{o}bius transformations and half-circles we define the set
\begin{equation}\label{eq:55}
\begin{array}{cl}
J &:=\bigcup\limits_{n\in\mathbb{N}} J_n,\\
\mathcal{C}&:=\bigcup\limits_{n\in\mathbb{N}}\mathcal{C}_{n},\\
\end{array}
\end{equation}
%
and we denote as $\Gamma$ the subgroup of $PSL(2,\mathbb{R})$ generated by the union $J$. We note that by construction each M\"{o}bius transformation of $J$ is hyperbolic and the half-circle of $\mathcal{C}$ are pairwise disjoint.

\vspace{2mm}
\noindent \textbf{Step 2. The group $\Gamma$ is a Fuchsian group.} In order to $\Gamma$ will be a Geometric Schottky group, we shall define a Schottky description for it. Hence, by proposition \ref{p:2.16} conclude that $\Gamma$ is Fuchsian.

\vspace{1mm}
\noindent We notice that the elements belonged to the set $J$ (see the equation \ref{eq:55}) can be indexed by a symmetric subset of $\mathbb{Z}$. Merely, we let $P:=\{p_n\}_{n\in\mathbb{N}}$ be the subset of $\mathbb{N}$ composed by all the primes numbers, then  is easy to check that the map $\psi : J \to  \mathbb{Z}$  such that
\[
f_{n,k}(z) \mapsto p_{4+n}^k, \quad f_{n,k}^{-1}(z)\mapsto -p_{4+n}^k, \quad \Bigl( f_{n,k} \Bigr)_{s,m}(z)\mapsto p_s^m \cdot p_{4+n}^k, \quad \Bigl( f_{n,k} \Bigr)^{-1}_{s,m}(z)\mapsto  -p_s^m \cdot p_{4+n}^k,
\]
for every $n, m\in \mathbb{N}$, $k\in\{0,\ldots,2^{n-1}-1\}$, $s\in\{1,\ldots,4\}$, it is well-defined and injective. We note that the image of $J$ under $\psi$ is a symmetric subset of $\mathbb{Z}$, which we denote as $I$. Given that for each element $k$ belonged to $I$ there is a unique transformation $f \in J$, such that $\psi(f)=k$, we label the map $f$ as $f_k$ and its respective isometric circle $C(f)$ as $C(f_k)$. Hence, we re-write the sets $J$ and $\mathcal{C}$ as
\[
\begin{array}{cl}
J & = \left\{f_k(z)\right\}_{k\in I},\\
\mathcal{C}     & =\{C(f_k)\}_{k\in I}.
\end{array}
\]

On the other hand, we define the set $\{A_k\}_{k\in I}$ where $A_k$ is the straight segment in the real line $\mathbb{R}$ whose ends points coincide with the endpoints at  infinite of the half-circle $C(f_{k})$ (see the equation \ref{eq:55}). In other words, it is the straight segment joining the endpoints of $C(f_k)$ to the isometric circle of $f_k(z)$.

We claim that the pair
\begin{equation}\label{eq:56}
(\{A_k\}, \{f_k\})_{k\in I}
\end{equation}
is a Schottky description.

Regarding the recursive construction of the family $J=\{f_k\}_{k\in I}$ described above, it is immediate that the pair $(\{A_k\},\{f_k\})_{f_k\in J}$ satisfies the conditions from 1 to 4 of definition \ref{d:2.14}. Thus, we must only prove that the fact 5 is done. The proof is the same as in the Cantor tree case.

\vspace{2mm}
\noindent \textbf{Step 3. Holding the surface called the Blooming Cantor tree.} The Geometric Schottky group $\Gamma$ acts freely and properly discontinuously on the open subset $\mathbb{H}-K$, where the subset $K\subset \mathbb{H}$ is defined as in equation \ref{eq:12}. In this case we note that the set $K$ is empty because of any two different elements of $\mathcal{C}$ are empty. Then the quotient space
\begin{equation}\label{eq:57}
S:= \mathbb{H}/\Gamma
\end{equation}
is  a well-defined  and through the projection map $\pi: \mathbb{H} \to S$ is a hyperbolic surface.  We shall prove that $S$ is homeomorphic to the blooming Cantor tree. In other words, $S$ has ends spaces the cantor set and each end has infinite genus. To prove this we will use the same ideas as the Cantor tree case. First we will describe the end space of $S$ using the property of $\sigma$-compact of $S$. Moreover, we will show that ends of $S$ have infinite genus. To conclude, we will define a homeomorphism $f$ from the ends spaces of the Cantor binary tree $Ends(T2^\omega)$ onto the ends space $Ends(S)$. The following remark is necessary.

\begin{remark}
We let $F(\Gamma)$ be the standard fundamental of the Geometric Schottky group $\Gamma$, as such
\begin{equation}
F(\Gamma):=\bigcap_{i\in I} \overline{\hat{C}(f_i)}\subseteq \mathbb{H},
\end{equation}
By the proposition \ref{p:2:17} it is a fundamental domain for $\Gamma$ having the following properties.
\begin{enumerate}
\item It is connected and locally finite having infinite hyperbolic area. Further, its boundary is composed by the family of half-circle $\mathcal{C}$ (see equation \ref{eq:55}). In other words, it consists of infinitely many hyperbolic geodesic with ends points at infinite and mutually disjoint.

\item It is  a non-compact Dirichlet region and the quotient space $F(\Gamma)/\Gamma$ is homeomorphic to $S$, then the quotient space $S$ is also a non-compact hyperbolic surface with infinite hyperbolic area (see \cite[Theorem 14.3 p. 283]{KS2}).
 \end{enumerate}
\end{remark}

Since surfaces are $\sigma$-compact space, for $S$ there is an exhaustion of $S=\bigcup_{n\in\mathbb{N}} K_n$ by compact sets whose complements define the  ends spaces of the surface $S$. More precisely,

\vspace{1.5mm}
\noindent \textbf{For $n=1$.} We consider the radius $r(1)=\dfrac{1}{6\cdot 3\cdot 2}$ given in the recursive construction of $\Gamma$ and define the compact subset $\tilde{K}_1$ of the hyperbolic plane $ \mathbb{H}$ as follows
$$
\tilde{K}_1:=\{ z\in \mathbb{H}: -2 \leq Re(z) \leq 2, \, \text{ and } \, r(1) \leq Im (z) \leq 1\}.
$$
The image of the intersection $\tilde{K}_1\cap F(\Gamma)$ under the projection map $\pi$
$$
\pi ({\tilde{K}_1 \cap F(\Gamma)}):=K_1\subset S,
$$
is a compact subset of $S$. By definition of $\tilde{K}_1$ the different $S\setminus K_1$ consists of two connected components whose closure in $S$ are non-compact, and they have compact boundary. Hence, we can write
\[
S\setminus K_1:=U_0 \sqcup U_1.
\]
We note that by construction each connected component of $S\setminus K_1$ has infinite genus.

\begin{remark}\label{r:3.9}
The set of connected components of $S\setminus K_1$ and the set defined as $2^1:=\{0,1\}$ are equipotent. In other words, they are in one-to-one relation.
\end{remark}

\vspace{1mm}
\noindent \textbf{For $n=2$.} We consider the radius $r(2)=\dfrac{1}{6\cdot 3^2\cdot 2}$ given in the recursive construction of $\Gamma$ and define the compact subset $\tilde{K}_2$ of the hyperbolic plane $\mathbb{H}$ as follows
$$
\tilde{K}_2:=\{ z\in \mathbb{H}: -3 \leq Re(z) \leq 3, \, \text{ and } \, r(2)\leq Im(z) \leq 2\}.
$$
By construction $\tilde{K}_1\subset \tilde{K}_2$ and the image of the intersection $\tilde{K}_2\cap F(\Gamma)$ under the projection map $\pi$
$$
\pi ({\tilde{K}_2 \cap F(\Gamma)}):=K_2\subset S,
$$
is a compact subset of $S$ such that $K_1\subset K_2$. By definition of $\tilde{K}_2$ the different $S\setminus K_2$ consists of $2^2$ connected components whose closure in $S$ are non-compact, and they have compact boundary. Moreover, for every $l\in 2^1$ there exist exactly two connected components of $S\setminus K_2$ contained in $U_l\subseteq S\setminus K_1$. Hence, we can write
\[
S\setminus K_2:=U_{0,\, 0}\sqcup U_{0,\, 1} \sqcup U_{1,\, 0}\sqcup U_{1,\, 1}=\bigsqcup\limits_{l\in2^1} (U_{l,\, 0}\sqcup U_{l,\, 1}),
\]
so that $U_{l, \, 0}, U_{l, \,1}\subset U_l$ for every $l\in2^1$. We note that by construction each connected component of $S\setminus K_2$ has infinite genus.

\begin{remark}\label{r:3.10}
The set of connected components of $S\setminus K_2$ and the set defined as $2^2:= \prod\limits_{i=1}^{2}\{0,1\}_i$ are equipotent. In other words, they are in one-to-one relation.
\end{remark}

\vspace{1mm}
Following recursive the construction above, for $n$ we consider the radius $r(n)$ given in the recursive construction of $\Gamma$ and define the compact subset $\tilde{K}_n$ of the hyperbolic plane $\mathbb{H}$ as follows
$$\tilde{K}_n:=\{ z\in \mathbb{H}: -(n+1) \leq Re(z) \leq n+1, \, \text{ and } \, r(n)\leq Im(z) \leq n\}.$$
By construction $\tilde{K}_{n-1}\subset \tilde{K}_n$ and the image of the intersection $\tilde{K}_n\cap F(\Gamma)$ under the projection map $\pi$
$$\pi ({\tilde{K}_n \cap F(\Gamma)}):=K_n\subset S,$$
is a compact subset of $S$ in which $K_{n-1}\subset K_n$. By definition of $\tilde{K}_n$ the difference $S\setminus K_n$ consists of $2^n$ connected components whose closure in $S$ are non-compact, and they have compact boundary. Moreover, for every $l\in 2^{n-1}$ there exist exactly two connected components of $S\setminus K_{n-1}$ contained in $U_l \subset S\setminus K_{n-1}$. Hence, we can write
\[
S\setminus K_n:= \bigsqcup_{l\in 2^{n-1}} (U_{l,\,0} \sqcup U_{l,\, 1})
\]
such that $U_{l,\, 0}, U_{l,\, 1}\subset U_l$ for every $l\in 2^{n-1}$. We note that by construction each connected component of $S\setminus K_n$ has infinite genus.

\begin{remark}\label{r:3.11}
The set of connected components of $S\setminus K_n$ and the set defined as $2^n:= \prod\limits_{i=1}^{n}\{0,1\}_i$ are equipotent. In other words, they are in one-to-one relation.
\end{remark}

This recursive construction induces the desired numerable family of increasing compact subset $\{K_n\}_{n\in\mathbb{N}}$ covering the surface $S$, 
$$
S=\bigcup_{n\in\mathbb{N}} K_n.
$$
Thus, the ends space of $S$ is composed by all sequences $(U_{l_n})_{n\in \N}$ such that $U_{l_n} \subset S\setminus K_n$ and $U_{l_n}\supset U_{l_{n+1}}$, for each $n\in\mathbb{N}$. Further, $l_n\in 2^n$, $l_{n+1}\in 2^{n+1}$ such that $\pi_{i}(l_n)=\pi_{i}(l_{n+1})$ for every $i\in\{1,\ldots,n\}$, (see subsection Cantor binary tree). By construction, each element $U_{l_n}$ of the sequence $(U_{l_n})_{n\in \N}$ has infinite genus, it means each ends of the surface $S$ has infinite genus.

Hence, we define the map
\[
\begin{array}{ccc}
f :  Ends(T2^\omega) \to  Ends(S), &  &          (v_i)_{i\in\mathbb{N}}  \mapsto  [U_{v_i}]_{i\in \N},
\end{array}
\]
and proceed verbatim as in the case of the Cantor tree.

\qed

\begin{corollary}
For all $n\in\mathbb{N}$, there is a classical Schottky subgroup $\Gamma_n$ of $\Gamma$ having rank $n$, such as the quotient space $S_n:=\mathbb{H}/\Gamma_n$ is a hyperbolic surface homeomorphic to the sphere punctured by $2n$ points and $Ends(S)=Ends_{\empty}(S)$. Further, the fundamental group of $S_n$ is isomorphic to $\Gamma_n$.
\end{corollary}

Indeed, consider $\Gamma_n$ as the Fuchsian group generated by the set $J_n$ (see the equation \ref{eq:54}) and proceed verbatim.

On the other hand, the following corollaries are immediate from Theorem \ref{T:uni} and the construction of the groups $\Gamma$.

\begin{corollary}
The fundamental group of the Blooming Cantor tree is isomorphic to $\Gamma$.
\end{corollary}

\begin{corollary}
The fundamental group of the Cantor tree is isomorphic to any subgroup of the fundamental group of the blooming Cantor tree.
\end{corollary}

\textbf{Acknowledgments.} The authors sincerely thank  Jes\'us Muci\~no Raymundo, Rub\'en Antonio Hidalgo Ortega, and Fernando Hern\'andez Hern\'andez for their constructive conversations and  valuable help.




\begin{bibdiv}
 \begin{biblist}

\bib{Abi}{article}{
   author={Abikoff, William},
   title={The uniformization theorem},
   journal={Amer. Math. Monthly},
   volume={88},
   date={1981},
   number={8},
   pages={574-592},
   issn={0002-9890},
}

\bib{Bear1}{book}{
   author={Beardon, Alan F.},
   title={A Premier on Riemann Surfaces},
   series={London Mathematical Society Lecture Note Series},
   volume={78},
   publisher={Cambridge University Press, Cambridge},
   date={1984},
   pages={x+188},
   isbn={0-521-27104-5},
}

\bib{Bear}{book}{
   author={Beardon, Alan F.},
   title={The geometry of discrete groups},
   series={Graduate Texts in Mathematics},
   volume={91},
   publisher={Springer-Verlag, New York},
   date={1983},
   pages={xii+337},
   isbn={0-387-90788-2},
}



\bib{BJ}{article}{
   author={Button, Jack},
   title={All Fuchsian Schottky groups are classical Schottky groups},
   conference={
      title={The Epstein birthday schrift},
   },
   book={
      series={Geom. Topol. Monogr.},
      volume={1},
      publisher={Geom. Topol. Publ., Coventry},
   },
   date={1998},
   pages={117--125 (electronic)},
}

\bib{Carne}{book}{
       author={Carne, T. K},
       title={Geometry and groups},
        publisher={Cambridge University (electronic)},
        date={2012},
}


\bib{Dugu}{book}{
   author={Dugundji, James},
   title={Topology},
   note={Reprinting of the 1966 original;
   Allyn and Bacon Series in Advanced Mathematics},
   publisher={Allyn and Bacon, Inc., Boston, Mass.-London-Sydney},
   date={1978},
   pages={xv+447},
}

\bib{Far}{book}{
   author={Farkas, H. M.},
   author={Kra, I}
   title={Riemann Surfaces},
   series={Graduate Texts in Mathematics},
   volume={71},
   edition={2},
   publisher={Springer-Verlag, New York},
   date={1992},
   pages={xvi+363},
   isbn={0-387-97703-1},
   doi={10.1007/978-1-4612-2034-3},
}
	
\bib{Ford}{article}{
   author={Ford, L. R.},
   title={The fundamental region for a Fuchsian group},
   journal={Bull. Amer. Math. Soc.},
   volume={31},
   date={1925},
   number={9-10},
   pages={531--539},
}



\bib{Fre}{article}{
   author={Freudenthal, Hans},
   title={\"Uber die Enden topologischer R\"aume und Gruppen},
   language={German},
   journal={Math. Z.},
   volume={33},
   date={1931},
   number={1},
   pages={692--713},
}



\bib{Ghys}{article}{
   author={Ghys, {\'E}tienne},
   title={Topologie des feuilles g\'en\'eriques},
   language={French},
   journal={Ann. of Math. (2)},
   volume={141},
   date={1995},
   number={2},
   pages={387--422},
}
	



\bib{Hilbert}{article}{
   author={Hilbert, David},
   title={Mathematical problems},
   note={Reprinted from Bull. Amer. Math. Soc. {\bf 8} (1902), 437-479},
   journal={Bull. Amer. Math. Soc. (N.S.)},
   volume={37},
   date={2000},
   number={4},
   pages={407-436},
}




\bib{KS}{book}{
   author={Katok, Svetlana},
   title={Fuchsian groups},
   series={Chicago Lectures in Mathematics},
   publisher={University of Chicago Press, Chicago, IL},
   date={1992},
   pages={x+175},
}



\bib{KS2}{article}{
   author={Katok, Svetlana},
   title={Fuchsian groups, geodesic flows on surfaces of constant negative
   curvature and symbolic coding of geodesics},
   conference={
      title={Homogeneous flows, moduli spaces and arithmetic},
   },
   book={
      series={Clay Math. Proc.},
      volume={10},
      publisher={Amer. Math. Soc., Providence, RI},
   },
   date={2010},
   pages={243--320},
}

\bib{Ker}{book}{
   author={Ker\'ekj\'art\'o, B\'ela.},
   title={Vorlesungen \"uber Topologie I},
   series={Mathematics: Theory \& Applications},
   publisher={Springer},
   place={Berl\'in},
   date={1923},
   }


\bib{LJ}{book}{
   author={Lee, John M.},
   title={Introduction to topological manifolds},
   series={Graduate Texts in Mathematics},
   volume={202},
   publisher={Springer-Verlag, New York},
   date={2000},
   pages={xviii+385},
}


\bib{TM}{article}{
   author={Maitani, Fumio},
   author={Taniguchi, Masahiko},
   title={A condition for a circle domain and an infinitely generated
   classical Schottky group},
   conference={
      title={Topics in finite or infinite dimensional complex analysis},
   },
   book={
      publisher={Tohoku University Press, Sendai},
   },
   date={2013},
   pages={169--175},
}


\bib{MB}{book}{
   author={Maskit, Bernard},
   title={Kleinian groups},
   series={Grundlehren der Mathematischen Wissenschaften [Fundamental
   Principles of Mathematical Sciences]},
   volume={287},
   publisher={Springer-Verlag, Berlin},
   date={1988},
   pages={xiv+326},
}



\bib{PSul}{article}{
   author={Phillips, Anthony},
   author={Sullivan, Dennis},
   title={Geometry of leaves},
   journal={Topology},
   volume={20},
   date={1981},
   number={2},
   pages={209--218},
}

\bib{Pra}{book}{
    author={Prada, D\'uwang},
    title={A golden Cantor Set},
    language={Spanish},
    series={Undergrade Dissertation},
    publisher={Industrial University of Santander, Bucaramanga, Colombia},
    date={2006},
}


\bib{Ray}{article}{
   author={Raymond, Frank},
   title={The end point compactification of manifolds},
   journal={Pacific J. Math.},
   volume={10},
   date={1960},
   pages={947--963},
   }


\bib{Ian}{article}{
   author={Richards, Ian},
   title={On the classification of noncompact surfaces},
   journal={Trans. Amer. Math. Soc.},
   volume={106},
   date={1963},
   pages={259--269},
}





\bib{SPE}{article}{
   author={Specker, Ernst},
   title={Die erste Cohomologiegruppe von \"Uberlagerungen und
   Homotopie-Eigenschaften dreidimensionaler Mannigfaltigkeiten},
   language={German},
   journal={Comment. Math. Helv.},
   volume={23},
   date={1949},
   pages={303--333},
}



\bib{Will}{book}{
   author={Willard, Stephen},
   title={General topology},
   publisher={Addison-Wesley Publishing Co., Reading, Mass.-London-Don
   Mills, Ont.},
   date={1970},
}

\bib{Ziel}{book}{
    author={Zielicz, Anna M.},
    title={Geometry and dynamics of infinitely generated Kleinian groups-Geometrics Schottky groups},
    series={PhD Dissertation},
    publisher={Universit\"{a}t Bremen},
    date={2015},
}


\end{biblist}
\end{bibdiv}
\end{document}